\documentclass[12pt,letterpaper]{article}
\usepackage{amitstyle}
\usepackage{calrsfs}

\DeclareMathOperator{\cost}{cost}
\DeclareMathOperator{\dist}{dist}        
\newcommand{\dGT}{\dist_{\mathrm{GM}}} 

\DeclareMathOperator{\Lan}{Lan}
\DeclareMathOperator{\Ran}{Ran}

\newcommand{\vect}{\mathsf{vec}}         
\newcommand{\Pos}{\mathbf{Pos}}          

\DeclareMathOperator{\Match}{Match}      
\DeclareMathOperator{\Bij}{Bij}          
\DeclareMathOperator{\Cone}{Cone}        
\DeclareMathOperator{\Summands}{Smd}
\newcommand{\smd}{\operatorname{smd}}

\newcommand{\frkU}{\mathfrak{U}}

\newcommand\Nname[1]{|[alias=#1]|}

\newcommand\al{\alpha}
\newcommand\be{\beta}
\newcommand\ep{\varepsilon}

\newcommand\Ga{\Gamma}

\DeclareMathOperator{\Ext}{Ext}

\newcommand\prj{\operatorname{prj}}

\newcommand\calA{{\mathcal A}}
\newcommand\calB{{\mathcal B}}

\newcommand\calC{{\mathcal C}}

\newcommand\calD{{\mathcal D}}

\newcommand\calI{{\mathcal I}}

\newcommand\bbN{\mathbb{N}}

\newcommand\bbR{\mathbb{R}}

\newcommand\bbZ{\mathbb{Z}}

\DeclareMathOperator{\Int}{Int}   

\newcommand\ds{\oplus}

\newcommand\Ds{\bigoplus}

\newcommand\id{1\kern-.25em{\text{{\rm l}}}} 

\newcommand\isoto{\ \raise.8ex\hbox{$^{\sim}$}\kern-.7em\hbox{$\to$}\ }

\newcommand\Cdot{\raisebox{1pt}{\scalebox{0.4}{$\bullet$}}}
\newcommand\down{_{\Cdot}}

\newcommand\blank{\operatorname{-}}

\newcommand\bfP{\mathbf{P}}

\newcommand\bfQ{\mathbf{Q}}

\newcommand\bfR{\mathbf{R}}
\newcommand\bfS{\mathbf{S}}
\newcommand\bfT{\mathbf{T}}

\newcommand\Fun{\operatorname{Fun}}

\newcommand\Cat{\mathbf{Cat}}

\DeclareMathOperator{\Hom}{Hom}

\renewcommand{\k}{\Bbbk}

\newcommand{\To}{\Rightarrow}
\newcommand\Ar[3]{\ar[from={#1}, to={#2}, #3]}
\newcommand{\si}{\sigma}
\newcommand{\de}{\delta}

\newcommand{\et}{\eta}

\newcommand{\ro}{\rho}

\newcommand{\ovl}{\overline}

\newcommand{\bsmat}[1]{\left[\begin{smallmatrix}#1\end{smallmatrix}\right]}
\newcommand{\ang}[1]{\langle #1 \rangle}
\newcommand{\Ker}{\operatorname{Ker}}
\newcommand{\Ac}{\operatorname{Ac}}
\newcommand{\kP}[1]{\k[\bfP]_{#1}}
\newcommand{\kIP}[1]{\k[\Int\oP]_{#1}}
\newcommand{\smat}[1]{\begin{smallmatrix}#1\end{smallmatrix}}
\newcommand{\ya}[1]{\xrightarrow{#1}}
\newcommand{\oP}{\ovl{\bfP}}
\newcommand{\oQ}{\ovl{\bfQ}}
\newcommand{\Into}[1]{\Int \ovl{#1}}
\newcommand{\ext}{\operatorname{ext}}

\newcommand{\supp}{\operatorname{supp}}
\newcommand{\gInt}{\operatorname{gInt}}
\newcommand{\dset}{\mathord{\downarrow}} 
\newcommand{\uset}{\mathord{\uparrow}} 
\newcommand{\vct}{\boldsymbol}
\newcommand{\mset}[1]{\{\!\{#1\}\!\}}
\newcommand{\Mat}{\operatorname{Mat}}
\newcommand{\pMatch}{\operatorname{pMatch}}
\newtheorem{ntn}[definition]{Notation}
\newcommand{\Kb}{\mathcal{K}^{\mathrm{b}}}
\newcommand{\Cb}{\mathcal{C}^{\mathrm{b}}}
\newcommand{\Knn}{\mathcal{K}^{\ge 0}}
\newcommand{\Cnn}{\mathcal{C}^{\ge 0}}
\newcommand{\Db}{\mathcal{D}^{\mathrm{b}}}

\newcommand{\Pad}{\operatorname{Pad}}
\newcommand{\Padnn}{\Pad^{\ge 0}}
\newcommand{\distCM}{\operatorname{dist}_{\mathrm{CM}}}
\newcommand{\distM}{\operatorname{dist}_{\mathrm{M}}}
\newcommand{\add}{\operatorname{add}}

\begin{document}

\author{Hideto Asashiba\thanks{Supported by JSPS Grant-in-Aid for Scientific Research (C) 18K03207
and 25K06922; 
JSPS Grant-in-Aid for Transformative Research Areas (A) (22A201); 
Osaka Central Advanced Mathematical Institute (MEXT Promotion of Distinctive Joint Research Center Program JPMXP0723833165).}
\and Amit K. Patel}

\title{Complex Matching Distance and Stability for Minimal Projective Resolutions, with Applications to Persistence}
\maketitle

\begin{abstract}

We develop a stability theory for minimal projective resolutions of
$\mathbf{P}$-modules, where $\mathbf{P}$ is a finite metric poset.
We use the G\"ulen--McCleary distance on $\mathbf{P}$-modules together
with a new complex matching distance on bounded complexes of finitely
generated projective $\mathbf{P}$-modules. The latter yields an
extended metric on homotopy classes of such complexes and restricts to
minimal projective resolutions. Our main theorem shows that this induced
distance on minimal projective resolutions is bounded above by the
G\"ulen--McCleary distance.

As an application, we pass to the interval poset and kernel construction,
interpreting persistence diagrams as the homotopy class of minimal projective resolutions of
kernel modules. This gives a corresponding stability theorem, which in
the one-parameter case recovers classical bottleneck stability and in the
multiparameter case addresses existing notions of signed diagrams.
\end{abstract}


\medskip

\noindent\textbf{Affiliations.}
Department of Mathematics, Faculty of Science, Shizuoka University,
836 Ohya, Suruga-ku, Shizuoka, 422-8529, Japan; Institute for Advanced Study, KUIAS, Kyoto University,
Yoshida Ushinomiya-cho, Sakyo-ku, Kyoto 606-8501, Japan; and
Osaka Central Advanced Mathematical Institute, 3-3-138 Sugimoto, Sumiyoshi-ku, Osaka, 558-8585, Japan. \par
Department of Mathematics, Colorado State University, Fort Collins, CO 80523, USA.

\medskip

\noindent\textbf{Emails.}
\href{mailto:asashiba.hideto@shizuoka.ac.jp}{asashiba.hideto@shizuoka.ac.jp}
\quad|\quad
\href{mailto:amit@akpatel.org}{amit@akpatel.org}

\tableofcontents

\section{Introduction}

Throughout the paper we fix a field $\k$, and all vector spaces and
linear maps are taken over $\k$. We write $\vect$ for the category of
finite-dimensional $\k$-vector spaces. For a poset $\bfP$, a
\emph{$\bfP$-module} is a functor $\bfP\to\vect$; thus $\bfP$-modules
are precisely finite-dimensional representations of the incidence
category $\k[\bfP]$ of $\bfP$.
Their category is denoted by $\vect^\bfP$.
This viewpoint encompasses both persistence modules in topological data
analysis and functor-valued representations studied in representation
theory.

In the one-parameter case, where $\bfP$ is a finite totally ordered
poset, finitely generated persistence modules enjoy a remarkably rigid
structure: every such module decomposes as a direct sum of interval
modules, and its persistence diagram may be understood either as the
multiset of indecomposable interval summands or as a M\"obius inversion
of kernel or rank functions on the interval poset
\cite{MR1949898,MR2121296,cohen2007stability,Patel:2018,
kim2021generalized,McClearyPatel2022EditDistance}. Bottleneck stability
then identifies the interleaving distance with an optimal-transport-type
distance between persistence diagrams
\cite{cohen2007stability,ChazalCohenSteinerGlisseGuibasOudot2009,
chazal2016structure}.

For multiparameter persistence, this picture breaks down. Such modules
need not decompose into intervals, and the M\"obius inversion of their
kernel or rank functions is, in general, \emph{signed}
\cite{kim2021generalized,McClearyPatel2022EditDistance,
ASASHIBA2023107397,ASASHIBA2023100007}. Viewed through M\"obius
homology~\cite{patel2023mobius_homology} and the rank-exact Betti-table
perspective, these signed coefficients naturally separate into
homological degrees, giving rise to the ``signed barcodes'' of
Botnan--Oppermann--Oudot--Scoccola \cite{BOTNAN2024109780}. As explained
in Remark~\ref{rmk:mobius-betti} and in more detail in~\cite[Appendix A]{bauer2026algebraic}, this relationship can be expressed in
terms of minimal projective resolutions.

A central difficulty has been to formulate a satisfactory \emph{stability theorem} in this signed, multiparameter context. In one parameter, the bottleneck distance is a genuine metric and, by the isometry theorem of Lesnick, coincides with the interleaving distance \cite{Lesnick2015}. In contrast, natural extensions of the bottleneck distance to signed barcodes typically fail the triangle inequality, yielding only weak lower bounds on interleaving \cite{BOTNAN2024109780}. Related issues arise in the setting of the M\"obius inversion interpretation of a multiparameter persistence diagram; see \cite{McClearyPatel2022EditDistance} and the clarification \cite{McClearyPatel2025Erratum}. Parallel work of Bubenik and Elchesen develops Wasserstein-type metrics on so-called virtual persistence diagrams \cite{BubenikElchesen2022}.

\medskip
\noindent\textbf{Main idea.}
In this paper we construct a metric stability theory that operates
directly on $\bfP$-modules and their minimal projective resolutions.

First, for a finite metric poset $(\bfP,d_\bfP)$, we combine the
G\"ulen--McCleary distance on $\vect^\bfP$ with a new complex matching
distance on bounded complexes of finitely generated projective
$\k[\bfP]$-modules. The latter yields an extended metric on the homotopy
category $\Kb(\prj \k[\bfP])$ and restricts to minimal projective
resolutions. Our main theorem (Theorem~\ref{thm:stability}) proves
\begin{equation}
\label{eq:stability-thm}
\distCM\bigl(P\down^M,P\down^N\bigr)\le \dGT(M,N)
\qquad (M,N\in\vect^\bfP),
\end{equation}
showing that Galois transport controls the induced complex matching
distance on minimal projective resolutions.

Second, we apply this framework to persistence via the kernel functor $K:\vect^\bfP\to\vect^{\Int\oP}$, treating the homotopuy class of the full minimal projective resolution $K\down^M$ as the persistence diagram for $M$. The classical diagram and the various generalizations discussed earlier are merely shadows of this enriched structure: by retaining all boundary maps, the resolution encodes these prior invariants alongside strictly more homological information. We prove a bottleneck stability theorem for these resolutions that recovers classical bottleneck stability in the one-parameter case.

\subsection{Purpose}

The purpose of this paper is twofold. First, at the level of arbitrary
$\bfP$-modules, we construct a metric stability theory for minimal
projective resolutions. Namely, we define an extended metric $\distCM$
on the set of isomorphism classes of minimal projective resolutions of
$\bfP$-modules {\em in such a way that the stability \eqref{eq:stability-thm} holds}. 
Second, we specialize this framework to
persistence by passing to the interval poset and the kernel
construction, recovering classical bottleneck stability in one
parameter and extending it to signed multiparameter diagrams.

Minimal projective resolutions and their Betti tables have recently
emerged as central invariants for multiparameter persistence modules,
both in the rank-exact setting and in the usual multigraded setting
\cite{BOTNAN2024109780,doi:10.1137/22M1489150}. On the computational
side, there is now a growing body of work on algorithms for computing
minimal presentations and resolutions, and hence Betti tables, for
low-parameter modules
\cite{LesnickWright2022,dey2023computing,CGRST2024FoCM}. Our approach
adds a metric layer to these developments by viewing minimal projective
resolutions as the recipients of a bottleneck type stability theorem.

For a general $\bfP$-module $M$, its M\"obius homology is encoded in a minimal projective resolution $P\down^M$ via $\Ext$-groups, providing a formal bridge between M\"obius inversion and minimal projective resolutions. We briefly describe this encoding in Remark~\ref{rmk:mobius-betti}, but refer to~\cite[Appendix A]{bauer2026algebraic} for a full treatment. In the present work, this bridge is relevant only in the final persistence section, where the minimal projective resolutions of kernel modules—stripped of their boundary information—are interpreted as signed persistence diagrams or barcodes, although our results do not explicitly rely on this connection.

Concretely, for a finite metric poset $(\bfP,d_\bfP)$ we adopt the
Galois-theoretic framework introduced by G\"ulen and McCleary
\cite{GulenMcCleary,Guelen2024}. Section~\ref{sec:galois-transport}
recalls their construction of the \emph{G\"ulen--McCleary distance}
\[
\dGT:\vect^\bfP\times\vect^\bfP\longrightarrow[0,\infty].
\]
It is defined by infimizing the cost of \emph{Galois couplings}:
factorizations of two $\bfP$-modules $M$ and $N$ through a common apex
poset $\bfQ$ via Galois insertions
$f:\bfQ\rightleftarrows\bfP:g$ and $h:\bfQ\rightleftarrows\bfP:i$,
together with a module $\Gamma\in\vect^\bfQ$ whose pullbacks along the
right adjoints recover $M$ and $N$. The cost is the $L^\infty$-type
quantity
\[
\cost(\Gamma):=\sup_{q\in\bfQ} d_\bfP\bigl(f(q),h(q)\bigr),
\]
and $\dGT(M,N)$ is the infimum of this quantity over all couplings.
By results of G\"ulen and McCleary, this defines an extended metric on
isomorphism classes of $\bfP$-modules, and in the classical persistence
setting recovers the interleaving distance
\cite[Proposition~6.4]{GulenMcCleary}.

We next turn to the homological side. In
Section~\ref{sec:bneck} we define the \emph{complex matching distance}
on bounded complexes of finitely generated projective $\k[\bfP]$-modules
by matching indecomposable projective summands degreewise, subject to a
compatibility condition on differentials, while allowing padding by
contractible cones, which play the role of matching to the diagonal.
The distance between summands is measured using $d_\bfP$ on the
underlying points. Theorem~\ref{thm:metric_homotopy} shows that this
construction yields an extended metric on the homotopy category
$\Kb(\prj \k[\bfP])$, and Theorem~\ref{thm:extended_metric_projectives}
shows that its restriction to minimal projective resolutions is an
extended metric
\[
\distCM:
(\{P\down^M \in \Kb(\prj\k[\bfP]) \mid M \in \vect^\bfP\}/\!\cong)^2 
\longrightarrow[0,\infty].
\]

Our first stability theorem (Theorem~\ref{thm:stability}) shows that the
G\"ulen--McCleary distance controls this complex matching distance:
\[
\distCM\bigl(P\down^M,P\down^N\bigr)\le \dGT(M,N)\qquad(M,N\in\vect^\bfP).
\]
Conceptually, the proof comes from a single functorial construction:
given a Galois coupling for $(M,N)$, we choose a projective resolution
of the apex module $\Gamma$, pull it back along the two right adjoints
to obtain resolutions of $M$ and $N$, and then match their
indecomposable summands degreewise, with compatibility of differentials
built in.

We then connect this framework to persistence. For a finite metric poset $\bfP$, we consider the interval poset
\[
\Int\oP:=\{[x,y]\mid x,y\in\oP,\ x\le y\}
\]
of the augmented poset $\oP:=\bfP\sqcup\{\top\}$ under the product order, and define a kernel functor
\[
K:\vect^\bfP\longrightarrow\vect^{\Int\oP}
\]
by sending an interval $[x,y]$ to the kernel of the structure map $M(x\to y)$, with $M(\top):=0$ by convention. Proposition~\ref{prop:K-Lipschitz} shows this associated kernel module $K(M)$ makes $K$ $1$-Lipschitz with respect to the G\"ulen--McCleary distance. We then propose the persistence diagram of $M$ to be the homotopy class of the minimal projective resolution $K\down^M$ (Definition~\ref{def:persistence_diagram}). 
Combining this with the stability inequality above yields our main bottleneck stability theorem for persistence (Theorem~\ref{thm:persistence-stability}):
\[
\distCM\bigl(K\down^M,K\down^N\bigr)\le \dGT^\bfP(M,N).
\]
We restrict our focus to $\Int\oP\cong\{[x,y]\subseteq\oP\mid x\le y\}$ rather than more general families of subsets (cf.~\cite{kim2021generalized} or Definition~\ref{dfn:gen-interval}). Even this setting demands delicate homological control, and extending the construction further would require additional combinatorial machinery (see Section~\ref{ssec:beyond-ker}).

\subsection{Outline}
Section~\ref{sec:prelim} collects the categorical and homological
preliminaries for $\bfP$-modules, including projective resolutions and
cones. Section~\ref{sec:galois-transport} reviews the notion of Galois
couplings and the resulting G\"ulen--McCleary distance $\dGT$ on
$\vect^\bfP$, introduced by G\"ulen and McCleary
\cite{GulenMcCleary,Guelen2024}.

Section~\ref{sec:bneck} defines the complex matching distance $\distCM$
on bounded complexes of finitely generated projectives, shows that it
yields an extended metric on the homotopy category $\Kb(\prj
\k[\bfP])$, and then restricts it to minimal projective resolutions. In
the subsequent section we combine this construction with the
G\"ulen--McCleary distance to obtain our main stability theorem
$\distCM\le\dGT$ (Theorem~\ref{thm:stability}).

Section~\ref{sec:persistence} treats persistence as an application. We
introduce the augmented interval-poset construction and the kernel
functor $K:\vect^\bfP\to\vect^{\Int\oP}$, show that these are compatible
with Galois couplings, and deduce that $K$ is $1$-Lipschitz with
respect to the G\"ulen--McCleary distance. 
Interpreting the homotopy class of the minimal
projective resolution of $K(M)$ as its persistence diagrams then
yields the bottleneck stability theorem for persistence
(Theorem~\ref{thm:persistence-stability}).

\subsection*{Acknowledgment}
The authors would like to thank Professor Yasuaki Hiraoka for making this joint work
possible by providing financial support that enabled HA to visit AP for
this project. During that visit, they obtained the key example (see
Example~\ref{exm:stability-thm} for a simplified version).

They are also grateful to Dr.\ Luis Scoccola (LS) for pointing out, through an
example (see Example~\ref{exm:Luis}), that the first version of the complex
matching distance, denoted by $\distCM'$ in the present paper, was too
coarse. This observation led them to refine the notion of matchings so that
it now incorporates substantially more information about the differentials
of complexes, in accordance with HA's original idea, which has thus been
fully realized.

The authors thank LS also for pointing out a similarity between the distance
$\hat{d}^\infty_{\calI}$ defined by Bjerkevik--Lesnick in
\cite[Section~3.1]{bjerkevik2021ell} and the complex matching distance
introduced in this paper. This connection became clear to the authors
after the characterization of the existence of a matching given in
Remark~\ref{rmk:charact-matching}.

\subsection*{Acknowledgment of AI assistance}

This paper was written with the assistance of ChatGPT~5~Pro. The authors have reviewed all results and take full
responsibility for the mathematical content and any remaining errors.

\section{Preliminaries}
\label{sec:prelim}

We start with some categorical facts that will be used throughout.

\begin{lemma}
\label{lem:pres-proj}
Let $L:\calC\to\calD$ be a functor between abelian categories.
If $L$ has an exact right adjoint $R:\calD\to\calC$,
then $L$ sends projectives to projectives.
\end{lemma}

\begin{proof}
Let $A$ be projective in $\calC$. Then the functor $\calC(A,\blank)$ is exact, hence so is
\[
\calC(A,\blank)\circ R=\calC(A,R(\blank))\cong \calD(L(A),\blank).
\]
Thus $L(A)$ is projective.
\end{proof}

Let $\frkU$ be a universe. An element of $\frkU$ is called a $\frkU$-\emph{small} set. A category $\calC$ is $\frkU$-\emph{small} if its object set is $\frkU$-small and, for any objects $X,Y$, each hom-set $\calC(X,Y)$ is $\frkU$-small. We denote by $\Cat^\frkU$ the $2$-category whose objects are the $\frkU$-small categories, whose $1$-morphisms are functors, and whose $2$-morphisms are natural transformations.

\begin{proposition}
\label{prp:Yoneda-2-functor}
If $\calC$ is an object of $\Cat^\frkU$, then
\[
\Cat^\frkU(\mbox{-},\calC):(\Cat^\frkU)^{\mathrm{op}}\longrightarrow\Cat^\frkU
\]
is a $2$-functor.
\end{proposition}

\begin{proof}
See Proposition~4.5.4 in \cite{JY}.
\end{proof}

\begin{remark}
\label{rmk:contravariant-adj}
In particular, the $2$-functor $\Cat^\frkU(\mbox{-},\calC)$ sends an adjoint system $(f,g,\eta,\varepsilon)$ to the adjoint system $(g^{*},f^{*},\eta^{*},\varepsilon^{*})$ in a contravariant manner. Concretely, if $f\dashv g$ between categories $\calA$ and $\calB$, then precomposition yields $g^{*}\dashv f^{*}$ between the functor categories $\Cat^\frkU(\calB,\calC)$ and $\Cat^\frkU(\calA,\calC)$. We use this repeatedly with $\calC=\vect$ below
by considering a universe $\frkU$ such that $\vect$ is a $\frkU$-small category.
\end{remark}

We also use the following notation throughout.
The set of nonnegative integers is denoted by $\bbN$; thus $0 \in \bbN$
in this paper.
If $A$ is an abelian (additive) group and $B$ is a set, then $A^B$
denotes the direct product of $B$-copies of $A$, namely the abelian
group of maps $B \to A$, and we set
\[
A^{(B)}:= \{(a_i)_{i \in B} \in A^B \mid \{i \in B \mid a_i \ne 0\}
\text{ is finite}\},
\]
the direct sum of $B$-copies of $A$.

For a finite sequence $\vct{x} = (x_1,\dots, x_n)$ of elements of a set $S$,
we denote by $\mset{x_1, \dots, x_n}$ the \emph{multiset} determined by
$\vct{x}$, obtained by forgetting the order of the elements.
Let $\{a_1, \dots, a_m\}$ be the set of distinct elements occurring among
$x_1, \dots, x_n$, so that $a_i \ne a_j$ whenever $i \ne j$.
For each $i = 1, \dots, m$, let $l_i$ be the multiplicity of $a_i$ in
$\vct{x}$.
Then we may identify $\mset{x_1, \dots, x_n}$ with the set
\[
\{(a_i,1), \dots, (a_i, l_i) \mid i = 1,\dots, m\}.
\]
Via this identification, we apply notions concerning sets to multisets as
well.

\subsection{Poset Modules}
\label{ssec:poset-mod}

Throughout, $\Bbbk$ is a field, and $\bfP=(\bfP,\le)$ is a poset. We
regard $\bfP$ as a small (thin) category: there is a unique morphism
\begin{equation}
\label{eq:mor-in-poset}
\pi_{y \ge x} \colon x \to y
\end{equation}
exactly when $x\le y$.
We assume $\bfP$ is finite.
Let $\vect$ denote the category of finite-dimensional $\Bbbk$-vector
spaces, and write $\vect^\bfP=\Fun(\bfP,\vect)$ for the category of
$\bfP$-modules. If $\bfP$ is finite, then $\vect^\bfP$ is a
$\Bbbk$-linear, abelian, Krull--Schmidt category, and for each
$M,N\in\vect^\bfP$ the $\Bbbk$-space $\vect^\bfP(M,N)$ is finite
dimensional.

Recall that the incidence category $\k[\bfP]$ of $\bfP$ is the
$\k$-linearization of the category $\bfP$.
Namely, it is a $\k$-linear category with the same objects as $\bfP$,
and $\k[\bfP](x,y)$ is the vector space with $\bfP(x,y)$ as a basis for
all $x, y \in \bfP$, and the composition of $\k[\bfP]$ is defined as
the bilinearization of that of $\bfP$.
Any functor $\bfP \to \vect$ is uniquely extended to a $\k$-linear
functor $\k[\bfP] \to \vect$, and this correspondence is uniquely
extended to an isomorphism of categories from $\vect^\bfP$ to the
category $\Fun_\k(\k[\bfP], \vect)$ of $\k$-linear functors from
$\k[\bfP]$ to $\vect$, by which we identify these categories.

Note here that the representable functor $\bfP(x, \blank)$ with
$x \in \bfP$ is a functor from $\bfP$ to the category of small sets, not
to $\vect$. However, if we take the incidence category $\k[\bfP]$ of
$\bfP$ instead of $\bfP$, then the representable functor
\[
\k[\bfP]_x:= \k[\bfP](x, \blank)
\]
becomes a functor $\bfP \to \vect$.
It sends $y \in \bfP$ to $\k[\bfP](x,y)=\k\cdot \pi_{y \ge x}$ if
$x \le y$ and to $0$ otherwise. For a relation $y \le z$, the induced map
\[
\k[\bfP](x,y)\longrightarrow \k[\bfP](x,z)
\]
sends $\pi_{y \ge x}$ to $\pi_{z \ge x}$ when $x \le y$, and is $0$
otherwise.
By the Yoneda lemma there is an isomorphism
\[
\vect^\bfP(\k[\bfP]_x,M)\cong M(x)
\]
natural in $x \in \bfP$ and in $M \in \vect^\bfP$, showing that each
$\k[\bfP]_x$ is projective.
Moreover, since
\[
\vect^\bfP(\k[\bfP]_x,\k[\bfP]_x)\cong \k[\bfP](x,x) \cong \k
\]
is a local algebra, the projective module $\k[\bfP]_x$ is indecomposable.
Again by the Yoneda lemma, for any $\bfP$-module $L$ and any
$l \in L(x)$, there exists a unique morphism
\[
\al \colon \k[\bfP]_x \to L
\]
in $\vect^\bfP$ such that $\al(\id_x) = l$.
We denote this $\al$ by $\ro(l)$, which is defined by the ``right
multiplication'' by $l$: $g \mapsto g(l)$ for all
$g \in \k[\bfP](x,y)$ and $y \in \bfP$.

If $I$ is an interval of $\bfP$, then by definition, $V_I(x) = \k$ for
all $x \in I$. We write this $\k$ as $\k = \k 1_x$ by setting $1_x$ to
be the identity element of $\k$.
We set
\[
\ro(y \ge x):= \ro(\pi_{y \ge x}) \colon \k[\bfP]_y \to \k[\bfP]_x
\qquad\text{and}\qquad
\ro(x):= \ro(1_x) \colon \k[\bfP]_x \to V_I
\]
for short.
By convention, we set $\ro(y \ge x):= 0$ unless $y \ge x$ for all
$x, y \in \bfP$. Then we have
\begin{equation}
\label{eq:mor-betw-ind-proj}
\vect^\bfP(\k[\bfP]_y, \k[\bfP]_x) = \k\, \ro(y \ge x)
\end{equation}
for all $x, y \in \bfP$.
These notations are used in computations of projective resolutions such
as in Example~\ref{ex:bneck-chain}.

We write $\prj \k[\bfP]$ for the full subcategory of projective objects
in $\vect^\bfP$.

\begin{lemma}
\label{lem:ind-proj}
The set $\{\k[\bfP]_x \mid x\in\bfP\}$ is a complete set of
representatives of the isomorphism classes of indecomposable projective
$\bfP$-modules. Hence each $M\in\prj \k[\bfP]$ decomposes as
\[
M \cong \bigoplus_{i = 1}^n \k[\bfP]_{x_i}
\]
for a unique $|M|:= n \in \bbN$ and a unique multiset
\[
\Summands(M) = \mset{\k[\bfP]_{x_1}, \dots, \k[\bfP]_{x_n}},
\]
which are called the \emph{size} and the \emph{summand set} of $M$,
respectively.
(By convention, $n =0$ and $\Summands(M) = \varnothing$ if $M = 0$.)
\end{lemma}

\begin{ntn}
\label{ntn:smd}
Let $M, N \in \vect^\bfP$.
For simplicity, we sometimes use $\smd(M):= \mset{x_1, \dots, x_n}$
instead of $\Summands(M)$, and identify $\smd(M)$ with $\Summands(M)$
by the bijection $x \mapsto \k[\bfP]_x$.
The set of all bijections $\smd(M) \to \smd(N)$ is denoted by
$\Bij(\smd(M), \smd(N))$ and is identified with the set
$\Bij(\Summands(M), \Summands(N))$ of bijections
$\Summands(M) \to \Summands(N)$ by regarding each
$B \in \Bij(\smd(M), \smd(N))$ as the bijection
\[
\k[\bfP]_x \mapsto B(\k[\bfP]_x):= \k[\bfP]_{B(x)}
\qquad\text{for all } x \in \smd(M).
\]

When we deal with morphisms by matrices, we need to fix an order of
elements of $\smd(M)$; thus in that case, we regard $\smd(M)$ as a
sequence $(x_1, \dots, x_n)$.
\end{ntn}

The following is immediate from \eqref{eq:mor-betw-ind-proj}.

\begin{lemma}
\label{lem:mor-betw-proj}
Let $\al \colon M \to N$ be a morphism in $\prj \k[\bfP]$, and assume
that
\[
M = \Ds_{x \in \smd(M)}\k[\bfP]_x
\qquad\text{and}\qquad
N = \Ds_{y \in \smd(N)}\k[\bfP]_y.
\]
Then there exists a unique matrix
\[
\Mat(\al):= [a_{x,y}]_{(y,x) \in \smd(N) \times \smd(M)}
\]
over $\k$, called the
\emph{coefficient matrix} of $\al$,
such that
\[
\al = [a_{x,y}\, \ro(x \ge y)]_{(y,x) \in \smd(N) \times \smd(M)},
\]
and that $a_{x,y} =0$ unless $x \ge y$.
\end{lemma}

\begin{remark}
In the above, the coefficient matrix of $\al$ is uniquely determined
because of the \emph{normalization} condition that $a_{x,y} =0$ unless
$x \ge y$. Thanks to this condition, it holds that if $a_{x,y} \ne 0$,
then $x \ge y$.

Notice that $\Mat(\al)$ is obtained from the normalized matrix
expression of $\al$ by deleting $\ro(x \ge y)$ from each $(y,x)$-entry,
and that it is the transpose of the matrix
$[a_{x,y}]_{(x,y) \in \smd(M) \times \smd(N)}$:
\[
\Mat(\al) = {}^t\left([a_{x,y}]_{(x,y) \in \smd(M) \times \smd(N)}\right).
\]
\end{remark}

\subsection{Monotone Maps and Kan Extensions}
\label{ssec:mont-funct}

We now record how monotone maps between posets induce adjoint triples
between the corresponding module categories via Kan extensions. This
will be used repeatedly once we restrict to Galois connections.

A map $f:\bfQ\to\bfP$ of posets is \emph{monotone} if $x\le y$ in $\bfQ$ implies $f(x)\le f(y)$ in $\bfP$. Viewing posets as thin categories, a functor $\bfQ\to\bfP$ is precisely a monotone map. Thus a monotone map $f:\bfQ\to\bfP$ induces the restriction (precomposition) functor
\[
f^{*}:\vect^\bfP\to\vect^\bfQ,\qquad f^{*}N = N\circ f .
\]

\begin{proposition}
\label{prop:Kan-posets}
For any monotone $f:\bfQ\to\bfP$, the restriction functor $f^{*}$ is exact.
Moreover, the left and right Kan extensions along $f$ exist; we denote them by
$f_{!}:=\Lan_f$ and $f_{*}:=\Ran_f$. By the defining property of Kan extensions
there are natural adjunctions $f_{!}\dashv f^{*}\dashv f_{*}$, so $f_{!}$ is
right exact and $f_{*}$ is left exact.
\end{proposition}

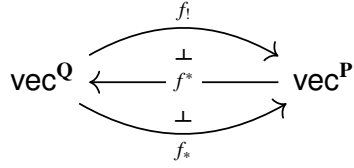
\[
\begin{tikzcd}[row sep=3em, column sep=6em]
\Nname{E}\vect^\bfQ & \Nname{C}\vect^\bfP
\arrow[from={E}, to={C}, "f_{!}", ""'{name=Fu}, bend left]
\arrow[from={C}, to={E},"f^{*}"{description, pos=.5, name=Fc}]
\arrow[from={E}, to={C}, "f_{*}"', ""'{name=Fd}, bend right]
\arrow[from={Fu}, to={Fc}, "\rotatebox{-90}{$\dashv$}" description, phantom]
\arrow[from={Fc}, to={Fd}, "\rotatebox{-90}{$\dashv$}" description, phantom]
\end{tikzcd}
\]

\subsection{Galois Connections}

We recall that Galois connections are adjunctions between posets and underlie the identifications among $f_{!}, f^{*}, f_{*}$ used later.

\begin{definition}
\label{dfn:Gal-conn}
A \emph{Galois connection} between posets $\bfQ,\bfP$ consists of monotone maps $f:\bfQ\to\bfP$ and $g:\bfP\to\bfQ$ such that
\[
f(u)\le x \iff u\le g(x)\qquad(u\in \bfQ,\ x\in \bfP).
\]
Equivalently, viewing $\bfQ$ and $\bfP$ as thin categories, this means $f\dashv g$ as functors. We write $f:\bfQ\rightleftarrows \bfP:g$, with $f$ left adjoint and $g$ right adjoint.
\end{definition}

The following characterization of a Galois connection is corresponding to
that of an adjunction using a unit and a counit.
Since the verification is straightforward, we omit the proof.

\begin{lemma}
\label{lem:Gal-conn}
In the setting of Definition \ref{dfn:Gal-conn},
$f:\bfQ\rightleftarrows \bfP:g$ is a Galois connection if and only if
\[
u \le gf(u) \quad (u \in \bfQ),\quad \text{and}\quad fg(x) \le x \quad (x \in \bfP).
\]
If this is the case, we have $fgf = f$ and $gfg = g$.
\end{lemma}

The existence of a natural transformation between monotone maps as functors is
verified by the following.
As it is easy to show, we omit the proof.

\begin{lemma}
Let $f, g \colon \bfP \to \bfQ$ be monotone maps of posets.
Then there exists a natural transformation $f \To g$ if and only if
$f(x) \le g(x)$ for all $x \in \bfP$.
The latter condition is expressed by writing $f \le g$.
\end{lemma}

To deal with pullbacks in the category of posets, we use the following.

\begin{lemma}
\label{lem:leq-pb}
Consider a diagram
$$
\begin{tikzcd}
\bfS & \bfR & \bfQ_2\\
& \bfQ_1 & \bfP
\Ar{1-1}{1-2}{"f", yshift=3pt}
\Ar{1-1}{1-2}{"g"', yshift=-3pt}
\Ar{1-2}{1-3}{"\pi_2"}
\Ar{1-2}{2-2}{"\pi_1"'}
\Ar{2-2}{2-3}{"\si_1"'}
\Ar{1-3}{2-3}{"\si_2"}
\end{tikzcd}
$$
with a pullback square in the category of posets.
Then
\begin{enumerate}[label=(\arabic*)]
\item 
If $\pi_i f = \pi_i g$ for all $i=1,2$, then $f = g$.
\item 
If $\pi_i f \le \pi_i g$ for all $i=1,2$, then $f \le g$.
\end{enumerate}
\end{lemma}

\begin{proof}
(1) This is immediate from the universality of the pullback.

(2) Recall that $\bfR$ is constructed as the full subposet of $\bfQ_1 \times \bfQ_2$
with the underlying set given by
\[
\bfR \ :=\ \bigl\{(q_1,q_2)\in\bfQ_1\times\bfQ_2 : \si_1(q_1)=\si_2(q_2)\bigr\},
\]
with projections $\pi_1:\bfR\to\bfQ_1$ and $\pi_2:\bfR\to\bfQ_2$,
where the partial order on $\bfQ_1 \times \bfQ_2$ is defined as follows:
For any $(q_1, q_2), (q'_1, q'_2) \in \bfQ_1 \times \bfQ_2$,
$(q_1, q_2) \le (q'_1, q'_2)$ if and only if $q_i \le q'_i$ for all $i=1,2$.
Then the assertion is clear from this definition.
\end{proof}

\begin{corollary}
If $f:\bfQ\rightleftarrows \bfP:g$ is a Galois connection, then $g^{*}\dashv f^{*}$.
\[
\begin{tikzcd}[row sep=3em]
\Nname{E}\vect^\bfQ & \Nname{C} \vect^\bfP 
\arrow[from={E}, to={C}, "g^{*}", ""'{name=Gs}, bend left]
\arrow[from={C}, to={E}, "f^{*}"{name=Fs}, bend left]
\arrow[from={Gs}, to={Fs}, "\rotatebox{-90}{$\dashv$}" description, phantom]
\end{tikzcd}.
\]
\end{corollary}

\begin{proof}
Immediate from the contravariant $2$-functoriality of $\Cat^\frkU(\mbox{-},\vect)$ (e.g.\ Proposition~4.5.4 in \cite{JY}).
\end{proof}

\begin{corollary}
\label{cor:3_adj}
In the same setting, there are natural isomorphisms $g^{*}\cong f_{!}$ and $f^{*}\cong g_{*}$. In particular, the adjoint pairs
\[
g^{*}\dashv f^{*},\qquad f_{!}\dashv f^{*},\qquad g^{*}\dashv g_{*}
\]
are compatible via these isomorphisms.
\[
\begin{tikzcd}[row sep=50pt, column sep = 70pt]
\Nname{E}\vect^\bfQ & \Nname{C} \vect^\bfP 
\arrow[from={E}, to={C}, "g^{*} \cong f_{!}"{description, pos=.5, name=gs}, bend left=15pt]
\arrow[from={C}, to={E}, "f^{*}\cong g_{*}"{description, pos=.5, name=fs}, bend left=15pt]
\arrow[from={C}, to={E}, "g_{!}"', ""{name=gL}, bend right=80pt]
\arrow[from={E}, to={C}, "f_{*}"', ""{name=fR}, bend right=80pt]
\arrow[from={gL}, to={gs}, "\rotatebox{-90}{$\dashv$}" description, phantom]
\arrow[from={gs}, to={fs}, "\rotatebox{-90}{$\dashv$}" description, phantom]
\arrow[from={fs}, to={fR}, "\rotatebox{-90}{$\dashv$}" description, phantom]
\end{tikzcd}.
\]
\end{corollary}

\begin{proof}
By uniqueness of adjoints in a $2$-category: $g^{*}$ and $f_{!}$ are both left adjoint to $f^{*}$, hence canonically isomorphic; dually, $f^{*}$ and $g_{*}$ are both right adjoint to $g^{*}$, hence canonically isomorphic.
\end{proof}

\begin{corollary}
\label{cor:g-ast-ex}
If $f:\bfQ\rightleftarrows \bfP:g$ is a Galois connection, then $g_{*}:\vect^\bfP\to\vect^\bfQ$ is exact.
\end{corollary}

\begin{proof}
Using Corollary~\ref{cor:3_adj}, $g_{*}\cong f^{*}$, and $f^{*}$ is exact by Proposition~\ref{prop:Kan-posets}.
\end{proof}

Combining Corollary~\ref{cor:g-ast-ex} with Lemma~\ref{lem:pres-proj} yields:

\begin{proposition}\label{prop:pullback_preserves_projectives}
Let $f:\bfQ\rightleftarrows \bfP:g$ be a Galois connection. If $M\in\vect^\bfQ$ is projective, then $g^{*}(M)$ is projective in $\vect^\bfP$.\qed
\end{proposition}

\medskip

The next basic facts will be used tacitly; they are immediate from the adjunction $f\dashv g$ and we omit the proof.

\begin{lemma}\label{lem:galois-connection-basics}
Let $f:\bfQ\rightleftarrows \bfP:g$ be a Galois connection. Then:
\begin{enumerate}[label=(\arabic*)]
\item The following are equivalent: $f$ is surjective; $g$ is injective; $f\circ g=\id_{\bfP}$.
If this is the case, this Galois connection is called a {\em Galois insertion}.
(For the proof, use Lemma \ref{lem:Gal-conn}.)
\item For $x\in \bfP$,
\[
g(x)=\max\{\,u\in \bfQ\mid f(u)\le x\,\}.
\]
In particular, if $f\circ g=\id_{\bfP}$, then
\[
g(x)=\max\{\,u\in \bfQ\mid f(u)=x\,\}.
\]
\end{enumerate}
\end{lemma}

\section{G\"ulen--McCleary Distance}
\label{sec:galois-transport}

Fix a finite poset $\bfP$ equipped with a metric $d_{\bfP}$. Motivated by optimal transport, we compare $\bfP$-modules by \emph{transporting} them through a common “apex” poset $\bfQ$ via \emph{Galois insertions}.

\begin{definition}{\cite[Definition~6.1]{GulenMcCleary}}
\label{def:coupling-insertion}
Let $M,N\in\vect^{\bfP}$. A \emph{Galois coupling} $(\bfQ, f \dashv g, h \dashv i, \Gamma)$
of the pair $(M, N)$ consists of a finite poset $\bfQ$, two Galois insertions
\[
f:\bfQ\rightleftarrows \bfP:g,
\qquad
h:\bfQ\rightleftarrows \bfP:i
\quad\text{with}\quad
f\circ g=\id_{\bfP}=h\circ i,
\]
and a module $\Gamma\in\vect^{\bfQ}$ such that $g^{*}\Gamma \cong M$ and $i^{*}\Gamma \cong N$. Equivalently (Corollary~\ref{cor:3_adj}), $M\cong f_{!}\Gamma$ and $N\cong h_{!}\Gamma$.
\end{definition}

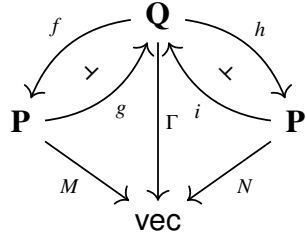
\[
\begin{tikzcd}
& \bfQ \ar[dd,"\Gamma"] \ar[dr, bend left, "h"{name=H}] \ar[dl, bend right, "f"'{name=F}] & \\
\bfP \ar[rd,"M"'] \ar[ru, bend right, "g"'{name=G}] && \bfP \ar[ld,"N"] \ar[lu, bend left, "i"{name=I}] \\
& \vect &
\arrow[phantom, from=F, to=G, "\rotatebox{-45}{$\dashv$}" description]
\arrow[phantom, from=H, to=I, "\rotatebox{-135}{$\dashv$}" description]
\end{tikzcd}
\]

\begin{definition}{\cite[Definition~6.1]{GulenMcCleary}}
\label{def:cost}
The \emph{cost} of a coupling $(\bfQ,f\dashv g,\ h\dashv i,\ \Gamma)$ is
\[
\cost(\Gamma):=\sup_{q\in\bfQ} d_{\bfP}\big(f(q),\,h(q)\big).
\]
\end{definition}

The notion of a Galois coupling and its associated cost was introduced by G\"ulen and McCleary~\cite{GulenMcCleary, Guelen2024}, who allow the two modules to be indexed by different posets. In this paper, we restrict attention to the case of a single finite metric poset $\bfP$. In the next two subsections, we use Galois couplings to
introduce a metric defined by G\"ulen and McCleary
that we call the G\"ulen--McCleary distance after them to distinguish it from classical interleaving distance. 
G\"ulen and McCleary prove in~\cite[Proposition 6.10]{GulenMcCleary} and~\cite[Proposition 3.4.10]{Guelen2024} that this distance satisfies the triangle inequality.

\subsection{Composition of Couplings}
\label{ssec:composition}

We now show that Galois couplings compose.  
Suppose we are given couplings
\[
(\bfQ_1,\ f_1\dashv g_1,\ h_1\dashv i_1,\ \Gamma_1) \text{ of $(M,N)$ }
\qquad\text{and}\qquad
(\bfQ_2,\ f_2\dashv g_2,\ h_2\dashv i_2,\ \Gamma_2) \text{ of $(N,O)$},
\]
displayed below:
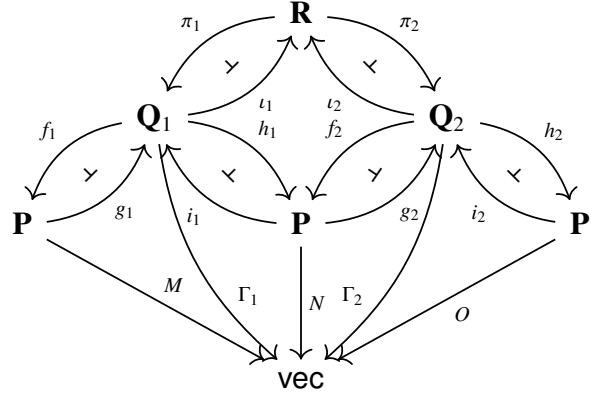
\[
\begin{tikzcd}
    && \bfR \ar[dr, bend left, "\pi_2"{name=piTwo}] 
           \ar[dl, bend right, "\pi_1"'{name=piOne}]  && \\
    & \bfQ_1 
        \ar[dddr, "\Gamma_1", near end, bend right = 20]
        \ar[ur, bend right, "\iota_1"'{name=iotaOne}]  
        \ar[dl, bend right, "f_1"'{name=fOne}]
        \ar[dr, bend left, "h_1"{name=hOne}]
        & & 
      \bfQ_2 
        \ar[ddld, "\Gamma_2"', bend left = 20, near end]
        \ar[ul, bend left, "\iota_2"{name=iotaTwo}]
        \ar[dl, bend right, "f_2"'{name=fTwo}]
        \ar[dr, bend left, "h_2"{name=hTwo}]
        & \\
    \bfP \ar[rrdd, "M"] 
        \ar[ur, bend right, "g_1"'{name=gOne}] 
      && 
    \bfP \ar[dd, "N"]  
        \ar[ul, bend left, "i_1"{name=iOne}] 
        \ar[ur, bend right, "g_2"'{name=gTwo}] 
      && 
    \bfP \ar[ldld, "O"] 
        \ar[ul, bend left, "i_2"{name=iTwo}] \\
    &  & &  & \\
    && \vect &&
    \arrow[phantom, from=piOne, to=iotaOne, "\rotatebox{-45}{$\dashv$}" description]
    \arrow[phantom, from=piTwo, to=iotaTwo, "\rotatebox{-135}{$\dashv$}" description]
    \arrow[phantom, from=fOne,  to=gOne,   "\rotatebox{-45}{$\dashv$}" description]
    \arrow[phantom, from=hOne,  to=iOne,   "\rotatebox{-135}{$\dashv$}" description]
    \arrow[phantom, from=fTwo,  to=gTwo,   "\rotatebox{-34}{$\dashv$}" description]
    \arrow[phantom, from=hTwo,  to=iTwo,   "\rotatebox{-135}{$\dashv$}" description]
\end{tikzcd}
\]
\vspace{-0.5em}

Let $\bfR$ be the pullback of $(h_1,f_2)$ in the category of posets:
$$
\begin{tikzcd}
\bfR & \bfQ_2\\
\bfQ_1 & \bfP
\Ar{1-1}{1-2}{"\pi_2"}
\Ar{1-1}{2-1}{"\pi_1"'}
\Ar{2-1}{2-2}{"h_1"'}
\Ar{1-2}{2-2}{"f_2"}
\end{tikzcd}.
$$

To show the triangle inequality, it is necessary to construct a module $\Psi\in\vect^{\bfR}$ whose pullbacks
recover $M$ and $O$ as in the following statement,
which is proved by Lemma 6.4 and by the proof of Proposition~6.10 in
\cite{GulenMcCleary}.
We therefore omit the proof here.

\begin{proposition}
\label{prop:compose-couplings}
There exists $\Psi\in\vect^{\bfR}$ and natural isomorphisms
\[
(\iota_1 g_1)^{*}\Psi \cong M,
\qquad
(\iota_2 i_2)^{*}\Psi \cong O,
\]
so that
\[
\bigl(\bfR,\ f_1\pi_1 \dashv \iota_1 g_1,\ 
\ h_2\pi_2 \dashv \iota_2 i_2,\ 
\ \Psi \bigr)
\]
is a Galois coupling of $(M,O)$.  
This composite is unique up to unique isomorphism.
\end{proposition}

\subsection{Transport Distance}
\label{ssec:Galois_metric}

With composition available, we now
give a definition of a transport distance
due to G\"ulen and McCleary, 
and record its basic properties.

\begin{definition}{\cite[Definition~6.1]{GulenMcCleary}}
\label{def:transport-distance}
The \emph{G\"ulen--McCleary distance} between $M,N\in\vect^{\bfP}$ is
\[
\dGT(M,N):=\inf\{\ \cost(\Gamma)\ \mid\ \Gamma\ \text{is a Galois coupling of }(M,N)\ \}.
\]
If there is no Galois coupling between $M$ and $N$, set $\dGT(M,N)=\infty$.
\end{definition}

Now the following is easy to verify, and we just cite it without proof.

\begin{lemma}{\cite[Proposition~6.9]{GulenMcCleary}}
\label{lem:cost-subadditive}
If $\Gamma_1$ is a coupling for $(M,N)$ and $\Gamma_2$ is a coupling for $(N,O)$, and $\Psi$ is their composite from Proposition~\ref{prop:compose-couplings}, then
\[
\cost(\Psi)\ \le\ \cost(\Gamma_1)\ +\ \cost(\Gamma_2).
\]
\end{lemma}

Using this G\"ulen and McCleary proved the following
(see \cite[Proposition~6.10]{GulenMcCleary} for the triangle inequality).

\begin{theorem}{\cite[Section~6]{GulenMcCleary}}
\label{thm:gt-pseudometric}
For a finite poset $\bfP$ with metric $d_{\bfP}$, the function
\[
\dGT:\ \mathrm{Ob}(\vect^{\bfP})\times \mathrm{Ob}(\vect^{\bfP})\longrightarrow [0,+\infty]
\]
is an extended pseudometric.
\end{theorem}

\begin{proof}
Nonnegativity is immediate. For any $M$, the identity coupling $\bfQ=\bfP$, $f=h=\id_{\bfP}$, $g=i=\id_{\bfP}$, $\Gamma=M$ has cost $0$, so $\dGT(M,M)=0$. Symmetry holds because swapping the two insertion legs $(f\dashv g,h\dashv i)$ of any coupling gives a coupling for $(N,M)$ with the same cost (the metric $d_{\bfP}$ is symmetric). The triangle inequality follows from Lemma~\ref{lem:cost-subadditive}.
\end{proof}

\begin{corollary}
\label{cor:gt-metric-on-iso}
On isomorphism classes, $\dGT$ is an extended metric: if $\dGT(M, N)=0$ then $M\cong N$.
\end{corollary}

\begin{proof}
Since $\bfP$ is finite, the set $\{d_{\bfP}(x,y)\mid x,y\in\bfP\}$ is finite; hence every coupling has cost in this finite set, and the infimum in the definition of $\dGT(M,N)$ is a \emph{minimum}. If $\dGT(M,N)=0$, there exists a coupling with $\cost(\Gamma)=0$, so $d_{\bfP}(f(q),h(q))=0$ for all $q\in\bfQ$, hence $f(q)=h(q)$ and therefore $f=h$ as maps $\bfQ\to\bfP$. Using Corollary~\ref{cor:3_adj},
\[
M\ \cong\ f_{!}\Gamma\ =\ h_{!}\Gamma\ \cong\ N,
\]
so $M\cong N$.
\end{proof}

\paragraph{Relation to interleavings.}
When $\bfP=\bbR$ with its usual total order, the G\"ulen--McCleary
distance coincides with the classical interleaving distance.
This equivalence was established by G\"ulen and McCleary
\cite[Proposition~6.4]{GulenMcCleary}.

\subsection{Examples}

We now present two illustrative examples—one in the 1-parameter setting
and one in the 2-parameter setting.  
These will serve as running test cases throughout the paper for the 
G\"ulen--McCleary distance and its comparison with later constructions.
First we collect terminologies for later use.

Recall that the \emph{Hasse quiver} $H(\bfP)$ of a finite poset $\bfP$
is defined  as follows:
The vertex set of $H(\bfP)$ is given by $\bfP$ itself.
For any $a, b \in \bfP$, the number of arrows from $a$ to $b$ 
in $H(\bfP)$ is at most one, and it is one if and only if
$a < b$ and there exist no $c \in \bfP$
such that $a < c < b$.
We sometimes express finite posets as their Hasse quivers below.
Note that any $M \in \vect^\bfP$ is expressed as a representation
of the quiver $H(\bfP)$ satisfying the full commutativity relations.

\begin{definition}
\label{dfn:gen-interval}
Let $\bfP$ be a finite poset.

(1) A full subposet $I$ of $\bfP$ is said to be \emph{connected} (resp.\ \emph{convex in $\bfP$}) if $H(I)$ is connected as a graph (resp.\ if $a \le c \le b$ in $\bfP$ with $a, b \in I$ implies $c \in I$).
Then $I$ is called a \emph{generalized interval} 
(or simply \emph{interval}) in $\bfP$ if
it is both connected and convex in $\bfP$.
By $\gInt \bfP$ we denote the set of all generalized intervals.

(2) A subset $A$ of $\bfP$ is called an \emph{antichain} in $\bfP$ if any distinct elements of $A$ are incomparable. We denote by $\Ac(\bfP)$ the set of all antichains in $\bfP$.
According to \cite{BBH2024approximations},
for any $A, B \in \Ac(\bfP)$, we define $A \le B$ if for any $a \in A$, there exists $b_a \in B$ such that $a \le b_a$, and for any $b \in B$, there exists $a_b \in A$ such that $a_b \le b$\ \footnote{
In other words, $A \le B$ if and only if
``$A \subseteq \dset B$ and $B \subseteq \uset A$''
if and only if
``$\dset A \subseteq \dset B$ and $\uset B \subseteq \uset A$'',
where $\uset A:= \{x \in \bfP \mid a \le x \text{ for some } a \in A\}$, the up-set of $A$ and $\dset B:= \{x \in \bfP \mid x \le b \text{ for some } b \in B\}$, the down-set of $B$.}.
In this case, we define
$$
[A, B]:= \{x \in \bfP \mid a \le x \le b \text{ for some }
a \in A \text{ and for some }b \in B\}.
$$
It is known that 
$\{[A, B] \mid A, B \in \Ac(\bfP), \ A \le B\}$
forms the set of all convex subsets in $\bfP$
(e.g., see \cite[Proposition 2.2]{BBH2024approximations}, \cite[Lemma 3.10]{asashiba2024interval}).
Hence any generalized interval $I$ has the form $I = [A, B]$, where
$A$ (resp.\ $B$) is the set of minimal (resp.\ maximal) elements in $I$
(here $A, B \in \Ac(\bfP)$ automatically).

Note here that a full subposet $S$ of $\bfP$ is an interval 
in the usual sense, namely, $S = [a, b]$ for some $a \le b$ in $\bfP$
if and only if $S$ is a generalized interval $[A, B]$ such that
both $A$ and $B$ are singletons.

(3) For any generalized interval $I$ in $\bfP$, a $\bfP$-module $V_I$
(called the \emph{interval module for $I$}) is defined by $V_I(x):= \k$ if $x \in I$, $V_I(x):= 0$ otherwise; and
$V_I(\al):= \id_\k$ if $\al$ is an arrow $a \to b$ with $a, b \in I$, $V_I(\al):= 0$ otherwise.
As is easily seen, the endomorphism algebra of $V_I$ is isomorphic
to $\k$, and hence it is an indecomopsable $\bfP$-module.
To emphasize the fact that $V_I$ is a $\bfP$-module, we denote it
by $V^\bfP_I$.
\end{definition}

\begin{example}\label{ex:gtd-chain}
Let $\bfP=\{1<2<3<4\}$ with metric $d_{\bfP}(i,j)=|i-j|$.  
Define
\[
M:= V_{[1,1]}\oplus V_{[2,3]},\qquad
N:=V_{[2,3]}.
\]
To construct a low-cost coupling, take the apex poset 
$\bfQ:=\{0<1<2<3<4<5\}$ and define $f,h:\bfQ\to\bfP$
and $g, i \colon \bfP \to \bfQ$ by the following tables:
\[
\begin{array}{c|cccccc}
q & 0 & 1 & 2 & 3 & 4 & 5\\ \hline
f(q) & 1 & 1 & 2 & 3 & 4 & 4\\
h(q) & 1 & 2 & 2 & 3 & 4 & 4\\
\hline
d_\bfP(f(q), h(q)) & 0 & 1 & 0 &0 &0 & 0
\end{array},
\quad\text{and}\quad
\begin{array}{c|cccccc}
p & 1 & 2 & 3 & 4\\ \hline
g(p) & 1 & 2 & 3 & 5\\
i(p) & 0 & 2 & 3 & 5
\end{array}.
\]
A short computation shows that $f\dashv g$ and $h\dashv i$ are Galois insertions.
Let
\[
\Gamma := V^{\bfQ}_{[1,1]}\ \oplus\ V^{\bfQ}_{[2,3]}.
\]  
Then 
\[
g^*\Gamma \cong M,\qquad i^*\Gamma \cong N,
\]
and therefore $(\bfQ,f\dashv g,h\dashv i,\Gamma)$ is a Galois coupling
of the pair $(M,N)$.
The coupling cost is
\[
\cost(\Gamma)=\sup_{q\in\bfQ} d_{\bfP}(f(q),h(q))=1,
\]
and hence $\dGT(M,N)\le 1$.
To see that this bound is sharp, note that $\bfP$ is finite, so 
$d_{\bfP}$ takes only integer values.  A coupling of cost $<1$ would
therefore have cost $0$, which would force $\dGT(M,N)=0$.  By
Corollary~\ref{cor:gt-metric-on-iso}, this would imply $M\cong N$, which
is false.  Hence no cost-$0$ coupling exists.
Thus $\dGT(M,N)\ge 1$.
Combining the upper and lower bounds, we conclude
\[
\dGT(M,N)=1.
\]
\end{example}

Before giving the next examples, we first give a general way of constructing Galois couplings.
Similar constructions are found in \cite[Lemma 6.5]{hem2025poset}.
For a $\bfP$-module $M$, we set $\supp M:= \{x \in \bfP \mid M(x) \ne 0\}$, the support of $M$, and
for a monotone map $\si \colon \bfP \to \bfP$, we set $\bfP^\si$ to be the set
$\{x \in \bfP \mid \si(x) = x\}$ of fixed points of $\bfP$ by $\si$.
Recall that a monotone map $\si$ above is called a \emph{translation} if
$x \le \si(x)$ for all $x \in \bfP$.

\begin{proposition}
\label{prp:gen-Gal-coupl}
Let $\bfP$ be a finite poset with metric $d_\bfP$,
and $\si \colon \bfP \to \bfP$ a translation.
Then we can construct a finite poset $\bfQ$ and Galois insertions
$f : \bfQ \rightleftarrows \bfP : g$ and 
$h : \bfQ \rightleftarrows \bfP : i$ as follows.

(1) Take the disjoint union
\[
\hat{\bfQ} \;:=\; \bfP_L \sqcup \bfP_R
\]
of two copies of $\bfP$, more explicitly,
\[
\bfP_L := \{\, x_L := (x,0) \mid x \in \bfP \,\}, 
\ 
\bfP_R := \{\, x_R := (x,1) \mid x \in \bfP \,\},
\  \text{and }\ \hat{\bfQ}:= \bfP_L \cup \bfP_R.
\]
For each $S \subseteq \bfP$, we set $S_T:= \{x_T \mid x \in S\}$ for all $T \in \{L, R\}$.

(2) We define a binary relation $\le$ on $\hat{\bfQ}$ as follows.
For any $x, y \in \bfP$, and $i, j \in \{0,1\}$,
$$
(x,i) \le (y,j)
\quad\text{if and only if}\quad
\begin{cases}
\si(x) \le y
  & \text{if $i \ne j$},\\
x \le y & \text{otherwise}.
\end{cases}
$$

(3) Then this is a preorder on $\hat{\bfQ}$.
Denote by $\sim_\si$ the equivalence relation with respect to this preorder, namely,
for any $x\, y \in \bfP$, and $i, j \in \{0,1\}$,
$(x,i) \sim_\si (y,j)$ if and only if $(x,i) \le (y,j)$ and $(y,j) \le (x,i)$.

(4) We define $\bfQ$ to be the quotient poset $\hat{\bfQ}/\!\!\sim_\si$, and denote
by $[\blank] \colon \hat{\bfQ} \to \bfQ,\, u \mapsto [u]$ the canonical surjection.
We set $[S]:= \{[s] \mid s \in S\}$ for all $S \subseteq \hat{\bfQ}$.
For each $T \in \{L, R\}$,
this restricts to a bijection $\bfP_T \to [\bfP_T]$,
by which we identify $\bfP_T$ and $[\bfP_T]$.

(5) For each $(x,i) \in \hat{\bfQ}$, we have
$$
[(x,i)] = \begin{cases}
\{x_L, x_R\} & \text{if } \si(x) = x,\\
\{(x,i)\} & \text{otherwise}.
\end{cases}
$$
Thus $\bfQ$ is obtained from $\bfP_L$ and $\bfP_R$ by making the identifications
$x_L = x_R$ for all $x \in \bfP^\si$:
\begin{equation}
\label{eq:gen-str-bfQ}
\bfP_L \cup \bfP_R = \bfQ, \ \text{and } \ \bfP_L\cap \bfP_R = [\bfP^\si_L] = [\bfP^\si_R].
\end{equation}

(6) The following defines a Galois insertion $f : \bfQ \rightleftarrows \bfP : g$.
$$
g(x) := x_L,
\quad\text{and}\quad
f(u):= \begin{cases}
x &\text{if $u = x_L$ for some $x \in \bfP$},\\
\si(x) & \text{if $u = x_R$ for some $x \in \bfP$}.
\end{cases}
$$
Note that $f$ is well-defined by \eqref{eq:gen-str-bfQ}.

(7) Similarly, the following defines a Galois insertion $h : \bfQ \rightleftarrows \bfP : i$.
$$
i(x) := x_R,
\quad\text{and}\quad
h(u):= \begin{cases}
\si(x) & \text{if $u = x_L$ for some $x \in \bfP$},\\
x &\text{if $u = x_R$ for some $x \in \bfP$}.
\end{cases}
$$

(8)
Now let $M$ and $N$ be $\bfP$-modules.
If there exists a $\bfQ$-module $\Ga$ such that
$g^*(\Ga) \cong M$ and $i^*(\Ga) \cong N$, then
$(\bfQ, f \dashv g, h \dashv i, \Ga)$ is a Galois coupling 
for $(M, N)$ by definition.
If this is the case, the $\cost(\Ga)$ is given by
$$
\cost(\Ga) = \sup_{u \in \bfQ} d_\bfP(f(u), h(u)) = \sup_{x \in \bfP} d_\bfP(x, \si(x)).
$$
For example, If
\begin{equation}
\label{eq:M=N-si}
M = N\si,
\end{equation}
then we can construct a Galois coupling
$(\bfQ, f \dashv g, h \dashv i, \Ga)$
by setting $\Ga:= h^*(N)$.
\end{proposition}

\begin{proof}
(3) For any $x \in \bfP$ and $i \in \{0,1\}$, since $x \le x$, we have $(x,i) \le (x,i)$.

Let $(x,i), (y,j), (z,k) \in \hat{\bfQ}$, and assume that
$(x,i) \le (y,j)$ and $(y,j) \le (z,k)$.
Then we have four cases: (a) $i=j$, $j = k$; (b) $i = j$, $j \ne k$;
(c) $i \ne j$, $j = k$; and (d) $i \ne j$, $j \ne k$.
It is easy to see that in any case, we have $(x,i) \le (z,k)$ by definition.

(5)
Let $(x,i), (y,j) \in \hat{\bfQ}$, and assume that
$(x,i) \sim_\si (y,j)$.  Then 

Case $i = j$: In this case, we have $x \le y$ and $y \le x$, and hence $x = y$.

Case $i \ne j$: In this case, we have $\si(x) \le y$ and $\si(y) \le x$, and hence
$\si^2(x) \le \si(y) \le x \le \si(x) \le \si^2(x)$, which shows that
$\si(x) = x = \si(y)$.  By symmetry we have $\si(y) = y$.
Thus $x = y$ and $\si(x) = x$.
This shows statement (5).

(6) First we show that for any $u \in \bfQ$, we have $f(u) = \min\{y \in \bfP \mid u \le g(y)\}$.
When $u = [x_L]$ for some $x \in \bfP$,
we have the following equivalences for any $y \in \bfP$:
$$
u \le g(y) \Longleftrightarrow [x_L] \le [y_L] \Longleftrightarrow x \le y.
$$
Hence $\min\{y \in \bfP \mid u \le g(y)\} = x = f(u)$.

When $u = [x_R]$ for some $x \in \bfP$, we have the following equivalences for any $y \in \bfP$:
$$
u \le g(y) \Longleftrightarrow [x_R] \le [y_L] \Longleftrightarrow \si(x) \le y.
$$
Hence $\min\{y \in \bfP \mid u \le g(y)\} = \si(x) = f(u)$.

Second, we show that for any $x \in \bfP$, we have $g(x) = \max\{u \in \bfQ \mid f(u) \le x\}$.
Note that
$$
\{u \in \bfQ \mid f(u) \le x\} = \{[y_L] \mid y \in \bfP, y \le x\} \cup \{[z_R] \mid z \in \bfP, \si(z) \le x\}.
$$
Then clearly $[x_L] = \max \{[y_L] \mid y \in \bfP, y \le x\}$, and for any $z \in \bfP$,
if $\si(z) \le x$, then $[z_R] \le [x_L]$, which shows that
$\max\{u \in \bfQ \mid f(u) \le x\} = [x_L] = g(x)$.

(8) This is verified as follows:
$$
g^*(\Ga) = N \circ h \circ g = N \circ \si = M,\ \text{and }\ 
i^*(\Ga) = N \circ h \circ i = N \circ \id_{\bfP} = N.
$$
\end{proof}

We now give further sufficient conditions when $\Ga$ can be constructed from $M, N$ in the setting above.

\begin{proposition}
\label{prp:Ga-MN}
Consider the same setting as in Proposition \ref{prp:gen-Gal-coupl}.
We define a subposet  $\bfQ'$ of $\bfQ$ as follows:
The underlying set of $\bfQ'$ is the same as that of $\bfQ$.
The partial order $\le'$ of $\bfQ'$ is defined by
$$
[(x,i)] \le' [(y,j)] :\Longleftrightarrow i=j \text{ and } x \le y \text{ in } \bfP
$$
for all $(x,i), (y,j) \in \hat{\bfQ}$.
For each $T \in \{L, R\}$, denote by $p_T \colon \bfP_T \to \bfP$,
$x_T \mapsto x$ the first projection.

(1) Assume that
\begin{equation}
\label{eq:M-N-compatible}
M|_{\bfP^\si} = N|_{\bfP^\si}
\end{equation}
(for example, this trivially holds if
$\bfP^\si  \cap \supp M = \varnothing = \bfP^\si \cap \supp N$), 
Then we define $\Ga' \in \vect^{\bfQ'}$ by
\[
\Gamma'|_{\bfP_L} = M \circ p_L,
\qquad
\Gamma'|_{\bfP_R} = N \circ p_R,
\]
which is well-defined by \eqref{eq:M-N-compatible} and \eqref{eq:gen-str-bfQ}.
Regard $\k[\bfQ]$ as $\k[\bfQ]$-$\k[\bfQ']$-bimodule via
the inclusion $q \colon \bfQ' \to \bfQ$, and consider the 
adjoint pair given by this bimodule:
$$
\k[\bfQ] \otimes_{\k[\bfQ']}\blank \colon \vect^{\bfQ'} 
\rightleftarrows \vect^{\bfQ} : \vect^{\bfQ}(\k[\bfQ], \blank).
$$
Set $\Ga$ to be the left Kan extension of $\Ga'$ along $q$:
$$
\Ga:= \k[\bfQ] \otimes_{\k[\bfQ']} \Ga'.
$$
In addition, assume that
$$
\Ga|_{\bfQ'} \cong \Ga'
$$
(this holds for example, if the unit morphism
$$
\et_{\Ga'} \colon \Ga' \to \vect^{\bfQ}(\k[\bfQ], \k[\bfQ] \otimes_{\k[\bfQ']} \Ga')
$$
is an isomorphism). Then 
$(\bfQ, f \dashv g, h \dashv i, \Ga)$ is a Galois coupling for $(M,N)$. 

(2) For each $1 \le n\in\bbN$, we write $[n]=\{1,\dots,n\}$. 
Let $M= \Ds_{i=1}^m V_{I_i}^\bfP$ and
$N = \Ds_{j=1}^n V_{J_j}^\bfP$ for some $I_i, J_j \in \gInt \bfP$
with $i \in [m], j \in [n]$ for some $1 \le m, n \in \bbN$.
Then it is obvious that $I_{i,L}:= (I_i)_L$ and $J_{j,R}:= (J_j)_R$ are in $\gInt \bfQ'$ for all $i \in [m], j \in [n]$.
Assume that $I_{i,L}$ and $J_{j,R}$ are convex also in $\bfQ$, or equivalently,
$I_{i,L}, J_{j,R} \in \gInt \bfQ$ for all $i \in [m], j \in [n]$.
Then by setting
$$
\Ga:= \left(\Ds_{i=1}^m V_{I_{i,L}}^\bfQ\right) \ds 
\left(\Ds_{j=1}^n V_{J_{j,R}}^\bfQ\right)
$$
We have a Galois coupling $(\bfQ, f \dashv g, h \dashv i, \Ga)$
for $(M,N)$.
\end{proposition}

\begin{proof}
We can define a monotone map $g'\colon \bfP \to \bfQ'$ by setting
$g'(x):= g(x)$ for all $x \in \bfP$.  Then $g = q \circ g'$ as monotone maps.

(1) By definition of $\Ga'$ and the additional assumption,
we have the following commutative diagrams:
\[
\begin{tikzcd}
\bfP_L & \bfQ'\\
\bfP & \vect
\Ar{1-1}{1-2}{hook}
\Ar{2-1}{2-2}{"M"'}
\Ar{1-1}{2-1}{"p_L"'}
\Ar{1-2}{2-2}{"\Ga'"}
\Ar{2-1}{1-2}{"g'"'}
\end{tikzcd}
\quad\text{and}\quad
\begin{tikzcd}
	\bfP && \bfQ \\
	& {\bfQ'} \\
	& \vect
	\arrow["g", from=1-1, to=1-3]
	\arrow["{g'}", from=1-1, to=2-2]
	\arrow["M"', from=1-1, to=3-2, bend right]
	\arrow["\Ga", from=1-3, to=3-2, bend left]
	\arrow["q", from=2-2, to=1-3]
	\arrow["{\Ga'}"', ""{name=A}, from=2-2, to=3-2]
    \Ar{A}{1-3}{"\cong"', Rightarrow, shorten=20pt}
\end{tikzcd},
\]
where only the triangle containing ``$\Rightarrow$'' is commutative 
up to natural isomorphsim.
Therefore, $g^*(\Ga) = \Ga \circ g \cong M$.
Similarly, we have $i^*(\Ga) \cong N$.

(2)
Set $\calI:= \{I_{i,L}, J_{j,R} \mid i \in [m], n \in [n]\}$.
Since $\calI \subseteq \gInt \bfQ$, we have
$V_I^{\bfQ} \circ q' = V_I^{\bfQ'}$ for all $I \in \calI$.
By definition of interval modules, we have the following commutative diagram for all $i \in [m]$:
\[
\begin{tikzcd}
	\bfP && \bfQ \\
	& {\bfQ'} \\
	& \vect
	\arrow["g", from=1-1, to=1-3]
	\arrow["{g'}", from=1-1, to=2-2]
	\arrow["{V_{I_{i}}^\bfP}"', from=1-1, to=3-2, bend right]
	\arrow["{V_{I_{i,L}}^\bfQ}", from=1-3, to=3-2, bend left]
	\arrow["q", from=2-2, to=1-3]
	\arrow["{V_{I_{i,L}}^{\bfQ'}}"', from=2-2, to=3-2]
\end{tikzcd}.
\]
Therefore, $g^*(V_{I_{i,L}}^\bfQ) = V_{I_{i,L}}^\bfQ\circ g
= V_{I_i}^\bfP$ for all $i \in [m]$.
Similarly, we have $i^*(V_{I_{j,R}}^\bfQ) = V_{I_{j}}^\bfP$
for all $j \in [n]$.
Thus we have $g^*(\Ga) = M$ and $i^*(\Ga) = N$.
\end{proof}

\begin{example}\label{ex:gtd-2d}
Let $\bfP:= \{1,2,3\}^2$ with the product order and 
$d_{\bfP}((i,j),(i',j'))=\max\{|i-i'|,|j-j'|\}$.
We denote $(x,y)$ simply by $xy$ for all $x, y \in \{1,2,3\}$.
Then $\bfP$ is visualized by its Hasse quiver as follows:
$$
H(\bfP) = 
\begin{tikzcd}
\Nname{13}13 & \Nname{23}23 & \Nname{33}33\\
\Nname{12}12 & \Nname{22}22 & \Nname{32}32\\
\Nname{11}11 & \Nname{21}21 & \Nname{31}31
\Ar{13}{23}{} \Ar{23}{33}{}
\Ar{12}{22}{} \Ar{22}{32}{}
\Ar{11}{21}{} \Ar{21}{31}{}
\Ar{11}{12}{} \Ar{12}{13}{}
\Ar{21}{22}{} \Ar{22}{23}{}
\Ar{31}{32}{} \Ar{32}{33}{}
\end{tikzcd}.
$$
Set $M_1:= V_{[12, 12]},\, M_2:= V_{[21,31]},\, M:= M_1 \ds M_2$ and $N:= V_{[22,\{23,32\}]}$, which are visualized as representations of $H(\bfP)$ as follows:
\begin{equation}
\label{eq:2D-M-N}
M = 
\begin{tikzcd}
\Nname{13}0 & \Nname{23}0 & \Nname{33}0\\
\Nname{12}\k & \Nname{22}0 & \Nname{32}0\\
\Nname{11}0 & \Nname{21}\k & \Nname{31}\k
\Ar{13}{23}{} \Ar{23}{33}{}
\Ar{12}{22}{} \Ar{22}{32}{}
\Ar{11}{21}{} \Ar{21}{31}{"1"}
\Ar{11}{12}{} \Ar{12}{13}{}
\Ar{21}{22}{} \Ar{22}{23}{}
\Ar{31}{32}{} \Ar{32}{33}{}
\end{tikzcd},\ \text{and}\quad
N = 
\begin{tikzcd}
\Nname{13}0 & \Nname{23}\k & \Nname{33}0\\
\Nname{12}0 & \Nname{22}\k & \Nname{32}\k\\
\Nname{11}0 & \Nname{21}0 & \Nname{31}0
\Ar{13}{23}{} \Ar{23}{33}{}
\Ar{12}{22}{} \Ar{22}{32}{"1"}
\Ar{11}{21}{} \Ar{21}{31}{}
\Ar{11}{12}{} \Ar{12}{13}{}
\Ar{21}{22}{} \Ar{22}{23}{"1"}
\Ar{31}{32}{} \Ar{32}{33}{}
\end{tikzcd}.
\end{equation}

We define a monotone map $\si \colon \bfP \to \bfP$ by
$\si(xy):= (\min\{x+1,3\}, \min\{y+1,3\})$ for each
$xy \in \bfP$.
Then as is easily seen, $\bfP^\si = \{33\}$.
To couple the modules $M, N$, we apply Proposition \ref{prp:gen-Gal-coupl} with this $\si$.
In this case, 
it is easy to verify that $[12,12]_L, [21,31]_L, [22, \{23,32\}]_R$ are
convex in $\bfQ$.
Hence by Proposition \ref{prp:Ga-MN} (2),
$\Ga:= V_{[12,12]_L}^\bfQ \ds V_{[21,31]_L}^\bfQ \ds V_{[22, \{23,32\}]_R}^\bfQ$
defines
a Galois coupling $(\bfQ, f \dashv g, h \dashv i, \Ga)$.
Then by Proposition \ref{prp:gen-Gal-coupl},
it follows from the definition of $\si$ that
\[
\cost(\Gamma)
\;=\;
\sup_{x \in \bfP} d_{\bfP}(x, \si(x)) = 1,
\]
and therefore
\[
\dGT^{\bfP}(M,N) \;\le\; 1.
\]

To see that this bound is sharp, note that $d_{\bfP}$ takes only integer values.
Thus a cost strictly smaller than $1$ would have to be $0$.
A cost-$0$ coupling forces $f=h$ and hence $M \cong N$
by Corollary~\ref{cor:gt-metric-on-iso}, but the modules are not isomorphic.
Combining the upper and lower bounds yields
\[
\dGT^{\bfP}(M,N) = 1.
\]
\end{example}

\begin{remark}
In Example \ref{ex:gtd-2d}, it is also possible to apply Proposition
\ref{ex:gtd-2d} (1) instead of (2).
In that case, we have a different $\Ga$ (to distinguish, we denote it
by $\Ga^!$).
First, $H(\bfQ')$ is given as follows:
$$
H(\bfQ') = 
\begin{tikzcd}
\Nname{13}13_R & \Nname{23}23_R & \Nname{33}33\\
\Nname{12}12_R & \Nname{22}22_R & \Nname{32}32_R\\
\Nname{11}11_R & \Nname{21}21_R & \Nname{31}31_R
\Ar{13}{23}{} \Ar{23}{33}{}
\Ar{12}{22}{} \Ar{22}{32}{}
\Ar{11}{21}{} \Ar{21}{31}{}
\Ar{11}{12}{} \Ar{12}{13}{}
\Ar{21}{22}{} \Ar{22}{23}{}
\Ar{31}{32}{} \Ar{32}{33}{}
\end{tikzcd}
\begin{tikzcd}
\Nname{13}13_L & \Nname{23}23_L & \Nname{33}33\\
\Nname{12}12_L & \Nname{22}22_L & \Nname{32}32_L\\
\Nname{11}11_L & \Nname{21}21_L & \Nname{31}31_L
\Ar{13}{23}{} \Ar{23}{33}{}
\Ar{12}{22}{} \Ar{22}{32}{}
\Ar{11}{21}{} \Ar{21}{31}{}
\Ar{11}{12}{} \Ar{12}{13}{}
\Ar{21}{22}{} \Ar{22}{23}{}
\Ar{31}{32}{} \Ar{32}{33}{}
\end{tikzcd},
$$
where we identify $33_L = 33 = 33_R$, and
$H(\bfQ)$ is obtained from $H(\bfQ')$ by adding arrows
$x_L \to \si(x)_R$ and $x_R \to \si(x)_L$ for all $x \in \bfP$.
Then $\Ga^!$ is given as follows as a representation of $H(\bfQ)$:
$$
\Ga^! = 
\begin{tikzcd}
\Nname{13}0 & \Nname{23}0 & \Nname{33}0\\
\Nname{12}\k & \Nname{22}0 & \Nname{32}0\\
\Nname{11}0 & \Nname{21}\k & \Nname{31}\k
\Ar{13}{23}{} \Ar{23}{33}{}
\Ar{12}{22}{} \Ar{22}{32}{}
\Ar{11}{21}{} \Ar{21}{31}{"1"}
\Ar{11}{12}{} \Ar{12}{13}{}
\Ar{21}{22}{} \Ar{22}{23}{}
\Ar{31}{32}{} \Ar{32}{33}{}
\end{tikzcd}
\begin{tikzcd}
\Nname{13}0 & \Nname{23}\k & \Nname{33}0\\
\Nname{12}0 & \Nname{22}\k & \Nname{32}\k\\
\Nname{11}0 & \Nname{21}0 & \Nname{31}0
\Ar{13}{23}{} \Ar{23}{33}{}
\Ar{12}{22}{} \Ar{22}{32}{"1"}
\Ar{11}{21}{} \Ar{21}{31}{}
\Ar{11}{12}{} \Ar{12}{13}{}
\Ar{21}{22}{} \Ar{22}{23}{"1"}
\Ar{31}{32}{} \Ar{32}{33}{}
\end{tikzcd},
$$
where for each additional arrow $\pi$, we have
$$
\Ga^!(\pi) = \begin{cases}
\id_\k & \text{if $\pi = \pi_{\si(x)_R, x_L}$ for $x = 21, 31$},\\
0 &\text{otherwise},
\end{cases}
$$
whereas $\Ga(\pi) = 0$ for all additional arrows $\pi$.
We remark that the Galois coupling given by
$(\bfQ,\ f\dashv g,\ h\dashv i,\ \Gamma^!)$ 
is precisely a $1$-interleaving in the sense of Lesnick’s theory of multiparameter persistence~\cite{Lesnick2015}.
\end{remark}

The following example will be used to illustrate the proof of
Theorem \ref{thm:stability}.

\begin{example}
\label{exm:stability}
Let $\bfP$ be a poset with $H(\bfP) = (1 \to 2)$ with metric $d_{\bfP}(a,b)=|a - b|$
for all $a, b \in \bfP$, and
define a monotone map $\si \colon \bfP \to \bfP$ by $\si(1) = 2 = \si(2)$.
Then $\bfP^\si = \{2\}$.
Apply Proposition \ref{prp:gen-Gal-coupl} to construct 
a finite poset $\bfQ$ and Galois insertions
$f : \bfQ \rightleftarrows \bfP : g$ and 
$h : \bfQ \rightleftarrows \bfP : i$.
In this case, $\bfQ$ is expressed by its Hasse quiver
$$
H(\bfQ) = (1_L\to 2 \leftarrow 1_R),
$$
where we set $2:= 2_L = 2_R$.
$f, h \colon \bfQ \to \bfP$ are defined as follows:
for any $x \in \bfP$,
$$
f(x_L):= x,\, f(1_R):= 2, \quad\text{and}\quad
h(x_R):= x,\, h(1_L):= 2.
$$
Their right adjoints $g, i \colon \bfP \to \bfQ$ are
defined as follows, respectively: for any $x \in \bfP$,
$$
g(x):= x_L,\quad\text{and}\quad i(x):= x_R.
$$

Consider $\bfP$-modules $M = V_\bfP \oplus V_{\{1\}}$ and
$N = V_\bfP$.
They are visualized as representations of the quiver $H(\bfP)$ as follows:
$$
M = \k^2 \ya{[1\ 0]} \k, \quad N = \k \ya{1} \k.
$$
It is obvious that \eqref{eq:M-N-compatible} is satisfied, and
note that $\bfQ' = \bfQ$ in this case.
Hence by Proposition \ref{prp:Ga-MN} (1),
we can take $\Ga:= \Ga'$, which is givne by
$(\k^2 \xrightarrow{[1\ 0]} \k \xleftarrow{1} \k)$ as a
representation of $H(\bfQ)$, and this defines a Galois coupling
$(\bfQ, f \dashv g, h \dashv i, \Ga)$ for $(M, N)$.

Since $d_\bfP(x, \si(x)) = 1, 0$
for $x = 1, 2$, respectively, we have
$\cost(\Ga) = 1$.
Hence $\dGT^{\bfP}(M, N) \le 1$.
Moreover, since $M \not\cong N$, we have $\dGT^{\bfP}(M, N) > 0$
by Corollary \ref{cor:gt-metric-on-iso}.
Therefore, $\dGT^{\bfP}(M, N) = 1$ because the image of
$\dGT^{\bfP}$ consists of integers in this case.
\end{example}

\section{Complex Matching Distance }
\label{sec:bneck}

In this section we define the complex matching distance. The construction is carried out first for bounded complexes of finitely generated projective $\k[\bfP]$-modules: after allowing padding by contractible cones, we compare complexes by degreewise matchings of their indecomposable summands that are compatible with the differentials. This yields an extended metric on the homotopy category $\Kb(\prj \k[\bfP])$. We then restrict the construction to non-negative complexes, and in particular to minimal projective resolutions of objects of $\vect^\bfP$, obtaining the version that we will compare with the GM distance in the next section. Throughout this section, $(\bfP,d_{\bfP})$ is a finite metric poset.
\subsection{Bounded complexes of finitely generated projectives}

We denote by $\Cb(\prj \k[\bfP])$ the category of bounded chain complexes of objects in $\prj \k[\bfP]$, and 
by $\Kb(\prj \k[\bfP])$ the homotopy category of $\Cb(\prj \k[\bfP])$.

Let  $M\down=(M_i,\partial^M_i)_{i\in \bbZ} \in \Cb(\prj \k[\bfP])$.
Then by Lemmas \ref{lem:ind-proj} and \ref{lem:mor-betw-proj},
we may set
\begin{equation}
\label{eq:proj-terms-cpx}
M_i = \Ds_{x^{(i)} \in \smd(M_i)} \k[\bfP]_{x^{(i)}}
\end{equation}
for all $i \in \bbZ$, and
\begin{equation}
\label{eq:diffl-cpx}
\partial^M_i = [a_{x^{(i+1)}, x^{(i)}}^{(i)}\, \ro(x^{(i+1)}\ge x^{(i)})]_{(x^{i}, x^{i+1})\, \in \smd(M_{i}) \times \smd(M_{i+1})}
\colon M_{i+1} \to M_i
\end{equation}
with
$\Mat(\partial^M_i) = {}^t[a_{x^{(i+1)}, x^{(i)}}^{(i)}]_{(x^{(i+1)}, x^{(i)})\, \in \smd(M_{i+1}) \times \smd(M_{i})}$ for all $i \in \bbZ$.

The \emph{size vector} $|M\down|$ of $M\down$ is
defined as $|M\down|:=(|M_i|)_{i\in \bbZ}$. For indecomposable projectives $U,V$ (identified with representables $U \cong \k[\bfP]_x$, $V \cong \k[\bfP]_y$), we set
\[
\dist(U,V):=d_{\bfP}(x,y).
\]

If $E$ is any projective object in $\vect^{\bfP}$, the mapping cone $\Cone(\id_E)$ is the two-term acyclic complex
\[
\cdots \to 0 \to \overset{\smat{(\deg\,1)\\\ }}{E} \xrightarrow{\ \id_E\ } \overset{\smat{(\deg\,0)\\\ }}{E} \to 0 \to \cdots,
\]
concentrated in degrees 0 and 1; its shift $\Cone(\id_E)[a]$ places the two copies of $E$ in degrees $a$ and $a+1$ for all $a \in \bbZ$.
An object $A\down \in \Cb(\prj \k[\bfP])$ is called
a \emph{contractible cone} if 
$A \cong \Cone(\id_E)[a]$ in $\Cb(\prj \k[\bfP])$
for some $E \in \prj \k[\bfP]$ and $a \in \bbZ$.

First of all, we introduce an extended pseudometric on $\Cb(\prj \k[\bfP])$, which will be shown to induce an extended metric on $\Kb(\prj \k[\bfP])$.
In the next subsection, this metric restricts to the subset consisting of the minimal projective resolutions of objects in $\vect^\bfP$ .
We begin with the following well-known facts (for fundamental facts on homological algebra see e.g., 
\cite{weibel1994introduction}).

\begin{proposition}
\label{prp:htp-eq=upto-acyc}
Let $M\down, N\down \in \Cb(\prj \k[\bfP])$.
Then we have the following:
\begin{enumerate}[label=(\arabic*)]
\item 
$M\down \cong N\down$ in $\Kb(\prj k[\bfP])$ if and only if
there exist acyclic complexes $A\down, B\down \in \Cb(\prj \k[\bfP])$
such that
\[
M\down \ds A\down \cong N\down \ds B\down \quad \text{ in } \Cb(\prj \k[\bfP]).
\]

\item
Any acyclic complex $A\down \in \Cb(\prj \k[\bfP])$
is isomorphic in $\Cb(\prj \k[\bfP])$
to the direct sum of finitely many contractible cones.
\end{enumerate}
\end{proposition}

\begin{proof}
(1) ($\Rightarrow$).
Assume that $M\down$ and $N\down$ are isomorphic in $\Kb(\prj \k[\bfP])$.
Then they are chain-homotopy equivalent.
By the standard mapping-cylinder argument for
homotopy equivalences of complexes, there exist contractible complexes
$C\down$ and $D\down$ such that
\[
M\down \ds C\down \cong N\down \ds D\down
\]
in $\Cb(\prj \k[\bfP])$. Since contractible complexes are acyclic, this
gives (b).

($\Leftarrow$).
Since acyclic objects are zero in the derived category $\Db(\vect^\bfP)$, 
the existence of a canonical embedding $\Kb(\prj \k[\bfP]) \to \Db(\vect^\bfP)$ shows that all acyclic objects
in $\Kb(\prj \k[\bfP])$ are zero.
(Alternatively, 
Since contractible cones are zero in $\Kb(\prj \k[\bfP])$,
this assertion follows also from (2).)

(2)
We argue by induction on the number of nonzero terms of $A\down$.
Since there is nothing to prove if $A\down=0$, we may
assume $A\down \neq 0$. Let $m$ be the smallest integer such that
$A_m\neq 0$. Then $A_{m-1}=0$, and $A\down$ has the form
$$
\cdots \to A_{m+1} \ya{\partial^A_m} A_m \to 0 \to 0 \to \cdots.
$$
Since $A\down$ is acyclic, $\partial_{m}^A$ is an epimorphism in $\vect^\bfP$, which splits because $A_m$ is projective. 
Thus we may write
\[
A_{m+1}\cong \Ker(\partial_{m}^A)\oplus A_m
\]
in such a way that $\partial_{m}^A$ restricts to the identity on the
second summand.

It follows that the two-term subcomplex
\[
0 \to A_m \xrightarrow{\id_{A_m}} A_m \to 0
\]
splits off from $A\down$; this subcomplex $C\down$ is a contractible cone.
Hence
\[
A\down \cong C\down \ds A'\down
\]
in $\Cb(\prj \k[\bfP])$, where 
$A'\down$ is again an acyclic complex in $\Cb(\prj \k[\bfP])$ with fewer nonzero terms than $A\down$.
By the induction hypothesis, $A'\down$ is isomorphic to a finite direct
sum of contractible cones. Therefore so is $A\down$.
\end{proof}

\begin{ntn}
\label{ntn:Pad}
Let $M\down, E\down, N\down \in \Cb(\prj \k[\bfP])$.
We say that $E\down$ is a \emph{padding of $M\down$ by contractible cones} if $E\down \cong M\down \ds A\down$ in $\Cb(\prj \k[\bfP])$ for some
finite direct sum $A\down$ of contractible cones.
By $\Pad(M)$, we denote the set of all paddings of $M\down$ by contratible cones:
\begin{multline}
\Pad(M\down):= \{E\down \in \Cb(\prj \k[\bfP])
\mid E\down \cong M\down \ds A\down
\text{ in } \Cb(\prj \k[\bfP])\\
\text{ for some finite direct sum $A\down$ of
contractible cones}\}.
\end{multline}
Using this, we set
$$
\Pad(M\down, N\down):= \{(E\down, F\down) \in \Pad(M\down) \times \Pad(N\down) \mid |E\down| = |F\down|\}.
$$
\end{ntn}

The set $\Pad$ characterizes for two objects in $\Kb(\prj \k[\bfP])$ to be isomorphic as follows, which immediately follows by Proposition \ref{prp:htp-eq=upto-acyc}:

\begin{lemma}
Let $M\down, N\down \in \Cb(\prj \k[\bfP])$.
Then $M\down \cong N\down$ in
$\Kb(\prj \k[\bfP])$ if and only if
$$
\Pad(M\down \ds A\down) = \Pad(N\down \ds B\down)
$$
for some finite direct sums $A\down, B\down$ of contractible cones.
\end{lemma}

\begin{definition}
\label{dfn:matching-cpx}
Let $M\down, N\down \in \Cb(\prj \k[\bfP])$,
$(E\down,F\down)\in\Pad(M\down, N\down)$, and set
$$
\Mat(\partial^E_i) = {}^t[a_{x^{(i+1)}, x^{(i)}}^{(i)}]_{(x^{(i+1)}, x^{(i)})\, \in \smd(E_{i+1}) \times \smd(E_{i})},
\quad
\Mat(\partial^F_i) = {}^t[b_{y^{(i+1)}, y^{(i)}}^{(i)}]_{(y^{(i+1)}, y^{()})\, \in \smd(F_{i+1}) \times \smd(F_{i})}
$$
for all $i \in \bbZ$ as in \eqref{eq:proj-terms-cpx} and \eqref{eq:diffl-cpx}.
Then a \emph{pre-matching} of $(E\down, F\down)$ is a family of degreewise bijections
\begin{equation}
\label{eq:prematching}
B=(B_i)_{i\ge 0},\quad B_i\in\Bij(\smd(E_i),\,\smd(F_i)).
\end{equation}
A \emph{matching} of $(E\down, F\down)$ is a pre-matching of $(E\down, F\down)$
satisfying the \emph{compatibility condition on differentials}:
$$
a^{(i)}_{x^{(i+1)},x^{(i)}} = b^{(i)}_{B_{i+1}(x^{(i+1)}), B_i(x^{(i)})}
$$
for all $i \in \bbZ$, $x^{(i)} \in \smd(E_i), x^{(i+1)} \in \smd(E_{i+1})$
whenever $\smd(E_{i+1}) \times \smd(E_i) \ne \varnothing$.
Note that this condition implies that if $a^{(i)}_{x^{(i+1)},x^{(i)}} \ne 0$, then $B_{i+1}(x^{(i+1)}) \ge B_i(x^{(i)})$.

By $\Match(E\down, F\down)$ (resp.\ $\pMatch(E\down, F\down)$),
we denote the set of matchings (resp.\ pre-matchings) of $(E\down, F\down)$.

Define the \emph{cost} of a pre-matching $B$ as the $L^\infty$-type aggregate of the underlying poset metric,
\[
\cost(B)\ :=\
\sup\bigl\{\,\dist\bigl(U,\ B_i(U)\bigr) \bigm| i\in \bbZ,\ U \in \Summands(E_i)\,\bigr\} = \sup\{d_\bfP(x, B_i(x)) \mid i \in \bbZ, x \in \smd(E_i)\}.
\]
(Equivalently, $\cost(B)=\sup\{\dist(B_i^{-1}(V),V)\mid i\in \bbZ,\ V\in\Summands(F_i)\}$.)  
We refer to the quantity
\[
\distM(E\down,F\down)\ :=\ \inf_{\,B\in\Match(E\down,F\down)}\ \cost(B)
\]
as the \emph{matching distance} for complexes with a fixed size vector.
It is defined to be infinity if $\Match(E\down, F\down) = \varnothing$.
\end{definition}

\begin{lemma}\label{lem:reg-triangle-cpx}
Let $M\down, N\down, O\down \in \Cb(\prj k[\bfP])$, and 
$(E\down, F\down) \in \Pad(M\down, N\down)$,
$(F\down, G\down) \in \Pad(N\down, O\down)$.
Then
\[
\distM(E\down,G\down)\ \le\ \distM(E\down,F\down)\ +\ \distM(F\down,G\down).
\]
\end{lemma}

\begin{proof}
If $(E\down, F\down)$ or $(F\down, G\down)$ has no matching, then
the right hand side is infinity, and the inequality holds.
Therefore, we may assume that both of them have some matchings.
Choose matchings
\begin{equation}
\label{eq:B-C-cpx}
B\in\Match(E\down,F\down),
\qquad
C\in\Match(F\down,G\down).
\end{equation}
Define $D\in\Match(E\down,G\down)$ degreewise by
\[
D_i := C_i\circ B_i
\qquad (i\ge 0),
\]
so that $D_i\colon\Summands(E_i)\to\Summands(G_i)$ is again a bijection.
Then $D \in \pMatch(E\down, G\down)$.
Moreover, $D$ satisfies the compatibility condition on differentials for $(E\down, G\down)$,
and hence $D \in \Match(E\down, G\down)$.
Indeed, set $\Mat(\partial^E_i), \Mat(\partial^F_i)$ to be as in Definition 
\ref{dfn:matching}, and
$$
\Mat(\partial^G_i):=
{}^t[c_{z^{(i+1)}, z^{(i)}}^{(i)}]_{(z^{(i+1)}, z^{(i)})\, \in \smd(G_{i+1}) \times \smd(G_{i})}
$$
for all $i \in \bbZ$.
Then it follows from \eqref{eq:B-C-cpx} that
$$
a^{(i)}_{x^{(i+1)},x^{(i)}} = b^{(i)}_{B_{i+1}(x^{(i+1)}), B_i(x^{(i)})}
= c^{(i)}_{C_{i+1}(B_{i+1}(x^{(i+1)})), C_i(B_i(x^{(i)}))}
= c^{(i)}_{D_{i+1}(x^{(i+1)}), D_i(x^{(i)})}
$$
for all $i \in \bbZ$, $x^{(i)} \in \smd(E_i), x^{(i+1)} \in \smd(E_{i+1})$,
as desired.

For any $i$ and any $U\in\Summands(E_i)$, the metric $\dist$ on indecomposable
projectives satisfies
\[
\dist\bigl(U,D_i(U)\bigr)
=
\dist\bigl(U,C_i(B_i(U))\bigr)
\le
\dist\bigl(U,B_i(U)\bigr)
+
\dist\bigl(B_i(U),C_i(B_i(U))\bigr).
\]
Taking the supremum over all $i$ and $U$ gives
\[
\cost(D)\ \le\ \cost(B)\ +\ \cost(C),
\]
and therefore
\[
\distM(E\down,G\down)
\ \le\
\distM(E\down,F\down)\ +\ \distM(F\down,G\down). \qedhere
\]
\end{proof}

\begin{definition}
\label{dfn:CM-dist-cpx}
For $M\down, N\down \in \Cb(\prj \k[\bfP])$ (not necessarily of the same size vector), the \emph{complex matching distance} is
\[
\distCM\bigl(M\down, N\down\bigr)
\ :=\ 
\inf_{\ (E\down,F\down)\in\Pad(M\down, N\down)}\ \distM(E\down,F\down).
\]
We adopt the extended-value convention that $\distCM\bigl(M\down, N\down\bigr)=\infty$ if $\Pad(M\down, N\down)=\varnothing$.
\end{definition}

\begin{remark}
By using $\pMatch(E\down, F\down)$ instead of $\Match(E\down, F\down)$ for a pair $(E\down, F\down) \in \Pad(M\down, N\down)$, we can define other functions
$$
\begin{aligned}
\distM'(E\down, F\down)&:= \inf_{B \in \pMatch(E\down, F\down)} \cost(B),
\ \text{and}\\
\distCM'(M\down, N\down)&:= \inf_{(E\down, F\down) \in \Pad(M\down, N\down)}
\distM'(E\down, F\down).
\end{aligned}
$$
\end{remark}

\begin{theorem}
\label{thm:bneck-pseudometric-cpx}
The function $\distCM$ defines an extended pseudometric on $\Cb(\prj \k[\bfP])$.
\end{theorem}

\begin{proof}
First, it holds that $\distCM(M\down, M\down)=0$ for every $M\down$.
Indeed, $B:= (\id_{\smd(M_i)})_{i \in \bbZ}$ is a matching
of $(M\down, M\down)$, and $\cost(B) =0$.

Symmetry holds because any matching $B=(B_i)$ between two padded complexes
induces an inverse matching $B^{-1}=(B_i^{-1})$ of the same cost in the
opposite direction.  Hence
\[
\distCM(M\down, N\down)
=
\distCM(N\down, M\down).
\]

To show the triangle inequality, let $M\down,N\down,O\down \in \Cb(\prj \k[\bfP])$.
We have to verify
\begin{equation}
\label{eq:tri-ineq-cpx}
\distCM(M\down, O\down)
\ \le\
\distCM(M\down, N\down)\ +\ \distCM(N\down, O\down).
\end{equation}
If $\Pad(M\down, N\down) = \varnothing$ or $\Pad(N\down, O\down) = \varnothing$,
then the right hand side is infinity, and the inequality holds.
Therefore, we may assume that 
$\Pad(M\down, N\down) \ne \varnothing$ and $\Pad(N\down, O\down) \ne \varnothing$.
If $\Match(E\down, F\down) = \varnothing$
for all $(E\down, F\down) \in \Pad(M\down, N\down)$, then again
the right hand is infinity.
Hence we may assume that $\Match(E\down, F\down) \ne \varnothing$
for some $(E\down, F\down) \in \Pad(M\down, N\down)$.
Similarly, we may assume that 
$\Match(F'\down, G\down) \ne \varnothing$
for some $(F'\down, G\down) \in \Pad(N\down, O\down)$.
Let
$$
\begin{aligned}
(E\down, F\down) &\in \Pad(M\down, N\down) \text{ with }
B \in \Match(E\down, F\down), \text{and }\\ 
(F'\down, G\down) &\in \Pad(N\down, O\down) \text{ with }
B' \in \Match(F\down, O\down).
\end{aligned}
$$
By definition, both $F\down$ and $F'\down$ have the forms
$$
F\down = C\down \ds N\down,\quad
F'\down = N\down \ds C'\down
$$
for some $C\down$ and $C'\down$ that are finite direct sums of contractible cones.
Set
$$
E^*\down:= E\down \ds C'\down,\, H\down:= C\down \ds P^N\down \ds C'\down,
\text{ and } G^*\down:= C\down \ds G\down.
$$
Then as is easily checked we have
$(B_i \sqcup \id_{\smd(C'_i)})_{i \ge 0} \in \Match(E^*\down, H\down)$
and $(\id_{\smd(C_i)} \sqcup B'_i)_{i \ge 0} \in \Match(H\down, G^*\down)$,
which are illustrated as
$$
\smd(E_i) \sqcup \smd(C'_i) \ya{B_i \sqcup \id_{\smd(C'_i)}}
\smd(F_i) \sqcup \smd(C'_i) = \smd(C_i) \sqcup\smd(F'_i)
\ya{\id_{\smd(C_i)} \sqcup B'_i} \smd(C_i) \sqcup \smd(G_i).
$$

Then since $|E\down| = |N\down| + |C\down|$ and $|G\down| = |N\down| + |C'\down|$,
we have 
$$
|E^*\down| = |H\down| = |G^*\down| = |N\down| + |C\down| + |C'\down|\ \ .
$$
Thus 
$$
(E^*\down,H\down)\in\Pad(M\down, N\down),
\qquad
(H\down,G^*\down)\in\Pad(N\down, O\down).
$$
Applying Lemma~\ref{lem:reg-triangle-cpx} yields
\[
\distM(E^*\down,G^*\down)
\ \le\
\distM(E^*\down,H\down)\ +\ \distM(H\down,G^*\down).
\]
Therefore \eqref{eq:tri-ineq-cpx} holds.
\end{proof}

The proof above verifies the following:

\begin{corollary}
The function $\distCM'$ also defines an extended pseudometric on $\Cb(\prj \k[\bfP])$.
\end{corollary}

\begin{ntn}
\label{ntn:permutation}
For a permutation $\si$ of the set $[n]\ (n \ge 1)$,
we set $\Mat_\si$ to be the permutation matrix
$$
\Mat_\si:= [\de_{i,\si(j)}]_{(i,j) \in [n]^2}.
$$
For a pre-matching $B$ in \eqref{eq:prematching}, if
$\smd(E_i) = (x_1, \dots, x_{d_i})$ and $\smd(F_i) = (y_1, \dots, y_{d_i})$,
then for each $i \in \bbZ$, the bijection $B_i \colon \smd(E_i) \to \smd(F_i))$ is regarded as the permutation $\si_i$
of $(1, \dots, d_i)$ defined by setting $\si_i(j):= k$ if $B_i(x_j) = y_{k}$ with $j, k \in \{1, \dots, d_i\}$.
\end{ntn}

To strengthen Theorem \ref{thm:bneck-pseudometric-cpx} by the proposition below, we first interpret the compatibility condition on differentials in terms of commutative diagram.
The following is easy to verify, and is left to the reader:

\begin{lemma}
\label{lem:matching-comm-cpx}
Let $M\down, N\down \in \Cb(\prj \k[\bfP])$,
$(E\down, F\down) \in \Pad(M\down, N\down)$, and
$B \in \pMatch(E\down, F\down)$.
Then $B \in \Match(E\down, F\down)$ if and only if
the diagram
\begin{equation}
\label{eq:matching-eq-cpx}
\begin{tikzcd}[column sep=40pt]
\Ds_{x^{(i+1)} \in \smd(E_{i+1})}\k x^{(i+1)} & \Ds_{x^{(i)} \in \smd(E_{i})}\k x^{(i)}\\
\Ds_{y^{(i+1)} \in \smd(F_{i+1})} \k y^{(i+1)}& \Ds_{y^{(i)} \in \smd(F_{i})} \k y^{(i)}
\Ar{1-1}{1-2}{"\Mat(\partial^E_i)"}
\Ar{2-1}{2-2}{"\Mat(\partial^F_i)"'}
\Ar{1-1}{2-1}{"\Mat_{B_{i+1}}"'}
\Ar{1-2}{2-2}{"\Mat_{B_{i}}"}
\end{tikzcd}
\end{equation}
is commutative for all $i \in \bbZ$.
\end{lemma}

\begin{proposition}
\label{prp:metric-on-isoclasses-cpx}
For any $M\down, N\down \in \Cb(\prj \k[\bfP])$, the following are equivalent:
\begin{enumerate}[label=(\arabic*)]
\item
$\distCM(M\down, N\down) = 0,$ and
\item 
$M\down \cong N\down$ in $\Kb(\prj \k[\bfP])$.
\end{enumerate}
\end{proposition}

\begin{proof}
 (1) $\Rightarrow$ (2).
 Assume (1) holds.
 Then there exists a pair $(E\down, F\down) \in \Pad(M\down, N\down)$ and some $B \in \Match(E\down, F\down)$ such that $\cost(B) = 0$.
 Thus we have a commutative diagram \eqref{eq:matching-eq-cpx}.
 Here $\cost(B) = 0$ shows that for each $i \in \bbZ$ with
 $\smd(E_i) \ne \varnothing$ and each $x^{(i)} \in \smd(E_{i})$, we have
\begin{equation}
\label{eq:B-fix-cpx}
B(x^{(i)}) = x^{(i)}.
\end{equation}
 Therefore, the commutative diagram \eqref{eq:matching-eq-cpx} yields a commutative diagram
 $$
 \begin{tikzcd}
 E_{i+1} & E_i\\
 F_{i+1} & F_i
 \Ar{1-1}{1-2}{"\partial^E_i"}
 \Ar{2-1}{2-2}{"\partial^F_i"'}
 \Ar{1-1}{2-1}{"\hat{B}_{i+1}"'}
 \Ar{1-2}{2-2}{"\hat{B}_{i}"}
 \end{tikzcd}
 $$
 for all $i \in \bbZ$, where we set
\begin{equation}
\label{eq:hat-B}
\hat{B}_i:= [\de_{y^{(i)}, B_i(x^{(i)})}\,\ro(x^{(i)} \ge B_i(x^{(i)}))]_{(y^{(j)},x^{(i)})},
\end{equation}
each of which turns out to be an isomorphisms by \eqref{eq:B-fix-cpx}.
Thus $E\down \cong F\down$, which shows that
$M\down \ds A\down \cong N\down \ds B\down$ in $\Cb(\prj \k[\bfP])$
for some acyclic complexes $A\down, B\down$ by definition of
$\Pad(M\down, N\down)$.
Hence (2) holds.

(2) $\Rightarrow$ (1).
If (2) holds, then by Proposition \ref{prp:htp-eq=upto-acyc},
$M\down \ds A\down \cong N\down \ds B\down$ in $\Cb(\prj \k[\bfP])$
for some finite direct sums $A\down, B\down$ of contractible cones.
By setting $E\down:= M\down \ds A\down$,
we have $(E\down, E\down) \in \Pad(M\down, N\down)$.
Since $\distM(E\down, E\down) = 0$, (1) holds.
\end{proof}

By combining Theorem \ref{thm:bneck-pseudometric-cpx} and
Proposition \ref{prp:metric-on-isoclasses-cpx}, we have
the following.

\begin{theorem}
\label{thm:metric_homotopy}
The function $\distCM$ defines an extended metric on $\Kb(\prj \k[\bfP])$.
\end{theorem}

\begin{remark}
Since the global dimension of $\k[\bfP]$ is finite, the fully faithful embedding $\Kb(\prj \k[\bfP])$ $\to \Db(\vect^\bfP)$ becomes an equivalence
with a quasi-inverse $P\down \colon \calD(\vect^\bfP) \to \Kb(\prj \k[\bfP])$,
through which $\distCM$ defines an extended metric on the derived category $\Db(\vect^\bfP)$ by setting $\distCM(M, N):= \distCM(P\down(M\down), P\down(N\down))$
for all $M\down, N\down \in \Db(\vect^\bfP)$.
\end{remark}

\begin{remark}
\label{rmk:charact-matching}
Let $M\down, N\down \in \Cb(\prj \k[\bfP])$, and let
$(E\down, F\down) \in \Pad(M\down, N\down)$.
To understand when $\Match(E\down, F\down) \ne \varnothing$, it is useful
to record only the coefficient matrices of the differentials.

For simplicity, we identify $\prj \k[\bfP]$ with its full subcategory
$\add \k[\bfP]$ consisting of all finite direct sums of the representables
$\k[\bfP]_x$ with $x \in \bfP$, as above.
For any object
\[
P=\Ds_{i=1}^n \k[\bfP]_{x_i} \in \prj \k[\bfP]
\qquad (n\in\bbN,\ (x_1,\dots,x_n)\in \bfP^n),
\]
we set
\[
\Mat(P):=\Ds_{i=1}^n \k x_i \in \vect.
\]
For any morphism $\al \colon P \to Q$ in $\prj \k[\bfP]$, we let
$\Mat(\al)$ be its coefficient matrix as in Lemma~\ref{lem:mor-betw-proj}.
In this way we obtain a functor
\[
\Mat \colon \prj \k[\bfP] \to \vect.
\]

For each complex
$M\down=(M_i,\partial^M_i)_{i\in\bbZ} \in \Cb(\prj \k[\bfP])$,
we set
\[
\Mat(M\down):=(\Mat(M_i),\Mat(\partial^M_i))_{i\in\bbZ} \in \Cb(\vect).
\]
If $B=(B_i)_{i\in\bbZ}\in \pMatch(E\down,F\down)$,
then each bijection
$B_i \colon \smd(E_i)\to \smd(F_i)$ induces a linear isomorphism
\[
\Mat_{B_i}\colon \Mat(E_i)\to \Mat(F_i),
\]
and by Lemma~\ref{lem:matching-comm-cpx}, 
\[
\Mat_B:=(\Mat_{B_i})_{i\in\bbZ}\colon \Mat(E\down)\to \Mat(F\down)
\]
turns out to be a morphism of complexes of vector spaces
if and only if $B \in \Match(E\down, F\down)$.

Namely, Lemma~\ref{lem:matching-comm-cpx} can be rephrased as the equivalence
of the following statements:
\begin{enumerate}[label=(\arabic*)]
\item
$\Match(E\down, F\down) \ne \varnothing$.
\item
There exists a $B \in \pMatch(E\down, F\down)$ such that
$\Mat_B \colon \Mat(E\down) \to \Mat(F\down)$
is an isomorphism in $\Cb(\vect)$.
\end{enumerate}

In particular, the nonemptiness of $\Match(E\down,F\down)$ is controlled
by the coefficient matrices of the differentials, rather than by homology
alone. Note also that we do not require
$x^{(i)} \ge B_i(x^{(i)})$ for each $i \in \bbZ$ and
$x^{(i)} \in \smd(E_i)$.
Namely, $\hat{B}_i$ in \eqref{eq:hat-B} are not always assumed to be morphisms in $\prj \k[\bfP]$.
\end{remark}

\begin{remark}
By the characterization of matchings in Remark~\ref{rmk:charact-matching},
the idea of the distance $\hat{d}^\infty_\calI(M,N)$ for
$M,N\in\vect^{\bbR^n}$ introduced by Bjerkevik--Lesnick
\cite[Section~3.1]{bjerkevik2021ell} is very close to that of the
complex matching distance. If one adapts their construction to a finite
poset $\bfP$, then for $M,N\in\vect^\bfP$ one obtains
\[
\hat{d}^\infty_\calI(M,N)=\distCM(P(M),P(N)),
\]
where $P(M)$ and $P(N)$ denote minimal projective presentations of
$M$ and $N$, regarded as complexes concentrated in degrees $0$ and $1$.

The difference lies mainly in how the matching is represented.
Given $(E\down,F\down)\in\Pad(P(M),P(N))$, and assuming that \eqref{eq:matching-eq-cpx} is commutative for $i = 0$,
Bjerkevik--Lesnick absorb the matching into the choice of orderings of $\smd(E_i)$ and $\smd(F_i)$, making the corresponding bijections $B_0, B_1$ the identity permutations in the sense of Notation \ref{ntn:permutation}.
In contrast, in this paper we keep the summand sets $\smd(E_i)$ and $\smd(F_i)$ fixed and record the matching explicitly via the bijections $B_0, B_1$.
\end{remark}

\subsection{Complex matching distance of minimal projective resolutions}

In this subsection, we work over the full subcategory $\Cnn(\prj \k[\bfP])$ of $\Cb(\prj \k[\bfP])$ consisting of \emph{non-negative complexes}, that is, objects $C\down$ in $\Cb(\prj \k[\bfP])$ such that $C_i = 0$ for all $i < 0$, which are sometimes written $C\down = (C_i, \partial^C_i)_{i \ge 0}$.
For each $M \in \vect^\bfP$, any projective resolution of $M$ is regarded as an object in $\Cnn(\prj \k[\bfP])$.  We denote by $P^M\down$ a minimal projective resolution of $M$,
which is unique up to isomorphism of complexes.

A contractible cone $\Cone(\id_E)[a]$ with $E \in \prj \k[\bfP]$ and $a \in \bbZ$ is said to be \emph{non-negative} if $a \ge 0$.
Let $C\down \in \Cnn(\prj \k[\bfP])$.
When we work over $\Cnn(\prj \k[\bfP])$, \emph{padding of $C\down$ by contractible cones} means padding of $C\down$ by non-negative contractive cones.
Parallel to Notation \ref{ntn:Pad}, we use the following.

\begin{ntn}
\label{ntn:Padnn}
Let $M\down, N\down \in \Cnn(\prj \k[\bfP])$.
By $\Padnn(M\down)$, we denote the set of all paddings of $M\down$ by non-negative contratible cones, and set
$$
\Padnn(M\down, N\down):= \{(E\down, F\down) \in \Padnn(M\down) \times \Padnn(N\down) \mid |E\down| = |F\down|\}.
$$
\end{ntn}

Since $\k[\bfP]$ has a finite global dimension, 
we have the following.

\begin{lemma}
\label{lem:all-prj-resol}
Let $M \in \vect^\bfP$.
Then the set of all projective resolutions of $M$ coincides with
$\Padnn(P^M\down)$.
\end{lemma}

\begin{remark}
\label{rmk:mobius-betti}
For $b\in\bfP$, let $1_b$ denote the interval $\bfP$-module
$V_{\{b\}}$ of the singleton
$\{b\}$ (a spread in the sense of~\cite{ElchesenPatel}), and let $M\in\vect^\bfP$.

In the language of~\cite{ElchesenPatel}, the M\"obius cohomology of $M$ at $b$
is computed by the Ext–groups
\[
\Ext^d_{\vect^\bfP}(1_b,M),
\]
and the M\"obius homology groups $H^\downarrow_d M(b)$ of~\cite{patel2023mobius_homology}
are, for $\vect$, canonically dual:
\[
H^\downarrow_d M(b)\ \cong\ \Ext^d_{\vect^\bfP}(1_b,M)^\vee,
\]
where $(\blank)^\vee$ denotes $\k$-linear dual.
In particular, if $P\down^M\to M$ is a minimal projective resolution, then the
Ext-groups (and hence the M\"obius homology at $b$) are functorially determined
by $P\down^M$ via the standard computation
\[
\Ext^d_{\vect^\bfP}(1_b,M)
\ \cong\
H^d\!\bigl(\Hom_{\vect^\bfP}(P\down^{1_b},P\down^M)\bigr),
\]
for any projective resolution $P\down^{1_b}\to 1_b$.
Thus, up to duality, the M\"obius homology of $M$ is encoded in its minimal
projective resolution.

We will not use this identification in this paper, but it provides a conceptual
bridge between the M\"obius homology of~\cite{patel2023mobius_homology} and the
projective–resolution viewpoint developed in this paper; full details in the
(cohomological) setting can be found in~\cite{ElchesenPatel}.
For a detailed discussion connecting M\"obius homology to minimal projective resolutions, see~\cite[Appendix A]{bauer2026algebraic} 
\end{remark}

We will replace $\Pad$ by $\Padnn$, and see that all statements on $\Cb(\prj \k[\bfP])$ and $\Kb(\prj \k[\bfP])$ remain valid on $\Cnn(\prj \k[\bfP])$ and $\Knn(\prj \k[\bfP])$, respectively.
In this way, we will restrict the metric $\distCM$ on $\Kb(\prj \k[\bfP])$ to $\Knn(\prj \k[\bfP])$ and to its subset $\{P^M\down \mid M \in \vect^\bfP\}$.

\begin{proposition}
\label{prop:common-padding}
Let $M\down, N\down \in \Cnn(\prj \k[\bfP])$.
Then they have the same alternating sum
\[
\sum_{i\ge 0}(-1)^i|M_i|\ =\ \sum_{i\ge 0}(-1)^i|N_i|
\]
if and only if $\Padnn(M\down, N\down)\neq \varnothing$.
\end{proposition}

\begin{proof}
Let $p=|M\down|$, $q=|N\down|$, $e_i:= (\de_{j,i})_{j \in \bbN} \in \bbZ^{(\bbN)}$, and $e^{(i)}:= e_i + e_{i+1}$ for all $i \in \bbN$.
Then $|\Cone(\id_{\k[\bfP]_x})[i]| = e^{(i)}$ for all $x \in \bfP$ and $i \in \bbN$.
We set $\al \colon \bbZ^{(\bbN)} \to \bbZ$ to be
the alternating-sum map: $\alpha(r)=\sum_i(-1)^ir_i$\ ($r \in \bbZ^{(\bbN)}$).
Then as is easily seen, the kernel of $\al$ is generated by the set
$\{e^{(i)} \mid i \in \bbN\}$ as an abelian group:
$\Ker \al = \ang{e^{(i)} \mid i \in \bbN}$ (The proof is similar to and easier than
that of Lemma \ref{lem:kernel-hat-alpha}, and we omit it).

($\Rightarrow$).
Assume $\alpha(p)=\alpha(q)$.
Then we have $p-q \in \Ker \al = \ang{e^{(i)} \mid i \in \bbN}$, and hence
$$
p-q = \sum_{i\in \bbN}k_i e^{(i)} =\sum_{i \in S_+}k_ie^{(i)} - \sum_{i \in S_-}|k_i|e^{(i)},\quad\text{and thus}\quad
p+\sum_{i \in S_-}|k_i|e^{(i)} = q+ \sum_{i \in S_+}k_ie^{(i)}
$$
for some $(k_i)_{i\in \bbN} \in \bbZ^{(\bbN)}$, where $S_+:= \{i \in \bbN \mid k_i \ge 0\},\, S_-:= \{i \in \bbN \mid k_i < 0\}$.
Take any $x \in \bfP$, and set $E:= \k[\bfP]_x$.
Then the last equality above shows that
$$
\left|M\down \ds \left(\Ds_{i \in S_-}\Cone(\id_E)[i]\right)^{(k_i)}\right|
= \left|N\down \ds \left(\Ds_{i \in S_+}\Cone(\id_E)[i]\right)^{(k_i)}\right|.
$$

($\Leftarrow$).
The argument above can be reversed.
\end{proof}

Definition \ref{dfn:matching-cpx} restricts to the following.

\begin{definition}
\label{dfn:matching}
Let $M, N \in \vect^\bfP$,
$(E\down,F\down)\in\Padnn(P\down^M,P\down^N)$, and set
$$
\Mat(\partial^E_i) = {}^t[a_{x^{(i+1)}, x^{(i)}}^{(i)}]_{(x^{(i+1)}, x^{(i)})\, \in \smd(E_{i+1}) \times \smd(E_{i})},
\quad
\Mat(\partial^F_i) = {}^t[b_{y^{(i+1)}, y^{(i)}}^{(i)}]_{(y^{(i+1)}, y^{()})\, \in \smd(F_{i+1}) \times \smd(F_{i})}
$$
for all $i \ge 0$.
Then a \emph{pre-matching} of $(E\down, F\down)$ is a family of degreewise bijections
\[
B=(B_i)_{i\ge 0},\quad B_i\in\Bij(\smd(E_i),\,\smd(F_i)).
\]
A \emph{matching} of $(E\down, F\down)$ is a pre-matching of $(E\down, F\down)$
satisfying the \emph{compatibility condition on differentials}:
$$
a^{(i)}_{x^{(i+1)},x^{(i)}} = b^{(i)}_{B_{i+1}(x^{(i+1)}), B_i(x^{(i)})}
$$
for all $i \ge 0$, $x^{(i)} \in \smd(E_i), x^{(i+1)} \in \smd(E_{i+1})$.
Note that this condition implies that if $a^{(i)}_{x^{(i+1)},x^{(i)}} \ne 0$, then $B_{i+1}(x^{(i+1)}) \ge B_i(x^{(i)})$.

By $\Match(E\down, F\down)$ (resp.\ $\pMatch(E\down, F\down)$),
we denote the set of matchings (resp.\ pre-matchings) of $(E\down, F\down)$.

Define the \emph{cost} of a pre-matching $B$ as follows:
\[
\cost(B)\ :=\
\sup\bigl\{\,\dist\bigl(U,\ B_i(U)\bigr)\ \bigm|\ i\ge 0,\ U \in \Summands(E_i)\,\bigr\} = \sup\{d_\bfP(x, B_i(x)) \mid i \ge 0, x \in \smd(E_i)\}.
\]
(Equivalently, $\cost(B)=\sup\{\dist(B_i^{-1}(V),V)\mid i\ge 0,\ V\in\Summands(F_i)\}$.)  
We refer to the quantity
\[
\distM(E\down,F\down)\ :=\ \inf_{\,B\in\Match(E\down,F\down)}\ \cost(B)
\]
as the \emph{matching distance} for resolutions with a fixed size vector.
It is defined to be infinity if $\Match(E\down, F\down) = \varnothing$.
\end{definition}

The same argument as in the proof of Lemma \ref{lem:reg-triangle-cpx}
shows the following.

\begin{lemma}\label{lem:reg-triangle}
Let $M, N, O \in \vect^\bfP$, and 
$(E\down, F\down) \in \Padnn(P^M\down, P^N\down)$,
$(F\down, G\down) \in \Padnn(P^N\down, P^O\down)$.
Then
\[
\distM(E\down,G\down)\ \le\ \distM(E\down,F\down)\ +\ \distM(F\down,G\down).
\]
\end{lemma}

Definition \ref{dfn:CM-dist-cpx} restricts to the following.

\begin{definition}
For minimal projective resolutions $P\down^M$ and $P\down^N$ (not necessarily of the same size), the \emph{complex matching distance} is
\[
\distCM\bigl(P\down^M,P\down^N\bigr)
\ :=\ 
\inf_{\ (E\down,F\down)\in\Padnn(P\down^M,P\down^N)}\ \distM(E\down,F\down).
\]
We adopt the extended-value convention that $\distCM\bigl(P\down^M,P\down^N\bigr)=\infty$ if there is no compatible padding (i.e.\ $\Padnn(P\down^M,P\down^N)=\varnothing$).
\end{definition}

\begin{remark}
Let $M, N \in \vect^\bfP$, and $(E\down, F\down) \in \Padnn(P^M\down, P^N\down)$.
Then by using $\pMatch(E\down, F\down)$ instead of $\Match(E\down, F\down)$,
it is possible to define a distance $\distM'(E\down, F\down)$ between
$E\down$ and $F\down$, and a distance $\distCM'(P^M\down, P^N\down)$ between
$P^M\down$ and $P^N\down$, and we can prove the parallel statements for these
distances.
However, these distances are very coarse as Example \ref{exm:Luis} below shows.
Note that $\distCM'(P^M\down, P^N\down) = 0$ if and only if there exist
a pair $(E\down, F\down) \in \Padnn(P^M\down, P^N\down)$ with
the property that $\cost(B) = 0$ for some $B \in \pMatch(E\down, F\down)$.
\end{remark}

Proposition~\ref{prop:common-padding} gives a sufficient condition for finiteness (equality of alternating sums) of $\distCM'$; in general,
$\distCM'$ may be infinite.
However, as will be shown in Corollary \ref{cor:Res-nonempty},
$\Padnn(P\down^M,P\down^N)$ has some pair $(E\down, F\down)$
such that $\Match(E\down, F\down) \ne \varnothing$ (and hence $\distCM(P\down^M,P\down^N) < \infty$)
whenever there exists a Galois coupling of $(M, N)$.

\begin{example}
\label{exm:noniso-distB-zero}
There exists a finite poset $\bfP$ and $\bfP$-modules $M$ and $N$ such that
$\distCM'(P\down^M,P\down^N) = 0$,
but $P^M\down \not\cong P^N\down$.
For example, let $\bfP:= \{1 < 2\}$, $M:= V_{\{1\}} \ds V_{\{2\}}$, and
$N:= V_\bfP = \k[\bfP]_1$.
Then $P^M\down$ and $F\down:= P^N\down \ds \Cone(\id_{\k[\bfP]_2})$
(as nonnegative complexes) are given as follows:
$$
\begin{aligned}
P^M\down &= (\dots \to 0 \to \k[P]_2 \xrightarrow{\bsmat{\ro(2\ge 1)\\0}} \k[\bfP]_1 \ds \k[\bfP]_2), \text{and}\\
F\down&= (\dots \to 0 \to \k[P]_2 \xrightarrow{\bsmat{0\\\id}} \k[\bfP]_1 \ds \k[\bfP]_2).
\end{aligned}
$$
Hence $(P^M\down,\, F\down) \in \Padnn(P^M\down, P^N\down)$.
Define $B \in \pMatch(P^M\down,\, F\down)$ by the following table:
$$
\begin{array}{c|c|c}
\deg & 1 & 0 \\
\hline
P^M\down & \kP{2} & \kP{1}\ \ \kP{2}\\
F\down   & \kP{2} & \kP{1}\ \ \kP{2}\\
\hline
\dist & 0 & 0 \ \ \ \ \ \ \ \ 0
\end{array}.
$$
Then we have  $\cost(B) = 0$, and hence $\distCM'(P^M\down, P^N\down) = 0$.
However, it is clear that $P^M\down \not\cong P^N\down = (\dots\to 0 \to  \k[\bfP]_1)$.
Note that $B$ above does not satisfy the compatibility on differentials.
\end{example}

We next give a criterion when $\distCM'(P^M\down, P^N\down) = 0$,
which gives a necessary condition
for $\distCM(P^M\down, P^N\down) = 0$.
We denote by $K_{\prj}(\bfP)$ the (split) Grothendieck group of the
category of projective $\bfP$-modules.
For each projective $\bfP$-module $P$, we denote by $[P]$ the element
of $K_{\prj}(\bfP)$ containing the isomorphism class of $P$.
Then the set $\{[\k[\bfP]_x] \mid x \in \bfP\}$ forms a basis of
$K_{\prj}(\bfP)$.
For each nonnegative bounded complex $E\down = (E_i, \partial_i)_{i \in \bbN}$
of projective $\bfP$-modules, we set
$[E\down]:= ([E_i])_{i\in \bbN} \in K_{\prj}(\bfP)^{(\bbN)}$.
For any $A, B \in K_{\prj}(\bfP)^{(\bbN)}$, we write
$A \equiv B$ if $A - B \in 
\ang{[\Cone(\id_{\k[\bfP]_x})[i]] \mid x \in \bfP,\, i \in \bbN}$.

\begin{lemma}
\label{lem:kernel-hat-alpha}
Define a group homomorphism $\hat{\al} \colon K_{\prj}(\bfP)^{(\bbN)} \to K_{\prj}(\bfP)$ by $\hat{\al}((A_i)_{i \in \bbN}):= \sum_{i\in \bbN}(-1)^iA_i$
for all $(A_i)_{i \in \bbN} \in K_{\prj}(\bfP)^{(\bbN)}$.
Then $\Ker \hat{\al} = \ang{[\Cone(\id_{\k[\bfP]_x})[i]] \mid x \in \bfP,\, i \in \bbN}$.
Therefore, $A \equiv B$ if and only if $\hat{\al}(A) = \hat{\al}(B)$.
\end{lemma}

\begin{proof}
$(\supseteq)$.
Let $x \in \bfP$ and $i \in \bbN$.
We set $e_{i,x}:= [\Cone(\id_{\k[\bfP]_x})[i]]$ for short.
Then
$$
\hat{\al}(e_{i,x}) = \hat{\al}(\dots, 0, \overset{\smat{(\deg\,i+1)\\\ }}{[\k[\bfP]_x]}, \overset{\smat{(\deg\,i)\\\ }}{[\k[\bfP]_x]}, 0, \dots, 0)
= (-1)^{i+1}[\k[\bfP]_x] + (-1)^{i}[\k[\bfP]_x] = 0.
$$
Hence $\Ker \hat{\al} \supseteq \ang{e_{i,x} \mid x \in \bfP, i \in \bbN}$.
Set the right hand side of this to be $K$.

$(\subseteq)$.
Let $A \in K_{\prj}(\bfP)^{(\bbN)}$.
We define its \emph{width} $w(A)$ by
$$
w(A):=
\begin{cases}
0 & \text{if $A_i = 0$  for all } i \in \bbN,\\
\max\supp(A) - \min\supp(A) + 1& \text{otherwise},
\end{cases}
$$
where $\supp(A):= \{i \in \bbN \mid A_i \ne 0\}$.
We show that if $A \in \Ker \hat{\al}$, then $A \in K$ by induction on $w(A)$.

When $w(A) = 0, 1$, then obviously $A = 0 \in K$.
Therefore, assume $w(A) \ge 2$.
By definition of width, there exist integers $m,n$ with $n > m \ge 0$ such that
$$
A = (\dots, 0, A_n,\dots, A_m, 0, \dots, 0)\quad\text{with } A_m, A_n \ne 0.
$$
Using the basis of $K_{\prj}(\bfP)$ given above, we can write
$$
A_m = \sum_{x \in \bfP} a_{x} [\k[\bfP]_x] = \sum_{x \in S} a_{x} [\k[\bfP]_x]
$$
for some $(a_{x})_{x\in \bfP} \in \bbZ^{(\bfP)}$, where $S:= \{x \in \bfP \mid a_{x} \ne 0\}$
is a finite set.
Define $B \in K_{\prj}(\bfP)^{(\bbN)}$ by
$$
B:= A - \sum_{x \in S}a_{x}e_{m,x} = 
(\dots, 0, A_n, \dots, A_{m+1} - \sum_{x \in S} a_{x}[\k[\bfP]_x],0,0,\dots,0).
$$
Then $w(B) \le w(A) -1$.
Moreover, we have $\hat{\al}(B) = 0$ by the linearity of $\hat{\al}$ and $(\supseteq)$ above.
Thus $B \in \Ker \hat{\al}$.
Therefore, by induction hypothesis, we have $B \in K$, and hence
$A = B + \sum_{x \in S}a_{x}e_{m,x} \in K$.
\end{proof}

\begin{lemma}
\label{lem:distB=zero}
(1) Let $E\down, F\down$ be nonnegative bounded complexes of projective $\bfP$-modules.
Then the following are equivalent:
\begin{enumerate}[label=(\alph*)]
\item 
$\distM'(E\down, F\down) = 0$,
\item 
There exists some $B \in \pMatch(E\down, F\down)$ such that $\cost(B) = 0$,
\item 
$[E\down] = [F\down]$.
\end{enumerate}

(2) Let $M, N$ be $\bfP$-modules.
Then the following are equivalent:
\begin{enumerate}[label=(\alph*)]
\item 
$\distCM'(P^M\down, P^N\down) = 0$,
\item 
$\Padnn(P^M\down, P^N\down)$ has a pair $(E\down, F\down)$
such that $\cost(B) = 0$ for some $B \in \pMatch(E\down, F\down)$,
\item 
$[P^M\down] \equiv [P^N\down]$,
\item 
$\hat{\al}([P^M\down]) = \hat{\al}([P^N\down])$.
\end{enumerate}
In particular, if $\distCM(P^M\down, P^N\down) = 0$, then
$\hat{\al}([P^M\down]) = \hat{\al}([P^N\down])$.
\end{lemma}

\begin{proof}
(1) is clear by definition.

(2) follows from (1) and Lemmas \ref{lem:all-prj-resol} and \ref{lem:kernel-hat-alpha}.
\end{proof}

By the same arguments as in the proofs of Theorem \ref{thm:bneck-pseudometric-cpx} and Proposition \ref{prp:metric-on-isoclasses-cpx},
we obtain the following three statements.

\begin{theorem}
\label{thm:bneck-pseudometric}
The function $\distCM$ defines an extended pseudometric on the set of
minimal projective resolutions $\{P\down^M \mid M\in\vect^{\bfP}\}$.
\end{theorem}

\begin{corollary}
The function $\distCM'$ defines an extended pseudometric on the set of
minimal projective resolutions $\{P\down^M \mid M\in\vect^{\bfP}\}$.
\end{corollary}

\begin{proposition}
\label{prp:metric-on-isoclasses}
For any $M, N \in \vect^\bfP$, the following are equivalent:
\begin{enumerate}[label=(\arabic*)]
\item
$\distCM(P^M\down, P^N\down) = 0,$
\item 
$P^M\down \cong P^N\down$ as complexes, and
\item 
$M \cong N$ in $\vect^\bfP$.
\end{enumerate}
\end{proposition}

By combining Theorem \ref{thm:bneck-pseudometric} and
Proposition \ref{prp:metric-on-isoclasses}, we have
the following.

\begin{theorem}
\label{thm:extended_metric_projectives}
The function $\distCM$ defines an extended metric on the set of isomorphism classes of
minimal projective resolutions $\{P\down^M \mid M\in\vect^{\bfP}\}/\!\cong$.
\end{theorem}

\subsection{Examples}

We now compute complex matching distances for our 1D and 2D running examples by
comparing minimal projective resolutions and equalizing degreewise sizes
via contractible cones.

\begin{example}\label{ex:bneck-chain}
Let $\bfP=\{1<2<3<4\}$ and consider
\[
M=V_{[1,2)}\oplus V_{[2,4)},\qquad N=V_{[2,4)}
\]
as in Example~\ref{ex:gtd-chain}.
Minimal projective resolutions (as nonnegative complexes)
of their summands are given as follows:
\[
\begin{aligned}
P\down^{V_{[1,2)}}& = (\cdots \to 0 \to \kP{2} \ya{\ro(2\ge 1)} \kP{1}),\\ P\down^{V_{[2,4)}}& = (\cdots \to 0 \to \kP4 \ya{\ro(4 \ge 2)} \kP2).
\end{aligned}
\]
Summing yields
\[
\begin{array}{c|c|c}
\deg & 1 & 0 \\
\hline
P\down^{M} & \kP2 \ds \kP4 & \kP1 \ds \kP2\\
P\down^{N} & \kP4 & \kP2
\end{array}
\]
with $|P\down^M|=(\dots,0,2,2)$ and $|P\down^N|=(\dots,0,1,1)$. Since the alternating sums agree, padding is possible.
Pad $P\down^{N}$ by $\Cone(\id_{\kP1})$
to have
$F\down:= P^N\down \ds \Cone(\id_{\kP1})$,
adding $\kP1$ in both degrees $1$ and $0$, which matches the size vector of $P\down^{M}$.
Define a matching $B \in \Match(P^M\down, F\down)$
by the following table:
\[
\begin{array}{c|cc|cc}
\deg & \multicolumn{2}{c|}1 & \multicolumn{2}{c}0 \\
\hline
P\down^{M} & \kP2 & \kP4 & \kP1 & \kP2\\
\hline
F\down & \kP1 & \kP4 & \kP1 & \kP2 \\
\hline
\dist & 1 & 0 & 0 & 0
\end{array}\ ,
\]
where since we have
$$
\Mat(\partial^{P^M}_0) = \bsmat{1&0\\0&1},\text{ and }\ 
\Mat(\partial^F_0) = \bsmat{1&0\\0&1}
$$
under the orders of summands as in the table,
$B$ satisfies the compatibility on differentials.
Thus $\distCM(P\down^{M},P\down^{N}) \le 1$.
Since $\hat{\al}(P^M\down) \ne \hat{\al}(P^N\down)$, it holds
by Lemma \ref{lem:distB=zero} (2) that
\[
\distCM(P\down^{M},P\down^{N})=1.
\]
\end{example}

\begin{example}
\label{ex:bneck-2d}
Let $\bfP=\{1,2,3\}^2$ with the product order and $L^\infty$ metric, and
consider the modules $M,N\in\vect^{\bfP}$ from Example~\ref{ex:gtd-2d}.

As is easily seen (e.g., see \citealp[Proposition 41]{MR4402576}),
$M_1,\, M_2$ and $N$ have the following
minimal projective resolutions (as nonnegative complexes):
$$
\begin{aligned}
P^{M_1}\down &= (\cdots \to 0 \to \k[\bfP]_{23} \ya{\bsmat{\ro(23\ge 22)\\-\ro(23\ge 13)}}
\k[\bfP]_{22} \ds  \k[\bfP]_{13} \ya{[\ro(22 \ge 12)\ \ro(13\ge 12)]} \k[\bfP]_{12})\\
P^{M_2}\down &= (\cdots \to 0 \ya{\hspace{2em}} 0 \ya{\hspace{3em}} \k[\bfP]_{22} \ya{\ro(22\ge 21)} \k[\bfP]_{21})\\
P^N\down &= (\cdots \to 0 \ya{\hspace{2em}} 0 \ya{\hspace{3em}} \k[\bfP]_{33} 
\ya{\ro(33\ge 22)} \k[\bfP]_{22} )\\
\end{aligned}
$$
with $|P\down^{M}|= (\dots,0,1,3,2), |P\down^{N}| = (\dots,0,0,1,1)$.
Padding $P^N\down$ by
$C\down:= \Cone(\id_{\k[\bfP]_{12}}) \ds \Cone(\id_{\k[\bfP]_{22}})[1]$
to have $F\down = P^N\down \ds C\down$,
we can define a pre-matching $B \in \pMatch(P^M\down, F\down)$
by the following table:
$$
\begin{array}{c|c|ccc|cc}
\deg& 2 &\multicolumn{3}{c|}1& \multicolumn{2}{c}0\\
\hline
P^{M_1}\down&\kP{23}&\kP{22}&\kP{13}& & \kP{12}&\\
P^{M_2}\down&& & & \kP{22}& &\kP{21}\\
\hline
P^{N}\down&& & &  \kP{33}& &\kP{22}\\
C\down&\kP{22}&\kP{22}&\kP{12}& & \kP{12}&\\
\hline
\dist & 1 & 0 & 1 & 1 & 0 & 1
\end{array}\ ,
$$
where the coefficient matrices of differentials under the above orders of
summands are given by the following table:
$$
\begin{array}{c|c|c}
 & \partial_1 & \partial_0 \\[5pt]
 \hline
 P^M\down & \bsmat{1\\-1\\0} & \bsmat{1 & 1 & 0\\0 & 0& 1}  \\[5pt]
 \hline
 F\down & \bsmat{1\\0\\0}  & \bsmat{0 & 1 & 0\\0 & 0& 1}
 \end{array}
$$
and hence $B$ does not satisfy the compatibility condition on differentials
for $(P^M\down, F\down)$.
However, we have the following commutative diagram,
where the vertical morphisms are isomorphisms (we omit $\ro(\cdots)$ parts in the entries of matrices for simplicity, namely express them by their coefficient matrices):
$$
\begin{tikzcd}[ampersand replacement=\&]
\makebox[3em][r]{$F'\down:\quad\kP{23}$} \& \kP{22} \ds \kP{13} \ds \kP{22} \& \kP{12} \ds \kP{21}\\
\makebox[3em][r]{$F\down:\quad\kP{23}$} \& \kP{22} \ds \kP{13} \ds \kP{22} \& \kP{12} \ds \kP{21}
\Ar{1-1}{1-2}{"{\bsmat{1\\-1\\0}}"}
\Ar{1-2}{1-3}{"{\bsmat{1&1&0\\0&0&1}}"}
\Ar{2-1}{2-2}{"{\bsmat{1\\0\\0}}"}
\Ar{2-2}{2-3}{"{\bsmat{0&1&0\\0&0&1}}"}
\Ar{1-1}{2-1}{equal}
\Ar{1-2}{2-2}{"{\bsmat{1&0&0\\1&1&0\\0&0&1}}"}
\Ar{1-3}{2-3}{equal}
\end{tikzcd}.
$$
Define a complex $F'\down$ by the upper row in the diagram.
Then $\pMatch(P^M\down, F\down) = \pMatch(P^M\down, F'\down)$,
and since $F'\down \cong F\down$, we have $(P^M\down, F'\down) \in \Padnn(P^M\down, P^N\down)$, and the $B$ above
satisfies the compatibility on differentials for $(P^M\down, F'\down)$,
and hence $B \in \Match(P^M\down, F'\down)$, and $\cost(B) = 1$.
Therefore $\distCM(P\down^{M},P\down^{N}) \le \cost(B) = 1$.  Since $\hat{\al}(P^M\down) \ne \hat{\al}(P^N\down)$, it holds
by Lemma \ref{lem:distB=zero} (2) that
\[
\distCM(P\down^{M},P\down^{N})=1.
\]
\end{example}

The following is a slightly simplified version of an example reported by
Luis Scoccola to the first version of the preprint.

\begin{example}
\label{exm:Luis}
Let $\bfP:= \{1 < 2 < 3 < 4 < 5\}$, and let $M:= V_{[2,4]}$ and $N:= 0$.
Then
\begin{equation}
\label{eq:Luis=1}
\distCM'(P^M\down, P^N\down) = 1.
\end{equation}
Indeed, we have
$$
\begin{aligned}
P^M\down &= (\dots \to 0 \to \kP{5} \ya{\ro(5\ge 2)} \kP{2}),\\
P^N\down &= (\dots \to 0 \to 0).
\end{aligned}
$$
Let $E\down:= P^M\down \ds \Cone(\id_{\kP{3}}) \ds \Cone(\id_{\kP{4}})$ and
$F\down:= \Cone(\id_{\kP{2}}) \ds \Cone(\id_{\kP{3}}) \ds \Cone(\id_{\kP{4}})$.
Then we can define $B \in \pMatch(E\down, F\down)$ by the following table:
$$
\begin{array}{c|c|c}
\deg & 1 & 0\\
\hline
E\down & \kP{3} \ \kP{4}\ \kP{5} & \kP{2}\ \kP{3} \ \kP{4}\\
F\down & \kP{2}\ \kP{3} \ \kP{4} & \kP{2}\ \kP{3} \ \kP{4}\\
\hline
\dist & 1\hspace{2em} 1\hspace{2em} 1 & 0\hspace{2em} 0\hspace{2em} 0
\end{array}.
$$
Hence $\distCM'(P^M\down, P^N\down) \le \cost(B) = 1$.
Since $\hat{\al}(P^M\down) \ne \hat{\al}(P^N\down)$, \eqref{eq:Luis=1} holds
by Lemma \ref{lem:distB=zero} (2).

In this case, under these orders of summands, we have
$$
\Mat(\partial^E_0) = \bsmat{0&0&1\\1&0&0\\0&1&0},\quad
\Mat(\partial^F_0) = \bsmat{1&0&0\\0&1&0\\0&0&1},
$$
and hence $B$ above does not satisfy the compatibility of differentials.
If we define $B' \in \pMatch(E\down, F\down)$ by the following table:
$$
\begin{array}{c|c|c}
\deg & 1 & 0\\
\hline
E\down & \kP{5}\ \kP{3} \ \kP{4} & \kP{2}\ \kP{3} \ \kP{4}\\
F\down & \kP{2}\ \kP{3} \ \kP{4} & \kP{2}\ \kP{3} \ \kP{4}\\
\hline
\dist & 3\hspace{2em} 0\hspace{2em} 0 & 0\hspace{2em} 0\hspace{2em} 0
\end{array},
$$
then $B' \in \Match(E\down, F\down)$, and $\cost(B') = 3$.

Similarly, for $\bfP = \{1 < 2 < \cdots < n\}$ with $n \ge 2$,
we see that $\distCM'(P^M\down, P^N\down) = 1$ for 
$M = V_{[a,b]}$ and $N = 0$ for all $1 \le a \le b < n$.

\end{example}

\section{Stability Theorem}
\label{sec:stability}

We now relate the two distances defined above. Informally: a Galois coupling of $M$ and $N$ controls, via restriction, a pair of projective resolutions whose degreewise summands can be matched with cost bounded by the coupling cost. Hence the complex matching distance between minimal projective resolutions is at most the G\"ulen--McCleary distance.

The next lemma says that pulling back along the right adjoint of a Galois connection sends the indecomposable projective at $x\in\bfQ$ to the indecomposable projective at $f(x)\in\bfP$.

\begin{lemma}
\label{lem:rt-adj-prj-ind}
If $f:\bfQ \rightleftarrows \bfP:g$ is a Galois connection of posets, then
for each $x \in \bfQ$, we have an isomorphism
\[
g^{*}\bigl(\k[\bfQ]_x\bigr)\ \cong\ \k[\bfP]_{f(x)}
\qquad\text{in }\vect^{\bfP}
\]
that is natural in $x$.
\end{lemma}

\begin{proof}
For $y\in\bfP$, the value of $g^{*}\k[\bfQ](x,\blank)$ at $y$ is
\[
(g^{*}\k[\bfQ](x,\blank))(y)=\k[\bfQ](x,g(y)) \cong \k[\bfP](f(x),y)
\]
natural in $x$ and $y$ because $\bfQ(x,g(y)) \cong \bfP(f(x),y)$.
Hence $g^{*}\k[\bfQ](x,\blank)\cong \k[\bfP](f(x),\blank)$ in $\vect^{\bfP}$.
\end{proof}

\begin{theorem}[Stability]
\label{thm:stability}
Let $(\bfP, d_\bfP)$ be a finite metric poset. Then for any $M, N \in \vect^\bfP$,
\[
\distCM\bigl(P\down^M, P\down^N\bigr)\ \le\ \dGT(M, N).
\]
\end{theorem}

\begin{proof}
If there is no Galois coupling of $(M,N)$, then $\dGT(M,N)=\infty$ and the claim is tautological. Otherwise fix $\varepsilon>0$ and choose a coupling $(\bfQ, f \dashv g,\ h \dashv i,\ \Gamma)$ with
\[
g^\ast\Gamma\cong M,\qquad i^\ast\Gamma\cong N,\qquad \cost(\Gamma)\ \le\ \dGT(M,N)+\varepsilon.
\]
Let $R\down\to\Gamma$ be any projective resolution in $\vect^{\bfQ}$. Since precomposition is exact (Proposition~\ref{prop:Kan-posets}) and, for a Galois connection, preserves projectives (Proposition~\ref{prop:pullback_preserves_projectives}), the complexes
\[
E\down:=g^{*}R\down\quad\text{and}\quad F\down:=i^{*}R\down
\]
are projective resolutions of $M$ and $N$ in $\vect^{\bfP}$.

By Lemmas \ref{lem:ind-proj} and \ref{lem:mor-betw-proj},
we may set
$$
\begin{aligned}
R_i &= \Ds_{x \in \smd(R_i)} \k[\bfQ]_x,\\
\Mat(\partial^R_i) &= [a^{(i)}_{x^{(i+1)},x^{(i)}}]_{(x^{(i+1)},x^{(i)})\,\in\smd(R_{i+1}) \times \smd(R_i)}, \text{ and}\\
\partial^R_i &= [a^{(i)}_{x^{(i+1)},x^{(i)}}\,\ro(x^{(i+1)} \ge x^{(i)})]_{(x^{(i+1)},x^{(i)})\,\in\smd(R_{i+1}) \times \smd(R_i)}
\colon R_{i+1} \to R_i.
\end{aligned}
$$
By Lemma~\ref{lem:rt-adj-prj-ind} (and the analogue for $h\dashv i$), we have
$$
\begin{aligned}
E_i &= \Ds_{x \in \smd(R_i)} \k[\bfP]_{f(x)},
\quad
F_i = \Ds_{x \in \smd(R_i)} \k[\bfP]_{h(x)},
\\
\partial^E_i &= [a^{(i)}_{x^{(i+1)},x^{(i)}}\,\ro(f(x^{(i+1)}) \ge f(x^{(i)}))]_{(x^{(i+1)},x^{(i)})\,\in\smd(R_{i+1}) \times \smd(R_i)},\\
\partial^F_i &= [a^{(i)}_{x^{(i+1)},x^{(i)}}\,\ro(h(x^{(i+1)}) \ge h(x^{(i)}))]_{(x^{(i+1)},x^{(i)})\,\in\smd(R_{i+1}) \times \smd(R_i)}.
\end{aligned}
$$
Here note that the following holds:
\begin{equation}
\label{eq:Mat-E,F}
\left\{
\begin{aligned}
\Mat(\partial^E_i) &= [a^{(i)}_{x^{(i+1)},x^{(i)}}]_{(x^{(i+1)},x^{(i)})\,\in\smd(R_{i+1}) \times \smd(R_i)},\\
\Mat(\partial^F_i) &= [a^{(i)}_{x^{(i+1)},x^{(i)}}]_{(x^{(i+1)},x^{(i)})\,\in\smd(R_{i+1}) \times \smd(R_i)}.
\end{aligned}
\right.
\end{equation}
Indeed, if $a^{(i)}_{x^{(i+1)},x^{(i)}} \ne 0$, then
$x^{(i+1)} \ge x^{(i)}$, which shows that both
$f(x^{(i+1)}) \ge f(x^{(i)})$ and $h(x^{(i+1)}) \ge h(x^{(i)})$,
as desired.

Hence $|E\down|=|F\down|$ and $(E\down,F\down)\in\Padnn(P\down^M,P\down^N)$. Define the degreewise bijection $B_i$ by the identity on indices $x\in \smd(R_i)$:
\[
B_i:\ \k[\bfP]_{f(x)} \longmapsto \k[\bfP]_{h(x)}\qquad(x\in \smd(R_i)).
\]
Then by \eqref{eq:Mat-E,F}, $B = (B_i)_{i \ge 0}$ satisfies the compatibility condition on
differentials, and hence $B$ is a matching of $(E\down, F\down)$.
Here we have
\[
\distM(E\down,F\down)\ \le\ \cost(B)\ =\ \sup_{i}\ \sup_{x\in S_i} d_{\bfP}\bigl(f(x),h(x)\bigr)
\ \le\ \sup_{x\in\bfQ} d_{\bfP}\bigl(f(x),h(x)\bigr)\ =\ \cost(\Gamma).
\]
Taking the infimum over all compatible paddings yields
\[
\distCM\bigl(P\down^M,P\down^N\bigr)\ \le\ \distM(E\down,F\down)\ \le\ \cost(\Gamma)\ \le\ \dGT(M,N)+\varepsilon.
\]
This completes the proof.
\end{proof}

Note that in the proof of Theorem \ref{thm:stability}, the following was proved.

\begin{corollary}
\label{cor:Res-nonempty}
For any $M, N \in \vect^\bfP$, if
there exists a Galois coupling of $(M, N)$, then we have
$\Padnn(P^M\down, P^N\down) \ne \varnothing$, and moreover,
there exists some $(E\down, F\down) \in \Padnn(P^M\down, P^N\down)$ such that
$$
\Match(E\down, F\down) \ne \varnothing.
$$
\end{corollary}

The following gives an illustration of our proof of
Theorem \ref{thm:stability}.

\begin{example}\label{exm:stability-thm}
We compute the complex matching distance of the minimal projective
resolutions of the modules $M$ and $N$ in
Example~\ref{exm:stability}.
We first compute a minimal projective resolution of
$\Ga = V^\bfQ_\bfQ \oplus V^\bfQ_{\{1_L\}}$.
A minimal projective resolutions of $V^\bfQ_\bfQ$ and
$V^\bfQ_{\{1_L\}}$ are given as follows:
$$
\begin{aligned}
0 \to \k[\bfQ]_2 \xrightarrow{\bsmat{\ro(2 \ge 1_L)\\-\ro(2 \ge 1_R)}} \k[\bfQ]_{1_L}
\oplus \k[\bfQ]_{1_R} \xrightarrow{[\ro(1_L)\ \ro(1_R)]} &\ V^\bfQ_\bfQ \to 0,\\
0 \to \k[\bfQ]_2 \xrightarrow{\ro(2 \ge 1_L)} \k[\bfQ]_{1_L}
\xrightarrow{\ro(1_L)} &\ V^\bfQ_{\{1_L\}} \to 0.
\end{aligned}
$$
Therefore, a minimal projective resolution $P^\Ga\down$ of $\Ga$
as a nonnegative complex is given by
$$
\cdots \to 0 \to \k[\bfQ]_2 \oplus \k[\bfQ]_2 \xrightarrow{\bsmat{\ro(2 \ge 1_L) & 0\\-\ro(2 \ge 1_R)& 0\\0& \ro(2 \ge 1_L)}}
\overbrace{\k[\bfQ]_{1_L} \ds \k[\bfQ]_{1_R}\ds \k[\bfQ]_{1_L}}^{\deg 0}.
$$
Up to the natural isomorphism stated in Lemma \ref{lem:rt-adj-prj-ind},
the exact functors $g^*$ and $i^*$ send this to the following, which are projective resolutions of $M$ and $N$, respectively:
\begin{equation}
\label{eq:E-F}
\begin{aligned}
E\down&:= (\cdots \to 0 \to \k[\bfP]_2 \ds \k[\bfP]_2
\xrightarrow{\bsmat{\ro(2 \ge 1) & 0\\-\ro(2\ge 2)& 0\\0& \ro(2 \ge 1)}}
\overbrace{\k[\bfP]_{1} \ds \k[\bfP]_{2} \ds \k[\bfP]_{1}}^{\deg 0} 
),\\[-30pt]
&\phantom{:= (\cdots \to 0 \to}
\hspace{20pt}\rotatebox{-90}{$\mapsto$}
\hspace{30pt}\rotatebox{-90}{$\mapsto$}
\phantom{\xrightarrow{\bsmat{\ro(1\le 2) & 0\\-\ro(2\le 2)& 0\\0& \ro(1\le 2)}}}
\hspace{30pt}\rotatebox{-90}{$\mapsto$}
\hspace{30pt}\rotatebox{-90}{$\mapsto$}
\hspace{30pt}\rotatebox{-90}{$\mapsto$}
\\[-20pt]
F\down&:= (\cdots \to 0 \to \k[\bfP]_2 \ds \k[\bfP]_2 
\xrightarrow{\bsmat{\ro(2\ge 2) & 0\\-\ro(2 \ge 1)& 0\\0& \ro(2\ge 2)}}
\underbrace{\k[\bfP]_{2} \ds \k[\bfP]_{1} \ds \k[\bfP]_{2}}_{\deg 0}
).
\end{aligned}
\end{equation}
Since $\ro(2\ge 1) \circ \ro(2\ge 2) = \ro(\pi_{2\ge 2}\pi_{2 \ge 1}) = \ro(2\ge 1)$, an elementary row operation shows that
the matrix $\bsmat{\ro(2 \ge 1) & 0\\-\ro(2\ge 2)& 0\\0& \ro(2 \ge 1)}$ in $E\down$ is equivalent to $\bsmat{0 & 0\\-\ro(2\ge 2)& 0\\0& \ro(2 \ge 1)}$, and
the matrix $\bsmat{\ro(2\ge 2) & 0\\-\ro(2 \ge 1)& 0\\0& \ro(2\ge 2)}$ in $F\down$
is equivalent to $\bsmat{\ro(2\ge 2) & 0\\0& 0\\0& \ro(2\ge 2)}$,
which means that
$$
\begin{aligned}
E\down &\cong P^M\down \ds \Cone(\id_{\k[\bfP]_2}),\ \text{and}\\
F\down &\cong P^N\down \ds \Cone(\id_{\k[\bfP]_2}) \ds \Cone(\id_{\k[\bfP]_2}).
\end{aligned}
$$
Define a matching $B \in \Match(E\down, F\down)$
vertically as displayed in \eqref{eq:E-F}.
Then clearly, $\cost(B) = 1$, and hence
we have $\distCM\bigl(P\down^M,P\down^N\bigr) \le 1$.
Since $\hat{\al}([P^M\down]) \ne \hat{\al}([P^N\down])$, we have
$\distCM\bigl(P\down^M,P\down^N\bigr) > 0$ by
Lemma \ref{lem:distB=zero}.
Therefore, $\distCM\bigl(P\down^M,P\down^N\bigr) = 1$.
\end{example}

\subsection{Examples}

We now revisit our running 1D and 2D examples to illustrate the stability
inequality.  
In both cases the G\"ulen--McCleary distance and the complex matching distance
coincide, showing that the bound in Theorem~\ref{thm:stability} is sharp.

\begin{example}\label{ex:stability-chain}
From Examples~\ref{ex:gtd-chain} and~\ref{ex:bneck-chain} we have
\[
\dGT(M,N)=1,
\qquad
\distCM(P\down^{M},P\down^{N})=1.
\]
Hence the stability inequality
\[
\distCM(P\down^{M},P\down^{N})\ \le\ \dGT(M,N)
\]
holds with equality.

On the transport side, the Galois coupling moves each interval in $M$ forward
by one step in the parameter.  On the resolution side, 
padding by
contractible cones aligns the minimal projective resolutions and produces a matching of the same cost.
Thus stability is numerically sharp in this one-parameter example.
\end{example}

\begin{example}\label{ex:stability-2d}
For the $2$-parameter modules $M$ and $N$ from Example~\ref{ex:gtd-2d}, we
computed
\[
\dGT^{\bfP}(M,N)=1
\qquad\text{and}\qquad
\distCM(P\down^{M},P\down^{N})=1
\]
in Examples~\ref{ex:gtd-2d} and~\ref{ex:bneck-2d}.  Hence
\[
\distCM(P\down^{M},P\down^{N})
\ =\
\dGT^{\bfP}(M,N)
\ =\ 1,
\]
so the stability inequality holds with equality.

Geometrically, the transport coupling shifts each subspace of $M$ by
$(+1,+1)$ or $(0,+1)$ into the corresponding arm of the $L$-shape $N$.
On the resolution
side,
padding $P\down^{N}$ by two shifted cones equalizes the degreewise sizes,
and permits a matching of cost~$1$.  Stability is therefore sharp in this
two-parameter example as well.
\end{example}

\section{Application to Persistence}
\label{sec:persistence}

In this section we extract a persistence-like construction from a
$\bfP$-module by passing to the interval poset and taking kernels of structure
maps. We then show that the G\"ulen--McCleary stability inequality descends to
this construction, recovering classical bottleneck stability when $\bfP$ is a
finite chain.

\begin{definition}[Augmented poset]
\label{dfn:Int-2-fun}
Let $(\bfP,d_{\bfP})$ be a finite metric poset, $\top$ a symbol
not contained in $\bfP$.
We set $\ovl{\bfP}:= \bfP \cup \{\top\}$, and define a partial order
on it by extending that of $\bfP$ with additional order
$x < \top$ for all $x \in \bfP$ and $\top \le \top$.
If $\bfP$ and  $\bfQ$ are distinct finite posets,
then we treat the top elements $\top$ of $\oP$ and $\oQ$ are distinct
although we denote them by the same symbol.
If we need to distinguish the top element,
then we denote it by $\top_\bfP$ for $\bfP$.

We extend $d_{\bfP}$ to an \emph{extended} metric
$d_{\oP} \colon \oP \to [0,+\infty]$ by
\[
d_{\oP}(x,\top)=d_{\oP}(\top,x)=+\infty\ (x\neq\top), 
\qquad d_{\oP}(\top,\top)=0,
\]
which defines a finite metric poset $(\oP, d_{\oP})$.

We extend this correspondence $\bfP \mapsto \oP$
to a 2-functor $\ovl{(\blank)} \colon \Pos \to \Pos$ of the 2-category $\Pos$ of finite posets.
If $f \colon \bfP \to \bfQ$ is a monotone map, then we define a monotone
map $\ovl{f} \colon \oP \to \oQ$ to be an extension of $f$ with
$\ovl{f}(\top):= \top$.

If $f, g \colon \bfP \to \bfQ$ are monotone maps as functors, and
$\al \colon f \To g$ is a natural transformation, then we define
a natural transformation $\ovl{\al} \colon \ovl{f} \To \ovl{g}$ in an obvious way.
Namely, the fact that $\al \colon f \To g$ is a natural transformation
means that $f(x) \le g(x)$ for all $x \in \bfP$ and $\al_x:= \pi_{g(x) \ge f(x)} \colon f(x) \to g(x)$ in $\bfQ$ (see \eqref{eq:mor-in-poset} for the notation of $\pi$), which implies that
$\ovl{f}(x) \le \ovl{g}(x)$ for all $x \in \oP$, and therefore,
we can set $\ovl{\al}_x:= \pi_{\ovl{g}(x) \ge \ovl{f}(x)} \colon
\ovl{f}(x) \to \ovl{g}(x)$ in $\oQ$ to define a natural transformation
$\ovl{\al}:= (\ovl{\al}_x)_{x \in \oP} \colon \ovl{f} \To \ovl{g}$.
\end{definition}

We will apply the following definition to the finite poset $\oP$
for a finite poset $\bfP$.
In order to avoid confusion, we use other letters $\bfS$ and $\bfT$
to denote finite posets in this definition.

\begin{definition}
\label{dfn:Int-poset}
Let $\bfS$ be a finite metric poset.
Define the \emph{interval poset} $\Int\bfS$ to be the
full subposet of the product poset $\bfS \times \bfS$ whose underlying set is given by the graph of the binary relation $\le$ on $\bfS$, thus
\begin{equation}
\label{eq:IntS}
\Int \bfS:= \{(x,y) \in \bfS \times \bfS \mid x \le y\}.
\end{equation}
Equip $\Int\bfS$ with the $L^\infty$ extended metric
\[
d_{\Int\bfS}\bigl((x_1,y_1),(x_2,y_2)\bigr)
:=\max\{\,d_{\bfS}(x_1,x_2),\, d_{\bfS}(y_1,y_2)\,\}
\]
to define a finite metric poset $(\Int \bfS, d_{\Int \bfS})$.
\end{definition}

\begin{remark}
\label{rmk:closed-intervals}
Let $\bfS$ be a finite poset.
Then the correspondence
$(x,y) \mapsto [x,y]$ gives an isomorphism
from $\Int \bfS$ to the poset $\{[x,y] \mid x,y \in \bfS, x \le y\}$ of
{\em closed intervals} 
in $\bfS$ whose partial order is defined by
\[
[x_1,y_1]\le [x_2,y_2] \quad\Longleftrightarrow\quad x_1\le x_2
 \ \text{and}\ y_1\le y_2.
\]
By this similarity, we adopt the notation $\Int$ standing for intervals.
But note that this order is not given by the inclusion relation of intervals
as in other literature such as \cite{ASASHIBA2023100007}.
\end{remark}

\begin{definition}
\label{dfn:Int-2fun}
We extend the correspondence $\bfS \mapsto \Int\bfS$
to a 2-functor $\Int \colon \Pos \to \Pos$.
Let $f \colon \bfS \to \bfT$ be a monotone map.
We define $\Int f \colon \Int \bfS \to \Int \bfT$ by setting
$(\Int f)((x,y)):= (f(x), f(y))$ for all $(x,y) \in \Int \bfS$.
Note here that $(f(x), f(y)) \in \Int \bfS$ because $f$ is monotone.
If $(x_1, y_1) \le (x_2, y_2)$ in $\Int\bfS$, then
we have $(f(x_1), f(y_1)) \le (f(x_2), f(y_2))$ again because $f$ is monotone.
Hence we can define
$$
(\Int f)(\pi_{(x_2,y_2)\ge (x_1,y_1)}):= \pi_{(f(x_2), f(y_2))\ge (f(x_1), (y_1))}.
$$
Finally let $\al \colon f \To g$ be a natural transformation
between monotone maps $f, g \colon \bfS \to \bfT$ regarded as functors.
Then for each $x \in \bfS$, $\al_x \colon f(x) \to g(x)$ means that
$f(x) \le g(x)$ in $\bfT$ and $\al_x = \pi_{g(x)\ge f(x)}$.
Then for any $(x,y) \in \Int \bfS$, we have
$(f(x), f(y)) \le (g(x), g(y))$ in $\Int \bfT$.
Hence we can set
$$
(\Int \al)_{(x,y)}:= \pi_{[g(x), g(y)) \ge [f(x), f(y))} \colon
(\Int f)((x,y)) \to (\Int g)((x,y)),
$$
which defines a natural transformation
$\Int \al:= ((\Int \al)_{[x,y)})_{[x,y)\in \Int \bfS}
\colon \Int f \To \Int g$.
\end{definition}

\begin{remark}
To define $\Int f \colon \Int \bfS \to \Int \bfT$ for a monotone map
$f \colon \bfS \to \bfT$ of finite posets,
we need to have $(z, z) \in \Int \bfT$ for all $z \in \bfT$ because it may occur
that $f(x) = f(y)$ even if $x < y$ in $\bfS$.
Therefore, in \eqref{eq:IntS}, we cannot replace $x \le y$ by $x < y$.
Thus even if $\bfS$ is a finite linearly ordered set,
we cannot replace $\Int \bfS$ by the poset $\{[x,y) \mid x, y \in \bfS, x\le y\}$ of half-open intervals in $\bfS$ whose partial order is defined by the same was as in the
closed interval case in Remark \ref{rmk:closed-intervals}.
\end{remark}

By composing these 2-functors,
we obtain a 2-functor
\begin{equation}
\label{eq:2-fun-Int-ovl}
\Int \circ \ovl{(\blank)} \colon \Pos \to \Pos.
\end{equation}
This 2-functor is used in the proof of Lemma \ref{lem:Int-Galois}.

\begin{definition}
\label{dfn:ext-module}
We define a functor $\ext \colon \vect^\bfP \to \vect^{\oP}$ as follows.
Let $M$ be in $\vect^\bfP$.
Then $\ext(M):= \ovl{M} \colon \oP \to \vect$ is defined by
extending $M$ with $\ovl{M}(\top):= 0$ and $\ovl{M}(\pi_{\top\ge x}):= 0$
for all $x \in \bfP$.

Let $\be \colon M \to N$ be a morphism in $\vect^\bfP$.
Then $\ext(\be):= \ovl{\be} \colon \ovl{M} \to \ovl{N}$ is defined 
by extending $\be$ with $\ovl{\be}_{\top}:= 0 \colon \ovl{M}(\top) \to \ovl{N}(\top)$.
\end{definition}

\begin{definition}
\label{dfn:KM}
Let $M\in\vect^{\bfP}$.
Then we define $K(M)\in\vect^{\Int\oP}$ as follows.
For each $(x,y) \in \Int \oP$, we set
\[
K(M)((x,y)) :=
\Ker \ovl{M}(\pi_{y \ge x}).
\]
Note that we have $K(M)((x,\top)) = M(x)$ 
by the definition of $\ovl{M}$.
Next let $(x_1,y_1)\le (x_2,y_2)$ in $\Int\oP$.
Then we have a commutative diagram
$$
\begin{tikzcd}[column sep=50pt]
0 & K(M)((x_1,y_1)) & \ovl{M}(x_1) & \ovl{M}(y_1)\\
0 & K(M)((x_2,y_2)) & \ovl{M}(x_2) & \ovl{M}(y_2)
\Ar{1-1}{1-2}{}
\Ar{1-2}{1-3}{}
\Ar{1-3}{1-4}{"\ovl{M}(\pi_{y_1 \ge x_1})"}
\Ar{2-1}{2-2}{}
\Ar{2-2}{2-3}{}
\Ar{2-3}{2-4}{"\ovl{M}(\pi_{y_2 \ge x_2})"'}
\Ar{1-2}{2-2}{"K(M)(\pi_{(x_2,y_2) \ge (x_1,y_1)})"', dashed}
\Ar{1-3}{2-3}{"\ovl{M}(\pi_{x_2\ge x_1})"}
\Ar{1-4}{2-4}{"\ovl{M}(\pi_{y_2\ge y_1})"}
\end{tikzcd}
$$
with exact rows (first ignore the dashed arrow).
By the universality of the kernel, there exists a unique morphism
$K(M)(\pi_{(x_2,y_2) \ge (x_1,y_1)})$ making the diagram
commutative, which is a restriction of $\ovl{M}(\pi_{x_2\ge x_1})$.
As is easily seen, the above defines a functor
$K(M)\colon \Int \oP \to \vect$, an object of $\vect^{\Int\oP}$.

We extend this to a functor $K:\vect^{\bfP}\to \vect^{\Int\oP}$.
Let $\be \colon M \to N$ be in $\vect^\bfP$.
Then we define a morphism $K(\be):= (K(\be)_{(x,y)})_{(x,y)\in \Int\oP} \colon K(M) \to K(N)$ in $\vect^{\Int \oP}$
as follows.
For each $(x,y) \in \Int \oP$, we have a commutative diagram
$$
\begin{tikzcd}[column sep=50pt]
0 & K(M)((x,y)) & \ovl{M}(x) & \ovl{M}(y)\\
0 & K(N)((x,y)) & \ovl{N}(x) & \ovl{N}(y)
\Ar{1-1}{1-2}{}
\Ar{1-2}{1-3}{}
\Ar{1-3}{1-4}{"\ovl{M}(\pi_{y \ge x})"}
\Ar{2-1}{2-2}{}
\Ar{2-2}{2-3}{}
\Ar{2-3}{2-4}{"\ovl{M}(\pi_{y \ge x})"'}
\Ar{1-2}{2-2}{"K(\be)_{(x,y)}"', dashed}
\Ar{1-3}{2-3}{"\be_x"}
\Ar{1-4}{2-4}{"\be_y"}
\end{tikzcd}
$$
with exact rows.
Again the universality of the kernel yields a unique morphism
$K(\be)_{(x,y)}$ above
making the diagram commutative,
which is just the restriction of $\be_x$ to $K(M)((x,y))$.
Then it is easy to see that
$K:\vect^{\bfP}\to
\vect^{\Int\oP}$ becomes a functor.
When we need to distinguish $\bfP$, we denote it by $K_\bfP$.
\end{definition}

We record two lemmas that will be used below.

\begin{lemma}\label{lem:Int-Galois}
If $f:\bfQ\rightleftarrows\bfP:g$ is a Galois connection (resp.\ insertion), then so is
\[
\Int \ovl{f}:\Int\oQ\rightleftarrows\Int\oP:\Int \ovl{g}.
\]
\end{lemma}

\begin{proof}
Apply the 2-functor $\Int\circ \ovl{(\blank)}$ defined
in \eqref{eq:2-fun-Int-ovl} to an adjoint system
$(f,g, \et, \ep)$ to obtain an adjoint system
$(\Into f,\Into g, \Into \et, \Into \ep)$.
Here for insertions, note that
if $f \circ g = \id_\bfP$, then $(\Into f) \circ (\Into g) = \id_{\Int \oP}$.
\end{proof}

\begin{lemma}\label{lem:Int-Lipschitz}
For any monotone maps $f,h:\bfQ\to\bfP$, we have
\[
\sup_{(u,v)\in\Int\oQ}
d_{\Int\oP}\bigl((\Into f)(u,v),(\Into h)(u,v)\bigr)
\ \le\
\sup_{w\in\oQ} d_{\oP}(\ovl{f}(w),\ovl{h}(w)).
\]
\end{lemma}

\begin{proof}
For any $(u,v) \in \Int \oQ$, we have
$$
d_{\Int\oP}\bigl((\Into f)(u,v),(\Into h)(u,v)\bigr)
= \max \{d_{\oP}(\ovl{f}(u), \ovl{h}(u)), d_{\oP}(\ovl{f}(v), \ovl{h}(v))\}
\le \sup_{w\in \oQ} d_{\oP}(\ovl{f}(w), \ovl{h}(w)),
$$
from which the inequality follows.
\end{proof}

\begin{lemma}\label{lem:K-commute}
For any monotone map $g:\bfP\to\bfQ$,
we have a strictly commutative diagram
$$
\begin{tikzcd}[column sep=40pt]
\vect^\bfQ & \vect^\bfP\\
\vect^{\Int \oQ} & \vect^{\Int \oP}
\Ar{1-1}{1-2}{"g^*"}
\Ar{2-1}{2-2}{"(\Into g)^*"'}\\
\Ar{1-1}{2-1}{"K_\bfQ"'}
\Ar{1-2}{2-2}{"K_\bfP"}
\end{tikzcd}
$$
of functors:
$(\Into g)^*\circ K_\bfQ = K_\bfP \circ g^*$.
\end{lemma}

\begin{proof}
Let $M \in \vect^\bfQ$, and $(x,y) \in \Int \oP$.
Then
$$
\begin{aligned}
((\Into g)^*\circ K_Q)(M)((x,y)) &= K_Q(M)((\ovl{g}(x), \ovl{g}(y)))
= \Ker \ovl{M}(\pi_{\ovl{g}(y) \ge \ovl{g}(x)}),\\
(K_\bfP \circ g^*)(M)((x,y)) &= K_\bfP(M \circ g)((x,y))
= \Ker (\ovl{M \circ g})(\pi_{y \ge x}) = \Ker \ovl{M}(\pi_{\ovl{g}(y) \ge \ovl{g}(x)}).
\end{aligned}
$$
Hence $(\Into g)^*\circ K_\bfQ = K_\bfP \circ g^*$.
\end{proof}

\begin{lemma}
\label{lem:Gal-coupling-Into}
Let $(\bfQ,\, f \dashv g,\, h \dashv i,\, \Gamma)$ be a Galois coupling of the pair $(M, N)$.
Then
$(\Int \oQ,\, \Into f \dashv \Into g,\, \Into h \dashv \Into i,\, K(\Gamma))$
is a Galois coupling of the pair $(K(M), K(N))$.
\end{lemma}

\begin{proof}
By Lemma~\ref{lem:Int-Galois},
we have Galois insertions $\Into f \dashv \Into g$
and $\Into h \dashv \Into i$.
Moreover, by Lemma \ref{lem:K-commute}, we have
$K(M) = K(g^*(\Ga)) = (\Into g)^*(K(\Ga))$, and similarly,
$K(N) = (\Into i)^*(K(\Ga))$.
\end{proof}

\begin{proposition}\label{prop:K-Lipschitz}
For all $M,N \in \vect^\bfP$, we have
\[
\dGT^{\Int\oP}(K(M),K(N))\ \le\ \dGT^{\bfP}(M,N).
\]
\end{proposition}

\begin{proof}
Let $(\bfQ,\, f \dashv g,\, h \dashv i,\, \Gamma)$ be a Galois coupling of the pair $(M, N)$.
Then by Lemma \ref{lem:Gal-coupling-Into},
$(\Int \oQ,\, \Into f \dashv \Into g,\, \Into h \dashv \Into i,\, K(\Gamma))$
is a Galois coupling of the pair $(K(M), K(N))$.
Therefore, by Lemma \ref{lem:Int-Lipschitz},
$\cost K(\Ga) \le \cost \Ga$.
Indeed,
$$
\cost K(\Ga) = \sup_{(u,v)\in\Int\oQ}
d_{\Int\oP}\bigl((\Into f)(u,v),(\Into h)(u,v)\bigr)
\ \le\
\sup_{w\in\oQ} d_{\oP}(\ovl{f}(w),\ovl{h}(w)).
$$
Here if $w = \top \in \oQ$,
then we have
$$
d_{\oP}(\ovl{f}(w),\ovl{h}(w)) = d_{\oP}(\ovl{f}(\top),\ovl{h}(\top))
= d_{\oP}(\top,\top) = 0.
$$
This shows that
$\sup_{w\in\oQ} d_{\oP}(\ovl{f}(w),\ovl{h}(w)) =
\sup_{w\in\bfQ} d_{\bfP}(f(w),h(w)) = \cost \Ga$.
Hence
$$
\dGT^{\Int \oP}(K(M), K(N)) \le \cost K(\Ga) \le \cost \Ga,
$$
which shows the assertion.
\end{proof}

\begin{definition}
\label{def:persistence_diagram}
For $M\in\vect^{\bfP}$ let $K\down^M$ denote a minimal projective resolution of
$K(M)$.  Define the \emph{persistence diagram} of $M$
as the homotopy class of $K\down^M$ in $\Cnn(\prj \k[\oP])$.
We will abuse notation and write $K\down^M$ to mean 
the persistence diagram of $M$.
\end{definition}

Padding a resolution by a cone $\Cone(\id_E)[a]$ adds one copy of $E$ in
degrees $a$ and $a+1$, playing the role of adding diagonal points in classical
bottleneck matchings.

\begin{theorem}\label{thm:persistence-stability}
For all $M,N \in \vect^\bfP$, we have
\[
\distCM(K\down^{M},K\down^{N})
\ \le\
\dGT^{\Int\oP}(K(M),K(N))
\ \le\
\dGT^{\bfP}(M,N).
\]
For $\bfP=\{1<\dots<n\}$, this recovers classical bottleneck stability.
\end{theorem}

\begin{proof}
Apply Theorem \ref{thm:stability} to the finite metric poset
$(\Int\oP, d_{\Int \oP})$ and $K(M), K(N) \in \vect^{\Int\oP}$ to
have the first inequality, the second is stated in
Proposition~\ref{prop:K-Lipschitz}.
\end{proof}

\subsection{Examples}

We now compute persistence diagrams for our running 1D and 2D examples, and verify stability at the level of minimal resolutions in $\vect^{\Int\oP}$.

\begin{example}
\label{ex:persistence-chain-correct}
Let $\bfP$ be as in Example \ref{ex:gtd-chain}.
Then $\oP =\{1<2<3<4<\top\}$ with $d_{\oP}(x,y)=|x-y|$ and
$d_{\oP}(x,\top)=\infty$ for all $x, y \in \bfP$ and $d_{\oP}(\top, \top) = 0$.  
Take the modules
\[
M= V_{[1,1]} \ds V_{[2,3]},\qquad N = V_{[2,3]}
\]
as in that example.
Then $\ovl{M}, \ovl{N}$ are given as representations of $H(\ovl{\bfP}) = 1 \to 2 \to 3 \to 4 \to \top$ as follows by Definition \ref{dfn:ext-module}:
\begin{equation}
\label{eq:rep-M-N}
\ovl{M} = (\k \ya{0} \k \ya{1} \k \to 0 \to 0),
\quad\text{ and }\quad
\ovl{N} = (0 \to \k \ya{1} \k \to 0 \to 0).
\end{equation}
Now $H(\Int \oP)$ is given as follows:
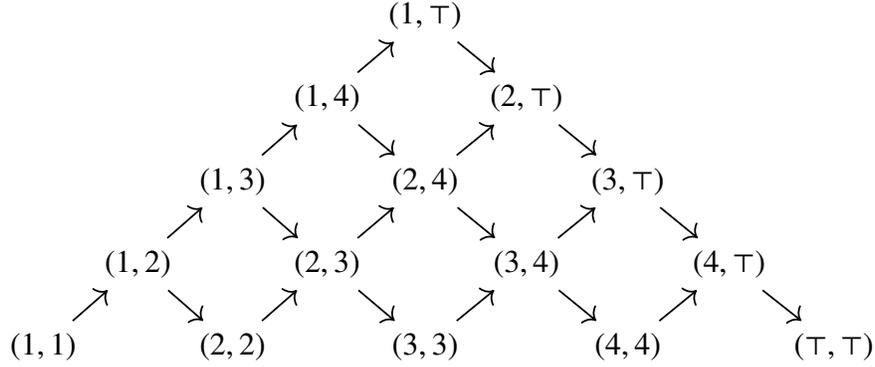
\[
\begin{tikzcd}[column sep = 0pt, row sep=12pt]
&&&& \Nname{1T}(1,\top) \\
&&& \Nname{14}(1,4) && \Nname{2T}(2,\top)\\
&& \Nname{13}{(1,3)} && \Nname{24}{(2,4)} && \Nname{3T}{(3,\top)} \\
&\Nname{12}{(1,2)} && \Nname{23}{(2,3)} && \Nname{34}{(3,4)} && \Nname{4T}{(4,\top)}\\
\Nname{11}(1,1) && \Nname{22}{(2,2)} && \Nname{33}{(3,3)} && \Nname{44}{(4,4)} && \Nname{TT}{(\top,\top)}
\Ar{11}{12}{}
\Ar{12}{13}{}
\Ar{13}{14}{}
\Ar{14}{1T}{}
\Ar{22}{23}{}
\Ar{23}{24}{}
\Ar{24}{2T}{}
\Ar{33}{34}{}
\Ar{34}{3T}{}
\Ar{44}{4T}{}
\Ar{1T}{2T}{}
\Ar{2T}{3T}{}
\Ar{3T}{4T}{}
\Ar{4T}{TT}{}
\Ar{14}{24}{}
\Ar{24}{34}{}
\Ar{34}{44}{}
\Ar{13}{23}{}
\Ar{23}{33}{}
\Ar{12}{22}{}
\end{tikzcd}
\]
Then by \eqref{eq:rep-M-N}, we can compute $K(M), K(N)$ as
representations of $H(\Int \oP)$ as follows:
{\footnotesize
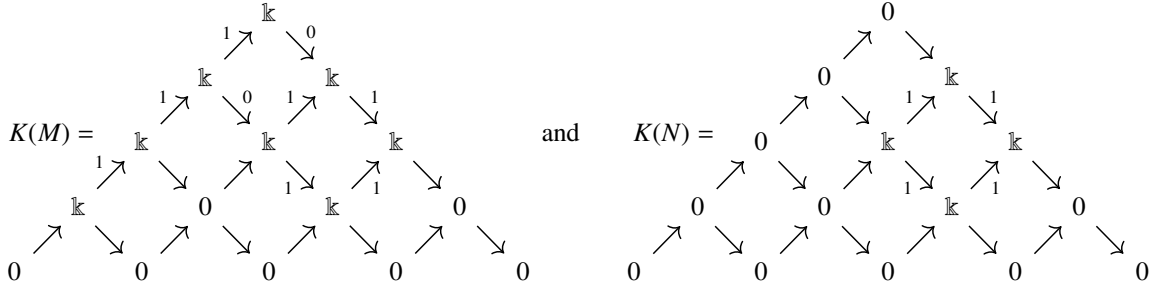
\[
K(M) = \hspace{-4em}
\begin{tikzcd}[column sep = 10pt, row sep=10pt]
&&&& \Nname{1T}\k \\
&&& \Nname{14}\k && \Nname{2T}\k\\
&& \Nname{13}\k && \Nname{24}\k && \Nname{3T}\k \\
&\Nname{12}\k && \Nname{23}0 && \Nname{34}\k && \Nname{4T}0\\
\Nname{11}0 && \Nname{22}0 && \Nname{33}0 && \Nname{44}0 && \Nname{TT}0
\Ar{11}{12}{}
\Ar{12}{13}{"1"}
\Ar{13}{14}{"1"}
\Ar{14}{1T}{"1"}
\Ar{22}{23}{}
\Ar{23}{24}{}
\Ar{24}{2T}{"1"}
\Ar{33}{34}{}
\Ar{34}{3T}{"1"'}
\Ar{44}{4T}{}
\Ar{1T}{2T}{"0"}
\Ar{2T}{3T}{"1"}
\Ar{3T}{4T}{}
\Ar{4T}{TT}{}
\Ar{14}{24}{"0"}
\Ar{24}{34}{"1"'}
\Ar{34}{44}{}
\Ar{13}{23}{}
\Ar{23}{33}{}
\Ar{12}{22}{}
\end{tikzcd}
\text{and}\qquad
K(N) = \hspace{-4em}
\begin{tikzcd}[column sep = 10pt, row sep=10pt]
&&&& \Nname{1T}0 \\
&&& \Nname{14}0 && \Nname{2T}\k\\
&& \Nname{13}0 && \Nname{24}\k && \Nname{3T}\k \\
&\Nname{12}0 && \Nname{23}0 && \Nname{34}\k && \Nname{4T}0\\
\Nname{11}0 && \Nname{22}0 && \Nname{33}0 && \Nname{44}0 && \Nname{TT}0
\Ar{11}{12}{}
\Ar{12}{13}{}
\Ar{13}{14}{}
\Ar{14}{1T}{}
\Ar{22}{23}{}
\Ar{23}{24}{}
\Ar{24}{2T}{"1"}
\Ar{33}{34}{}
\Ar{34}{3T}{"1"'}
\Ar{44}{4T}{}
\Ar{1T}{2T}{}
\Ar{2T}{3T}{"1"}
\Ar{3T}{4T}{}
\Ar{4T}{TT}{}
\Ar{14}{24}{}
\Ar{24}{34}{"1"'}
\Ar{34}{44}{}
\Ar{13}{23}{}
\Ar{23}{33}{}
\Ar{12}{22}{}
\end{tikzcd}.
\]
}
Therefore, $K^M\down$ and $K^N\down$ is given by
(differentials are given by their coefficient matrices)
$$
\begin{aligned}
K^M\down&= (\cdots \to 0 \to \kIP{(2,2)}\ds \kIP{(4,4)}
\ya{\bsmat{1&0\\0&1}} \kIP{(1,2)} \ds \kIP{(2,4)}),\\
K^N\down&= (\cdots \to 0 \to  \kIP{(4,4)}
\ya{1} \kIP{(2,4)}).
\end{aligned}
$$
Padding $K^N\down$ by $\Cone(\id_{\kIP{(1,2)}})$ to have
$F\down = K^N\down \ds  \Cone(\id_{\kIP{(1,2)}})$,
we can give a matching $B \in \Match(K^m\down, F\down)$ by the following table:
$$
\begin{array}{c|cc|cc}
\deg & \multicolumn{2}{c|}1 & \multicolumn{2}{c}0\\
\hline
K^M\down & \kIP{(2,2)} & \kIP{(4,4)} & \kIP{(1,2)} & \kIP{(2,4)} \\
\hline
F\down & \kIP{(1,2)} & \kIP{(4,4 )} & \kIP{(1,2)} & \kIP{(2,4)} \\
\hline
\dist & 1 & 0 & 0 & 0
\end{array}\ ,
$$
where we have
$$
\Mat(\partial^{K^M}_0) = \bsmat{1&0\\0&1},\text{ and }\ 
\Mat(\partial^F_0) = \bsmat{1&0\\0&1},
$$
and hence $B$ satisfies the compatibility condition on differentials.
Thus $\distCM(K^M\down, K^N\down) = 1$.
Combined with $\dGT^{\bfP}(M,N)=1$, this realizes the stability bound with equality.
\end{example}

\begin{example}
\label{exm:K(M)-2D}
Let $\bfP:= \{1,2,3\}^2$ as in Example \ref{ex:gtd-2d}.
Then $\oP$ is visualized by its Hasse quiver as follows:
$$
H(\oP) = 
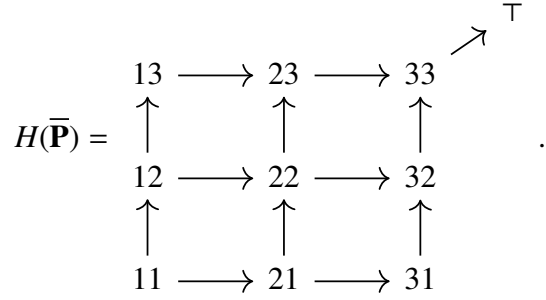
\begin{tikzcd}
&&&[-15pt] \Nname{T}\top\\[-15pt]
\Nname{13}13 & \Nname{23}23 & \Nname{33}33\\
\Nname{12}12 & \Nname{22}22 & \Nname{32}32\\
\Nname{11}11 & \Nname{21}21 & \Nname{31}31
\Ar{13}{23}{} \Ar{23}{33}{}
\Ar{12}{22}{} \Ar{22}{32}{}
\Ar{11}{21}{} \Ar{21}{31}{}
\Ar{11}{12}{} \Ar{12}{13}{}
\Ar{21}{22}{} \Ar{22}{23}{}
\Ar{31}{32}{} \Ar{32}{33}{}
\Ar{33}{T}{}
\end{tikzcd}.
$$
Then $\Int \oP$ is given by its Hasse quiver as follows:
{\tiny
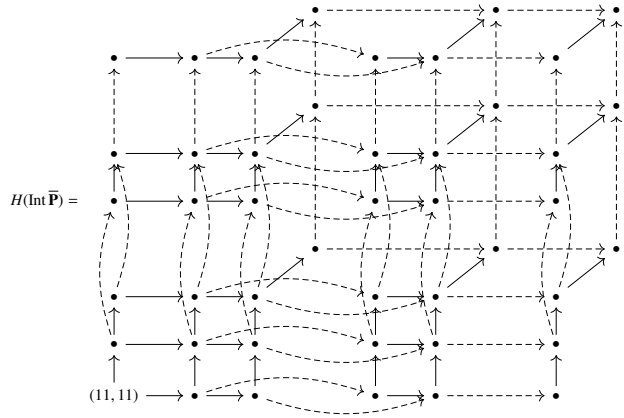
\[
H(\Int \oP) = 
\begin{tikzcd}
	&&& \bullet &&& \bullet && \bullet \\
	\bullet & \bullet & \bullet && \bullet & \bullet && \bullet \\
	&&& \bullet &&& \bullet && \bullet \\
	\bullet & \bullet & \bullet && \bullet & \bullet && \bullet \\
	\bullet & \bullet & \bullet && \bullet & \bullet && \bullet \\
	&&& \bullet &&& \bullet && \bullet \\
	\bullet & \bullet & \bullet && \bullet & \bullet && \bullet \\
	\bullet & \bullet & \bullet && \bullet & \bullet && \bullet \\
	(11,11) & \bullet & \bullet && \bullet & \bullet && \bullet
	\arrow[dashed, from=1-4, to=1-7]
	\arrow[dashed, from=1-7, to=1-9]
	\arrow[from=2-1, to=2-2]
	\arrow[from=2-2, to=2-3]
	\arrow[bend left=18pt, dashed, from=2-2, to=2-5]
	\arrow[from=2-3, to=1-4]
	\arrow[bend right=18pt, dashed, from=2-3, to=2-6]
	\arrow[from=2-5, to=2-6]
	\arrow[from=2-6, to=1-7]
	\arrow[dashed, from=2-6, to=2-8]
	\arrow[from=2-8, to=1-9]
	\arrow[dashed, from=3-4, to=1-4]
	\arrow[dashed, from=3-4, to=3-7]
	\arrow[dashed, from=3-7, to=1-7]
	\arrow[dashed, from=3-7, to=3-9]
    \arrow[dashed, from=6-7, to=3-7]
	\arrow[dashed, from=3-9, to=1-9]
	\arrow[dashed, from=4-1, to=2-1]
	\arrow[from=4-1, to=4-2]
	\arrow[dashed, from=4-2, to=2-2]
	\arrow[from=4-2, to=4-3]
	\arrow[bend left=18pt, dashed, from=4-2, to=4-5]
	\arrow[dashed, from=4-3, to=2-3]
	\arrow[from=4-3, to=3-4]
	\arrow[bend right=18pt, dashed, from=4-3, to=4-6]
	\arrow[dashed, from=4-5, to=2-5]
	\arrow[from=4-5, to=4-6]
	\arrow[dashed, from=4-6, to=2-6]
	\arrow[from=4-6, to=3-7]
	\arrow[dashed, from=4-6, to=4-8]
	\arrow[dashed, from=4-8, to=2-8]
	\arrow[from=4-8, to=3-9]
	\arrow[from=5-1, to=4-1]
	\arrow[from=5-1, to=5-2]
	\arrow[from=5-2, to=4-2]
	\arrow[from=5-2, to=5-3]
	\arrow[bend left=18pt, dashed, from=5-2, to=5-5]
	\arrow[from=5-3, to=4-3]
	\arrow[bend right=18pt, dashed, from=5-3, to=5-6]
	\arrow[from=5-5, to=4-5]
	\arrow[from=5-5, to=5-6]
	\arrow[from=5-6, to=4-6]
	\arrow[dashed, from=5-6, to=5-8]
	\arrow[from=5-8, to=4-8]
	\arrow[dashed, from=6-4, to=3-4]
	\arrow[dashed, from=6-4, to=6-7]
	\arrow[dashed, from=6-7, to=6-9]
	\arrow[dashed, from=6-9, to=3-9]
	\arrow[bend right=18pt, dashed, from=7-1, to=4-1]
	\arrow[from=7-1, to=7-2]
	\arrow[bend right=18pt, dashed, from=7-2, to=4-2]
	\arrow[from=7-2, to=7-3]
	\arrow[bend left=18pt, dashed, from=7-2, to=7-5]
	\arrow[bend right=18pt, dashed, from=7-3, to=4-3]
	\arrow[from=7-3, to=6-4]
	\arrow[bend right=18pt, dashed, from=7-3, to=7-6]
	\arrow[bend right=18pt, dashed, from=7-5, to=4-5]
	\arrow[from=7-5, to=7-6]
	\arrow[bend right=18pt, dashed, from=7-6, to=4-6]
	\arrow[from=7-6, to=6-7]
	\arrow[dashed, from=7-6, to=7-8]
	\arrow[bend right=18pt, dashed, from=7-8, to=4-8]
	\arrow[from=7-8, to=6-9]
	\arrow[bend left=18pt, dashed, from=8-1, to=5-1]
	\arrow[from=8-1, to=7-1]
	\arrow[from=8-1, to=8-2]
	\arrow[bend left=18pt, dashed, from=8-2, to=5-2]
	\arrow[from=8-2, to=7-2]
	\arrow[from=8-2, to=8-3]
	\arrow[bend left=18pt, dashed, from=8-2, to=8-5]
	\arrow[bend left=18pt, dashed, from=8-3, to=5-3]
	\arrow[from=8-3, to=7-3]
	\arrow[bend right=18pt, dashed, from=8-3, to=8-6]
	\arrow[bend left=18pt, dashed, from=8-5, to=5-5]
	\arrow[from=8-5, to=7-5]
	\arrow[from=8-5, to=8-6]
	\arrow[bend left=18pt, dashed, from=8-6, to=5-6]
	\arrow[from=8-6, to=7-6]
	\arrow[dashed, from=8-6, to=8-8]
	\arrow[bend left=18pt, dashed, from=8-8, to=5-8]
	\arrow[from=8-8, to=7-8]
	\arrow[from=9-1, to=8-1]
	\arrow[from=9-1, to=9-2]
	\arrow[from=9-2, to=8-2]
	\arrow[from=9-2, to=9-3]
	\arrow[bend left=18pt, dashed, from=9-2, to=9-5]
	\arrow[from=9-3, to=8-3]
	\arrow[bend right=18pt, dashed, from=9-3, to=9-6]
	\arrow[from=9-5, to=8-5]
	\arrow[from=9-5, to=9-6]
	\arrow[from=9-6, to=8-6]
	\arrow[dashed, from=9-6, to=9-8]
	\arrow[from=9-8, to=8-8]
\end{tikzcd}\]
}\noindent
Set $M_1:= V_{[12, 12]},\, M_2:= V_{[21,31]},\, M:= M_1 \ds M_2$ and $N:= V_{[22,\{23,32\}]}$ as in that example, which are visualized as representations of $H(\bfP)$ in \eqref{eq:2D-M-N}.
For simplicity, we only draw solid arrows as edges to express representations of $H(\Int \oP)$,
and $\bullet$ stands for 0.
Then $K(M_1),\, K(M_2)$ and $K(N)$ are given as follows:
{\tiny
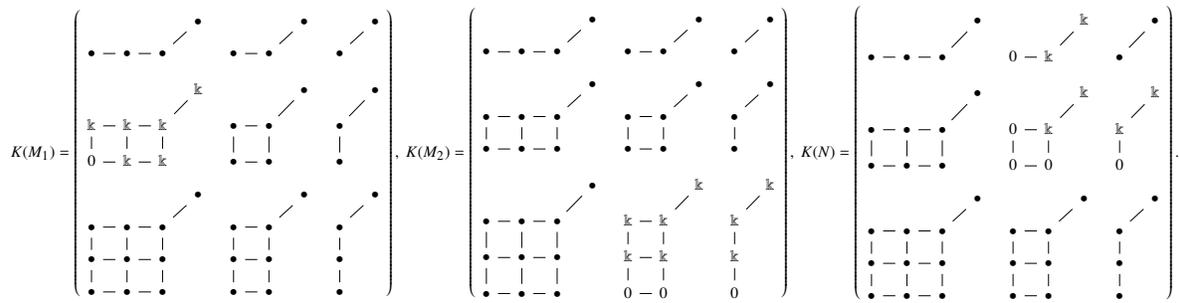
\[
K(M_1)= \left(
\begin{tikzcd}[column sep=5pt, row sep=5pt]
	&&& \bullet &&& \bullet && \bullet \\
	\bullet & \bullet & \bullet && \bullet & \bullet && \bullet \\
	&&& \k &&& \bullet && \bullet \\
	\k & \k & \k && \bullet & \bullet && \bullet \\
	0 & \k & \k && \bullet & \bullet && \bullet \\
	&&& \bullet &&& \bullet && \bullet \\
	\bullet & \bullet & \bullet && \bullet & \bullet && \bullet \\
	\bullet & \bullet & \bullet && \bullet & \bullet && \bullet \\
	\bullet & \bullet & \bullet && \bullet & \bullet && \bullet
	\arrow[no head, from=2-1, to=2-2]
	\arrow[no head, from=2-2, to=2-3]
	\arrow[no head, from=2-3, to=1-4]
	\arrow[no head, from=2-5, to=2-6]
	\arrow[no head, from=2-6, to=1-7]
	\arrow[no head, from=2-8, to=1-9]
	\arrow[no head, from=4-1, to=4-2]
	\arrow[no head, from=4-2, to=4-3]
	\arrow[no head, from=4-3, to=3-4]
	\arrow[no head, from=4-5, to=4-6]
	\arrow[no head, from=4-6, to=3-7]
	\arrow[no head, from=4-8, to=3-9]
	\arrow[no head, from=5-1, to=4-1]
	\arrow[no head, from=5-1, to=5-2]
	\arrow[no head, from=5-2, to=4-2]
	\arrow[no head, from=5-2, to=5-3]
	\arrow[no head, from=5-3, to=4-3]
	\arrow[no head, from=5-5, to=4-5]
	\arrow[no head, from=5-5, to=5-6]
	\arrow[no head, from=5-6, to=4-6]
	\arrow[no head, from=5-8, to=4-8]
	\arrow[no head, from=7-1, to=7-2]
	\arrow[no head, from=7-2, to=7-3]
	\arrow[no head, from=7-3, to=6-4]
	\arrow[no head, from=7-5, to=7-6]
	\arrow[no head, from=7-6, to=6-7]
	\arrow[no head, from=7-8, to=6-9]
	\arrow[no head, from=8-1, to=7-1]
	\arrow[no head, from=8-1, to=8-2]
	\arrow[no head, from=8-2, to=7-2]
	\arrow[no head, from=8-2, to=8-3]
	\arrow[no head, from=8-3, to=7-3]
	\arrow[no head, from=8-5, to=7-5]
	\arrow[no head, from=8-5, to=8-6]
	\arrow[no head, from=8-6, to=7-6]
	\arrow[no head, from=8-8, to=7-8]
	\arrow[no head, from=9-1, to=8-1]
	\arrow[no head, from=9-1, to=9-2]
	\arrow[no head, from=9-2, to=8-2]
	\arrow[no head, from=9-2, to=9-3]
	\arrow[no head, from=9-3, to=8-3]
	\arrow[no head, from=9-5, to=8-5]
	\arrow[no head, from=9-5, to=9-6]
	\arrow[no head, from=9-6, to=8-6]
	\arrow[no head, from=9-8, to=8-8]
\end{tikzcd}
\right),\ 
K(M_2)=
\left(
\begin{tikzcd}[column sep=5pt, row sep=5pt]
	&&& \bullet &&& \bullet && \bullet \\
	\bullet & \bullet & \bullet && \bullet & \bullet && \bullet \\
	&&& \bullet &&& \bullet && \bullet \\
	\bullet & \bullet & \bullet && \bullet & \bullet && \bullet \\
	\bullet & \bullet & \bullet && \bullet & \bullet && \bullet \\
	&&& \bullet &&& \k && \k \\
	\bullet & \bullet & \bullet && \k & \k && \k \\
	\bullet & \bullet & \bullet && \k & \k && \k \\
	\bullet & \bullet & \bullet && 0 & 0 && 0
	\arrow[no head, from=2-1, to=2-2]
	\arrow[no head, from=2-2, to=2-3]
	\arrow[no head, from=2-3, to=1-4]
	\arrow[no head, from=2-5, to=2-6]
	\arrow[no head, from=2-6, to=1-7]
	\arrow[no head, from=2-8, to=1-9]
	\arrow[no head, from=4-1, to=4-2]
	\arrow[no head, from=4-2, to=4-3]
	\arrow[no head, from=4-3, to=3-4]
	\arrow[no head, from=4-5, to=4-6]
	\arrow[no head, from=4-6, to=3-7]
	\arrow[no head, from=4-8, to=3-9]
	\arrow[no head, from=5-1, to=4-1]
	\arrow[no head, from=5-1, to=5-2]
	\arrow[no head, from=5-2, to=4-2]
	\arrow[no head, from=5-2, to=5-3]
	\arrow[no head, from=5-3, to=4-3]
	\arrow[no head, from=5-5, to=4-5]
	\arrow[no head, from=5-5, to=5-6]
	\arrow[no head, from=5-6, to=4-6]
	\arrow[no head, from=5-8, to=4-8]
	\arrow[no head, from=7-1, to=7-2]
	\arrow[no head, from=7-2, to=7-3]
	\arrow[no head, from=7-3, to=6-4]
	\arrow[no head, from=7-5, to=7-6]
	\arrow[no head, from=7-6, to=6-7]
	\arrow[no head, from=7-8, to=6-9]
	\arrow[no head, from=8-1, to=7-1]
	\arrow[no head, from=8-1, to=8-2]
	\arrow[no head, from=8-2, to=7-2]
	\arrow[no head, from=8-2, to=8-3]
	\arrow[no head, from=8-3, to=7-3]
	\arrow[no head, from=8-5, to=7-5]
	\arrow[no head, from=8-5, to=8-6]
	\arrow[no head, from=8-6, to=7-6]
	\arrow[no head, from=8-8, to=7-8]
	\arrow[no head, from=9-1, to=8-1]
	\arrow[no head, from=9-1, to=9-2]
	\arrow[no head, from=9-2, to=8-2]
	\arrow[no head, from=9-2, to=9-3]
	\arrow[no head, from=9-3, to=8-3]
	\arrow[no head, from=9-5, to=8-5]
	\arrow[no head, from=9-5, to=9-6]
	\arrow[no head, from=9-6, to=8-6]
	\arrow[no head, from=9-8, to=8-8]
\end{tikzcd}
\right),\ 
K(N)=
\left(
\begin{tikzcd}[column sep=5pt, row sep=5pt]
	&&& \bullet &&& \k && \bullet \\
	\bullet & \bullet & \bullet && 0 & \k && \bullet \\
	&&& \bullet &&& \k && \k \\
	\bullet & \bullet & \bullet && 0 & \k && \k \\
	\bullet & \bullet & \bullet && 0 & 0 && 0 \\
	&&& \bullet &&& \bullet && \bullet \\
	\bullet & \bullet & \bullet && \bullet & \bullet && \bullet \\
	\bullet & \bullet & \bullet && \bullet & \bullet && \bullet \\
	\bullet & \bullet & \bullet && \bullet & \bullet && \bullet
	\arrow[no head, from=2-1, to=2-2]
	\arrow[no head, from=2-2, to=2-3]
	\arrow[no head, from=2-3, to=1-4]
	\arrow[no head, from=2-5, to=2-6]
	\arrow[no head, from=2-6, to=1-7]
	\arrow[no head, from=2-8, to=1-9]
	\arrow[no head, from=4-1, to=4-2]
	\arrow[no head, from=4-2, to=4-3]
	\arrow[no head, from=4-3, to=3-4]
	\arrow[no head, from=4-5, to=4-6]
	\arrow[no head, from=4-6, to=3-7]
	\arrow[no head, from=4-8, to=3-9]
	\arrow[no head, from=5-1, to=4-1]
	\arrow[no head, from=5-1, to=5-2]
	\arrow[no head, from=5-2, to=4-2]
	\arrow[no head, from=5-2, to=5-3]
	\arrow[no head, from=5-3, to=4-3]
	\arrow[no head, from=5-5, to=4-5]
	\arrow[no head, from=5-5, to=5-6]
	\arrow[no head, from=5-6, to=4-6]
	\arrow[no head, from=5-8, to=4-8]
	\arrow[no head, from=7-1, to=7-2]
	\arrow[no head, from=7-2, to=7-3]
	\arrow[no head, from=7-3, to=6-4]
	\arrow[no head, from=7-5, to=7-6]
	\arrow[no head, from=7-6, to=6-7]
	\arrow[no head, from=7-8, to=6-9]
	\arrow[no head, from=8-1, to=7-1]
	\arrow[no head, from=8-1, to=8-2]
	\arrow[no head, from=8-2, to=7-2]
	\arrow[no head, from=8-2, to=8-3]
	\arrow[no head, from=8-3, to=7-3]
	\arrow[no head, from=8-5, to=7-5]
	\arrow[no head, from=8-5, to=8-6]
	\arrow[no head, from=8-6, to=7-6]
	\arrow[no head, from=8-8, to=7-8]
	\arrow[no head, from=9-1, to=8-1]
	\arrow[no head, from=9-1, to=9-2]
	\arrow[no head, from=9-2, to=8-2]
	\arrow[no head, from=9-2, to=9-3]
	\arrow[no head, from=9-3, to=8-3]
	\arrow[no head, from=9-5, to=8-5]
	\arrow[no head, from=9-5, to=9-6]
	\arrow[no head, from=9-6, to=8-6]
	\arrow[no head, from=9-8, to=8-8]
\end{tikzcd}
\right).
\]
}
Since all of them are (generalized) intervals,
it is easy to compute their minimal projective resolutions
(e.g., see \cite[Propositions 3.18, 3.32]{asashiba2024interval2}),
which are given as follows (we set $R_{(x,y)}:= \kIP{(x,y)}$ for short
and differentials are given by their coefficient matrices):
$$
\begin{aligned}
K^{M_1}\down &= (\cdots 0 \to R_{(33,33)} \ya{\bsmat{1\\1\\1}} R_{(12,23)}\ds R_{(13,13)}\ds R_{(32,32)} \ya{\bsmat{1&-1&0\\0&1&-1}} R_{(12,22)}\ds R_{(12,13)}),\\
K^{M_2}\down &= (\cdots 0 \to R_{(22,22)} \ya{1} R_{(21,22)}),\ \text{and}\\
K^N\down &= (\cdots 0 \to R_{(33,33)} \ya{1} R_{(22,33)}).\\
\end{aligned}
$$
Padding $K^N\down$ by
$C\down:= \Cone(\id_{R_{(12,22)}})\ds\Cone(\id_{R_{(12,13)}})\ds\Cone(\id_{R_{(32,32)}})[1]$ to have $F\down = K^N\down \ds C\down$,
we can give a candidate of a pre-matching $B \in \pMatch(K^M\down, F\down)$ by the following table:
$$
\begin{array}{c|c|cccc|ccc}
\deg & 2 & \multicolumn{4}{c|}1 & \multicolumn{3}{c}0\\
\hline
K^{M_1}\down & R_{(33,33)} & R_{(12,23)} & R_{(13,13)} & R_{(32,32)}& & R_{(12,22)} & R_{(12,13)}& \\
K^{M_2}\down & & & & & R_{(22,22)}& &  &R_{(21,22)} \\
\hline
K^N\down& & & & & R_{(33,33)}& & & R_{(22,33)}\\
C\down & R_{(32,32)} & R_{(12,22)} & R_{(12,13)} & R_{(32,32)}& & R_{(12,22)} & R_{(12,13)} & \\
\hline
\dist & 1 & 1 & 1 & 0 & 1 &0 &0 & 1
\end{array}\ ,
$$
where the coefficient matrices of differentials are given by the following table:
$$
\begin{array}{c|c|c}
 & \partial_1 & \partial_0 \\
 \hline
K^M\down & \bsmat{1\\1\\1\\0}  & \bsmat{1&-1&0&0\\0&1&-1&0\\0&0&0&1}\\
\hline
F\down  & \bsmat{0\\0\\1\\0} & \bsmat{1&0&0&0\\0&1&0&0\\0&0&0&1}
\end{array}.
$$
Thus, $B$ does not satisfy the compatibility on differentials for
$(K^M\down, F\down)$.
However, we have the following commutative diagram with vertical morphisms
isomorphisms:
$$
\begin{tikzcd}[ampersand replacement=\&, column sep=55pt]
\makebox[3em][r]{$F'\down:\quad R_{(32,32)}$} \& R_{(12,22)} \ds R_{(12,13)} \ds R_{(32,32)} \ds R_{(33,33)} \& R_{(12,22)} \ds R_{(12,13)} \ds R_{(22,33)}\\
\makebox[3em][r]{$F\down:\quad R_{(32,32)}$} \& R_{(12,22)} \ds R_{(12,13)} \ds R_{(32,32)} \ds R_{(33,33)} \& R_{(12,22)} \ds R_{(12,13)} \ds R_{(22,33)}
\Ar{1-1}{1-2}{"{\bsmat{1\\1\\1\\0}}"}
\Ar{1-2}{1-3}{"{\bsmat{1&-1&0&0\\0&1&-1&0\\0&0&0&1}}"}
\Ar{2-1}{2-2}{"{\bsmat{0\\0\\1\\0}}"}
\Ar{2-2}{2-3}{"{\bsmat{1&0&0&0\\0&1&0&0\\0&0&0&1}}"}
\Ar{1-1}{2-1}{equal}
\Ar{1-2}{2-2}{"{\bsmat{1&0&-1&0\\0&1&-1&0\\0&0&1&0\\0&0&0&1}}"}
\Ar{1-3}{2-3}{"{\bsmat{1&1&0\\0&1&0\\0&0&1}}"}
\end{tikzcd}.
$$
Define a complex $F'\down$ by the upper row.
Then $\pMatch(K^M\down, F'\down) = \pMatch(K^M\down, F\down)$, and
since $F'\down \cong F\down$, we have $(K^M\down, F'\down) \in \Padnn(K^M\down, K^N\down)$, and the $B$ above satisfies the compatibility on differentials for
the pair $(K^M\down, F'\down)$, thus $B \in \Match(K^M\down, F'\down)$,
and $\cost(B) = 1$.
Thus $\distCM(K^M\down, K^N\down) = 1$.
Since $\dGT^{\bfP}(M,N)=1$, stability holds with equality in this 2D example
as well.
\end{example}

\section{Future Directions}
\label{sec:directions}

We close by highlighting three directions that naturally extend the present
work. They are meant as themes rather than precise problems, in the hope that
they will spark further exploration rather than constrain it.

\subsection{Beyond intervals: kernels on convex subsets}
\label{ssec:beyond-ker}

In this paper, the kernel construction $K$ is defined on the interval poset
$\Int\oP$ and sends an interval $[x,y]$ to the kernel of a single
structure map $M(\pi_{y\ge x})$. It is natural to ask whether one can extend
this to a functor defined on a larger family of subsets, for example the convex
subposets or generalized intervals (= connected convex subsets) of $\oP$.

For intervals, the kernel functor $K$ produces a genuine functor
\[
K:\vect^\bfP\longrightarrow\vect^{\Int\oP},
\]
not just a numerical rank invariant, and its functoriality is crucial for our
stability results. In contrast, for more general families of convex subsets,
the generalized rank invariants of Kim--M\'emoli~\cite{kim2021generalized}
assign numbers to images of structure maps but do not assemble into a functor
on the poset of convex subposets ordered by reverse inclusion.

It would be very interesting to build an ``interval-like'' poset (or category)
$\gInt \oP$ whose objects are generalized intervals of $\oP$ and whose order
and metric reflect the combinatorics of $\bfP$, and to extend the kernel
construction to this setting. A satisfactory notion of ``generalized interval
persistence'' with a kernel functor and a stability statement in terms of the
G\"ulen--McCleary distance would give a conceptual home to many of the
generalized rank-type constructions that have appeared in the literature.

\subsection{Beyond \texorpdfstring{$L^\infty$}{Linfty}: other transport and matching costs}

Both metric constructions in this paper are inherently of $L^\infty$-type. On
the transport side, the cost of a Galois coupling is the supremum of the
displacements $d_\bfP(f(q),h(q))$ over points $q$ in the apex poset. On the
homological side, the complex matching distance on minimal resolutions takes a
supremum over matched indecomposable summands. This mirrors the classical
bottleneck distance on persistence diagrams.

In one-parameter persistence, there is also a substantial literature on
$L^p$- and Wasserstein-type distances between diagrams and their stability; see
for instance
\cite{cohen2010lipschitz,mileyko2011probability,divol2019understanding,bubenik2015metrics,bubenik2018wasserstein}.
In a closely related spirit, G\"ulen--M\'emoli--Patel introduce
$\ell^p$-type edit distances for weighted persistence diagrams and prove
$\ell^p$-stability with respect to Gromov--Hausdorff and
Gromov--Wasserstein distances, interpreting these classical metrics as edit
distances built from Galois-connection edits \cite{GulenMemoliPatel2025}. It
is natural to ask whether the constructions in this paper admit analogous
$L^p$-type formulations.

Roughly speaking, one would like to replace the supremum in the definition of
the G\"ulen--McCleary distance by a $p$-norm, perhaps by introducing weights or
measures on the apex poset, and to replace the bottleneck cost on resolutions
by a $p$-Wasserstein-type aggregate over matched indecomposable summands. The
main questions are whether such $L^p$-type transport and matching distances can
be defined in a way that still enjoys the triangle inequality and interacts
well with the pullback of resolutions along Galois couplings, and whether an
analogue of our main stability theorem survives for $p<\infty$.

\subsection{Beyond incidence algebras: other representation-theoretic settings}

Throughout this paper, a $\bfP$-module is equivalently a finite-dimensional
module over the incidence algebra $\k\bfP$; that is, $\vect^\bfP$ is equivalent
to the category of finite-dimensional left $\k\bfP$-modules. Our stability
results are phrased purely in terms of projective resolutions and the
G\"ulen--McCleary distance, and the appearance of M\"obius inversion comes only
from how these resolutions can be interpreted, especially in the persistence
setting via the kernel construction.

For a general finite-dimensional algebra $A$, the Cartan matrix records how
projective modules decompose into simples. When $A=\k\bfP$, this Cartan matrix
is, up to ordering, the zeta matrix of the poset, and its inverse, when it
exists, is the M\"obius matrix. Thus in the poset case, M\"obius inversion is
literally encoded in the homological relationship between projectives and
simples.

Understanding how much of this Cartan/M\"obius relationship persists for more
general algebras would be very interesting. For example, a projective
resolution of a simple functor
$\Hom_A(L,\blank)/\mathrm{rad}_A(L,\blank)$ corresponding to an indecomposable
$A$-module $L$ is given, through the Yoneda embedding
\[
\vect^A \to \Fun(\vect^A,\vect),\qquad M \mapsto \Hom_A(M,\blank),
\]
by the almost split sequence starting from $L$ (if $L$ is non-injective) or
the canonical epimorphism $L \to L/\operatorname{soc} L$ (if $L$ is
injective). This gives a formula for the multiplicity $d_M(L)$ of $L$ in the
indecomposable decomposition
\[
M = \Ds_{L \in \mathcal{L}}L^{d_M(L)}
\]
of an $A$-module $M$, as in \cite[Theorem~3]{Asashiba2017}, where
$\mathcal{L}$ is a complete set of representatives of indecomposable
$A$-modules. This suggests an alternative interpretation of the persistence
diagram of $M$.

\subsection{Homotopy interleavings, Galois transport, and the persistent pipeline}
\label{ssec:homotopy-interleavings}

It is insightful to compare our complex matching distance $\distCM$ with homotopy-invariant distances on persistent spaces, such as the homotopy interleaving distance $d_{HI}$ introduced by Frosini--Landi--M\'emoli~\cite{frosini2019persistent}, Blumberg--Lesnick~\cite{blumberg2023universality}, and Lanari--Scoccola~\cite{lanari2023rectification}.
While both notions incorporate metric and homotopical ideas, they operate in different categories and belong to distinct stages within the topological data analysis pipeline:
$$
\mathbf{Top}^\bfP \xrightarrow{\quad H_k \quad} \vect^\bfP \xrightarrow{\quad K \quad} \vect^{\Int\oP} \xrightarrow{\quad P\down^{(\blank)} \quad} \Kb(\prj \k[\Int\oP])
$$

\begin{enumerate}
    \item \textbf{The Interleaving Paradigm on Spaces:} The distance $d_{HI}$ operates at the beginning of the pipeline (e.g., on $\mathbf{Top}^{\bbR^n}$).
    It measures how persistent topological spaces can be shifted into one another up to weak homotopy equivalence, where distance zero corresponds to spatial weak equivalence ($X \simeq Y$).
    
    \item \textbf{The Matching Paradigm on Invariants:} The distance $\distCM$ operates at the end of the pipeline on minimal projective resolutions.
    In the context of persistence, this paradigm is realized by first applying the kernel functor $K$, so that the resolution $K\down^M$ acts as a persistence diagram.
    Here, distance zero corresponds to the algebraic analogue of homotopy equivalence: isomorphism in the homotopy category $\Kb(\prj \k[\Int\oP])$, which translates to strict isomorphism after padding by contractible cones.
\end{enumerate}

This perspective highlights natural directions for further study. While $d_{HI}$ is typically formulated for $\bbR^n$ using translation shifts, one could adapt the G\"ulen--McCleary distance $\dGT$ to persistent spaces $\mathbf{Top}^\bfP$ via space-level pullbacks along Galois insertions.
More broadly, investigating how these spatial transport costs pass through persistent homology $H_k$, the $1$-Lipschitz kernel functor $K$, and minimal resolutions $P\down$ could yield a unified stability framework that seamlessly connects spatial homotopy equivalences directly to the persistence diagram paradigm.

\bibliographystyle{alpha}
\bibliography{references}{}

\end{document}